\documentclass[a4paper, 10pt, twoside, notitlepage]{amsart}

\usepackage{amsmath,amscd}
\usepackage{amssymb}
\usepackage{amsthm}
\usepackage{comment}
\usepackage{graphicx, xcolor}
\usepackage{geometry}
\usepackage{mathrsfs}
\usepackage[ocgcolorlinks, linkcolor=blue]{hyperref}

\usepackage{bm}
\usepackage{bbm}
\usepackage{url}

\usepackage[utf8]{inputenc}
\usepackage{mathtools,amssymb,amsmath}
\usepackage{esint}
\usepackage{tikz}
\usepackage{dsfont}
\usepackage{relsize}
\usepackage{url}
\usepackage{xcolor}
\usepackage{graphicx}
\usepackage{mathrsfs}
\usepackage[shortlabels]{enumitem}
\usepackage{lineno}
\usepackage{amsmath}
\usepackage{enumitem}
\usepackage{amsthm} 
\usepackage{verbatim}
\usepackage{dsfont}
\numberwithin{equation}{section}

\usepackage{slashed}
\usepackage{appendix}

\allowdisplaybreaks

\mathtoolsset{showonlyrefs}

\graphicspath{{images}}

\newtheorem{theorem}{Theorem}[section]
\newtheorem{lemma}[theorem]{Lemma}
\newtheorem{definition}[theorem]{Definition}

\newtheorem{proposition}[theorem]{Proposition}

\newtheorem{remark}[theorem]{Remark}

\title[Stable determination of coefficients for the JMGT equation]{Stable determination of coefficients for the nonlinear third-order acoustic equations}

\author[S. Fu]{Song-Ren Fu}
\address{School of Mathematics and Statistics, Northwestern Polytechnical University, Xi'an 710129, People's Republic of China.}
\email{songrenfu@nwpu.edu.cn}

\author[D. Qiu]{Dong Qiu}
\address{School of Mathematics, Zhejiang University, Hangzhou 310058, People's Republic of China.}
\email{qiudong@zju.edu.cn}

\author[T. Zheng]{Tianyi Zheng}
\address{School of Mathematical Sciences and LPMC, Nankai University, Tianjin 300071, People's Republic of China.}
\email{9820250139@nankai.edu.cn}

\author[T. Zhou]{Ting Zhou}
\address{Corresponding author: School of Mathematics, Zhejiang University, Hangzhou 310058, People's Republic of China.}
\email{ting\_zhou@zju.edu.cn}

\newcommand{\R}{{\mathbb R}}

\newcommand{\N}{{\mathbb N}}

\newcommand{\Om}{\Omega}

\DeclareMathOperator{\supp}{supp} 

\begin{document}

\begin{abstract}
In this paper, we study an inverse boundary value problem for a generalized Jordan-Moore-Gibson-Thompson equation with Westervelt-type  nonlinearity and space-dependent coefficients. This third-order (in time) hyperbolic equation models nonlinear ultrasound propagation in viscous and thermally relaxing media. 
Assuming that the diffusivity and the sound speed are known a priori, we investigate the simultaneous stable determination of the friction coefficient, the weak damping coefficient, the potential, and the nonlinear coefficient from the associated Dirichlet-to-Neumann map. The proof combines the finite-difference linearization method with the construction of Gaussian beam and geometric optics solutions. These arguments reduce the inverse problem to stability estimates for (attenuated) geodesic ray transforms under the foliation condition.  As a consequence, we obtain H\"older-type stability estimates of recovering the linear and nonlinear coefficients.

\medskip

\medskip

		
\end{abstract}
\maketitle

\setlength{\parskip}{1.0ex} 

\section{Introduction and preliminaries}

This paper addresses the stable determination of multiple unknown coefficients appearing in a generalized Jordan-Moore-Gibson-Thompson (JMGT) equation using boundary measurements. The JMGT equation is a nonlinear acoustic model for high-intensity sound propagation in thermoviscous media. It is commonly derived as a relaxation modification of the classical Westervelt and Kuznetsov equations, motivated by replacing Fourier's heat-flux law with a Maxwell-Cattaneo law involving a finite thermal relaxation time. The linearization of the JMGT equation is often called the Moore-Gibson-Thompson (MGT) equation.  The resulting third-order-in-time hyperbolic model incorporates quadratic acoustic nonlinearity, thermoviscous damping, and relaxation effects, and is therefore relevant for high-intensity ultrasound applications. For more details on the derivation of the (J)MGT equation, we refer to \cite{BKILR} and \cite[Section 1]{fu2026calderon}.

The MGT equation differs from the classical second order wave equation in various topics, such as the analysis of well-posedness, dynamic behavior, controllability, and stability etc. More precisely, this equation displays a variety of dynamical 
behaviors for its solution that depend on the physical parameters appearing in the equation. For instance, the diffusivity affects the stability of solutions to the MGT equation, and the well-posedness of solutions fails, even in the simplest case when $b=0$ (see \cite[Theorem 2.1]{BKILR}). While for the second order equations, the presence of the structural damping is irrelevant for the well-posedness. The controllability properties of the MGT equations, compared to the classical wave equations, are much more complicated, see \cite{LIZAMA20197813}. Even for the one dimensional MGT equation, it is not exact and null controllable by a control supported on the boundary, see \cite{lizama2023boundary}. Moreover, it remains unclear whether the unique continuation property is satisfied by the MGT equations under certain geometric conditions. These features add to the complexity of analyzing the (J)MGT equations. 

Let $\Omega\subset\mathbb R^n$ ($n\ge 3$) be a bounded domain with a smooth boundary $\Gamma\vcentcolon=\partial\Omega$. In this paper, for any set $A\subset\mathbb R^n$, we denote $A_T\vcentcolon=(0,T)\times A$ for $T>0$.
In this work, we study a generalized JMGT equation (with some lower order terms) of Westervelt type in $\Omega_T$, namely
\begin{equation}\label{eq:linear-JMGT-Sec1-intro}
\begin{cases}
(\tau\partial_t^3+\alpha\partial_t^2-b\Delta\partial_t-c^2\Delta+\lambda\partial_t+q)u=\partial_t^2(\xi u^2) & {\rm in}\,\, \Omega_T,\\
u=f  & {\rm on}\,\, \Gamma_T,\\
u(0)=\partial_tu(0)=\partial_t^2u(0)=0  & {\rm in}\,\, \Omega.
\end{cases}
\end{equation}
Here, $\tau$, $\alpha$, $b$, and $c$ denote the relaxation parameter, the friction parameter related to
viscosity, the diffusivity, and the speed of sound, respectively. The lower-order coefficients $\lambda$ and $q$  represent the weak damping and the potential, respectively. The coefficient $\xi$ is the nonlinear parameter arising from the pressure-density relation in the medium, see \cite{BKVN}.

For generality and considering the practical scenarios of the nonlinear model, we assume in this paper that the coefficients $\alpha,b,c,\lambda,q$ and $\xi$ are all space-varying functions. 
The main purpose of this paper is to investigate the stable recovery of  the coefficients $\alpha,\lambda,q$ and $\xi$ simultaneously from  boundary measurements. 

Henceforth, for simplicity, we take the positive constant $\tau=1$ by re-scaling, and assume that the leading coefficient $b$ is a $C^\infty$-function having a positive lower bound $m_0>0$. That is, $\min_{x\in\overline\Omega}b(x)\ge m_0$.  We will sometimes use the notion $\Upsilon\vcentcolon=\alpha-\frac{c^2}{b}$ to denote the critical parameter for the MGT equation. We note that such parameter is widely used in analyzing the long time behavior of solutions, see for example \cite{LASIECKA20157610,BKILR} for more reading.

\subsection{Formulation of the inverse problem}
We first introduce some function spaces that will be frequently used later. For an integer $m\ge0$, let
\begin{equation}\label{def:E-space}
E^m(\Omega_T)
\vcentcolon=
\bigcap_{k=0}^m C^k([0,T];H^{m-k}(\Omega)),
\end{equation}
equipped with the norm 
\[
\|u\|_{E^m(\Omega_T)}
\vcentcolon=
\sum_{k=0}^m\sup_{t\in[0,T]}
\|\partial_t^ku(t)\|_{H^{m-k}(\Omega)}.
\]
Henceforth, $H^r(\Omega_T)$ and $H^r(\Gamma_T)$ denote the standard
Sobolev spaces on the space-time cylinders. For boundary data, define
\begin{equation}\label{def:A-space}
\mathcal A^m(\Gamma_T)
\vcentcolon=
\Big(
\bigcap_{k=0}^{m+1}
C^k([0,T];H^{m+3/2-k}(\Gamma))\Big)\bigcap H^{m+2}(\Gamma_T)
\end{equation}
with norm
\[
\|f\|_{\mathcal A^m(\Gamma_T)}
\vcentcolon=
\sum_{k=0}^{m+1}\sup_{t\in[0,T]}
\|\partial_t^kf(t)\|_{H^{m+3/2-k}(\Gamma)}
+\|f\|_{H^{m+2}(\Gamma_T)}.
\]
We also set
\begin{equation}\label{def:A0-space}
\mathcal A_0^m(\Gamma_T)
\vcentcolon=
\left\{
f\in\mathcal A^m(\Gamma_T):
\partial_t^kf(0,\cdot)=0,\ 0\le k\le m+1
\right\}
\end{equation}
and, for backward problems,
\begin{equation}\label{def:AT-space}
\mathcal A_T^m(\Gamma_T)
\vcentcolon=
\left\{
h\in\mathcal A^m(\Gamma_T):
\partial_t^kh(T,\cdot)=0,\ 0\le k\le m+1
\right\}.
\end{equation}
and
\begin{equation}\label{def:N-space}
\mathcal N_{m+1}(\Gamma_T)
\vcentcolon=
\bigcap_{k=1}^{m+1}
H^k(0,T;H^{m+1-k}(\Gamma)).
\end{equation}
We equip this space with the norm
\begin{equation}\label{def:N-norm}
\|g\|_{\mathcal N_{m+1}}
\vcentcolon=\Bigl(\sum_{k=1}^{m+1}\|g\|_{H^k(0,T;H^{m+1-k}(\Gamma))}^2\Bigr)^{1/2}.
\end{equation}

For a multi-index $\theta\in\mathbb N_0^n$, we use
$|\theta|=\theta_1+\cdots+\theta_n$ and
$\partial^\theta=\partial_{x_1}^{\theta_1}\cdots
\partial_{x_n}^{\theta_n}$. For a fixed integer $K\ge0$, define
\begin{equation}\label{set:U-K-M0}
\mathcal U_K(M_0)
\vcentcolon=
\left\{
h\in C^\infty(\overline\Omega;\mathbb R):
\|h\|_{C^K(\overline\Omega)}\le M_0
\right\},
\end{equation}
where
\[
\|h\|_{C^K(\overline\Omega)}
\vcentcolon=
\sum_{|\theta|\le K}
\|\partial^\theta h\|_{L^\infty(\Omega)}.
\]
For simplicity, we assume that, for a fixed integer $K$ specified in each result,
\begin{equation}
b,c,\alpha,\lambda,q,\xi\in \mathcal U_K(M_0),
\end{equation}
with the same \emph{a priori} bound $M_0$.

Fix an integer $s>n+2$. Let $\delta_{\rm wp}>0$ denote the uniform
small-data threshold in Proposition~\ref{prop:Westervelt} with $m=s$,
and choose $0<\epsilon_0\le\delta_{\rm wp}$. The admissible
Dirichlet data form the closed ball
\begin{equation}\label{set-for-f}
\mathcal U_{\Gamma_T}^{\epsilon_0,s}
\vcentcolon=
\left\{
f\in\mathcal A_0^s(\Gamma_T):
\|f\|_{\mathcal A^s(\Gamma_T)}\le\epsilon_0
\right\}.
\end{equation}
In particular, every initial trace appearing in \eqref{set-for-f} is
well-defined by \eqref{def:A-space}. We formally define the DN map associated with
\eqref{eq:linear-JMGT-Sec1-intro} by
\begin{equation}\label{def:DN-map}
\Lambda_{\alpha,\lambda,q,\xi}\vcentcolon
\mathcal U_{\Gamma_T}^{\epsilon_0,s}\to L^2(\Gamma_T),
\qquad
f\mapsto (b\partial_\nu \partial_tu+c^2\partial_\nu u)|_{\Gamma_T},
\end{equation}
where $u$ is the unique small solution established in Proposition~\ref{prop:Westervelt},
$f\in\mathcal U_{\Gamma_T}^{\epsilon_0,s}$, and $\nu$ is the unit outward normal field along $\Gamma$ of $\Omega$. 

The related inverse problem we mainly consider in this paper is stated as follows.

{\bf Inverse problem:} \emph{Can we recover the coefficients $\alpha,\lambda,q,\xi$ simultaneously by the associated DN map in a stable way?}


Throughout this paper, we sometimes use the notion $C_1\lesssim C_2$ to denote the relation $C_1\le CC_2$ for some $C>0$ independent of $C_1,C_2$.

\subsection{Main assumptions and results}\label{subsection-1.2}
We first present the geometric setting induced by $b(x)$. Let $I_n$ be the $n$-th identity matrix with the standard inner product 
$$\langle x,y\rangle=x_1y_1+\cdots+x_ny_n,\quad \forall x=(x_1,\cdots,x_n),\, y=(y_1,\cdots, y_n)\in\R^n.$$  
We introduce the metric
$g=b^{-1}ds^2$ with the inner product $g(\cdot,\cdot)=\langle\cdot,\cdot\rangle_g$, where $ds^2$ denotes the standard Euclidean metric. Thus, $(\Omega,g)$ is a Riemannian manifold.
Let $D$ be the Levi-Civita connection with respect to the metric $g$. We see that, for functions $\varphi,\psi\in H^1(\Omega)$, $D\varphi=b(x)\nabla\varphi$ and
\begin{equation}
g(D\varphi,D\psi)=\langle D\varphi,D\psi\rangle_g=\langle b(x)\nabla\varphi,\nabla\psi\rangle,
\end{equation}
In order to study the recovery of the coefficients $\alpha,\lambda,q$ and $\xi$, we need to impose some geometric assumption for $(\Omega,g)$ (i.e., for $b$ in $\Omega$). The main assumption of this paper reads as follows.

{\bf Assumption (A):}\, $(\overline\Omega,g)$ is a compact Riemannian manifold with strictly convex boundary, and satisfies the foliation condition.

\noindent Since the foliation condition implies that $(\Omega,g)$ is non-trapping, we have that the diameter of $\Omega$ in the metric $g$ is finite. That is,
\begin{equation}
{\rm diam}_g(\Omega)\vcentcolon=\sup\{{\rm lengths\,\, of\,\, all\,\, geodesics\,\, in\,\, \Omega} \}<\infty.
\end{equation}
In the next section, we will give some more discussions on the foliation condition, and provide some sufficient conditions for the foliation condition to hold.

We are now in a position to state the main theorems of this paper.
\begin{theorem}[Stable determination of $\alpha,\lambda,q$]\label{thm:stable-alp-lam-q}
Let $n\ge 3$, $T>{\rm diam}_g\Omega$, and let the assumption {\bf (A)} hold. Suppose that $\alpha_j,\lambda_j,q_j,\xi_j\in\mathcal U_K(M_0)$ for $j=1,2$ and sufficiently large $K\in\mathbb N$, and $\alpha_1=\alpha_2$ on $\Gamma$. If there exist $\delta>0$ and a sufficiently small constant $\epsilon_0>0$ such that the corresponding DN maps $\Lambda_{\alpha_1,\lambda_1,q_1,\xi_1}$ and $\Lambda_{\alpha_2,\lambda_2,q_2,\xi_2}$ satisfy
\begin{equation}\label{condition-for-DN-map}
\|\Lambda_{\alpha_1,\lambda_1,q_1,\xi_1}(f)-\Lambda_{\alpha_2,\lambda_2,q_2,\xi_2}(f)\|_{L^2(\Gamma_T)}\le\delta,\quad \forall f\in\mathcal U_{\Gamma_T}^{\epsilon_0,m},\quad m> n+1,
\end{equation}
then, there is a constant $\mu_0\in (0,1)$ such that
\begin{equation}
\|\alpha_2-\alpha_1\|_{L^\infty(\Omega)}+\|\lambda_2-\lambda_1\|_{L^\infty(\Omega)}+\|q_2-q_1\|_{L^\infty(\Omega)}\lesssim \delta^{\mu_0}.
\end{equation}
\end{theorem}

\begin{theorem}[Stable determination of $\alpha,\lambda,q,\xi$]\label{thm:stable-alp-lam-q-xi}
Let $n\ge 3$ and $b$ be some positive constant such that $\sqrt b T>{\rm diam}\,\Omega=\sup_{x,y\in\Omega}|x-y|$. Assume that $\Omega$ has a smooth and strictly convex boundary. Suppose that $\alpha_j,\lambda_j,q_j,\xi_j\in\mathcal U_K(M_0)$ for $j=1,2$ and sufficiently large $K\in\mathbb N$, and $\alpha_1=\alpha_2$, $\xi_1=\xi_2$ on $\Gamma$. If there exist $\delta>0$ and a sufficiently small constant $\epsilon_0>0$ such that the corresponding DN maps $\Lambda_{\alpha_1,\lambda_1,q_1,\xi_1}$ and $\Lambda_{\alpha_2,\lambda_2,q_2,\xi_2}$ satisfy \eqref{condition-for-DN-map},
then there is a constant $\tilde \mu_0\in (0,1)$ such that
\begin{equation}\label{est:a-lam-q-xi}
\|\alpha_2-\alpha_1\|_{L^\infty(\Omega)}+\|\lambda_2-\lambda_1\|_{L^\infty(\Omega)}+\|q_2-q_1\|_{L^\infty(\Omega)}+\|\xi_2-\xi_1\|_{L^\infty(\Omega)}\lesssim \delta^{\tilde\mu_0}.
\end{equation}
\end{theorem}
Below, we give some remarks.
\begin{remark}
In our proof, the stable recovery of the coefficients relies on the geodesic ray transform in the metric $g$ (see Section \ref{sec:geodesic-ray} for the introduction of such transform). For the stable recovery of \(\xi\), since a suitable stability result for the Jacobi-weighted geodesic ray transform of the first kind is not available, we restrict to the case where \(b\) is a positive constant and apply the stability theory for the Euclidean attenuated geodesic ray transform.
 We note that, to simplify the analysis, the explicit H\"older stability exponents are not given in the above Theorems and are not sharp. 
\end{remark}
\begin{remark}
The stable determination of the time-dependent $\xi$ was studied in the work \cite{Fu03052026}, where the coefficients $\alpha,b,c$ are all constants and $\lambda=q=0$ in $\Omega$. However, for a time-dependent nonlinear coefficient $\xi$, it is not known whether the estimate \eqref{est:a-lam-q-xi} remains valid when $b$ is a general positive space-varying function.
\end{remark}


\subsection{Related literature}
Inverse problems for hyperbolic equations have been actively studied. After the seminal works \cite{belishev1987approach} that developed the well-known boundary control (BC) method, there is
an extensive literature related to the recovery of time independent coefficients for hyperbolic equations.  An introduction to the method can be found in \cite{belishev2011boundary}. However, we note that the (J)MGT equation has poor controllability, which is a  significant difference from the classical second-order (in time) wave equation, see \cite{LIZAMA20197813}. Hence it seems that the BC method can not be applied to the (J)MGT equation.

Apart from this classical method, for the linear wave equation with space-varying wave speed, Bao and Zhang \cite{bao2014sensitivity} investigated the problem of how sensitive or stable is it to recover the velocity field from the dynamic DN map and characterized how a small change in the dynamic DN map affects the recovered velocity field. We also mention Kurylev, Lassas and Uhlmann's work \cite{kurylev2018inverse} that introduced a powerful approach to solve inverse coefficient problems for nonlinear hyperbolic equations on manifolds. This method relies on examining families of solutions that depend on several parameters, and on performing linearizations simultaneously with respect to each small parameter. Stable determination results for nonlinear (and linear) coefficients appearing in semilinear wave equations, in both Euclidean and manifold settings, can be found, e.g., in \cite{chen2025stable,lassas2022uniqueness,lassas2025stability} and the references therein.

The method of constructing geometric optics (GO) solutions for wave equations is also extensively used in determining coefficients, although this method requires much more restrictive geometric conditions than the Boundary Control method does. In \cite{kian2017unique,kian2016recovery}, by using GO solutions, the authors studied inverse problems for wave equations concerning the unique determination of the linear potential and damping coefficients from partial observations on the boundary. Moreover \cite{Aicha_2015,kian2016stability,kumar2026holder} considered the stability in inverse problems of determining linear coefficients by GO solutions together with some integral transforms.


In \cite{acosta2022nonlinear} the authors considered the recovery of the nonlinear coefficient for a nonlinear wave equation of Westervelt type. The second order linearization together with Gaussian beams were used. The recovery comes from inverting the weighted geodesic ray transform, if one assumes the domain has no conjugate points or satisfies the foliation condition. After this work, \cite{uhlmann2023inverse} considered a more general nonlinear term and allowed cut points or conjugate points without foliation conditions. Moreover, if the wave speed is unknown, they proved the unique recovery of it from the first order linearization of the DN map. We also mention the work \cite{wendels2025stable} that addressed the simultaneous stable determination of the nonlinear coefficient and the wave speed appearing in the nonlinear Westervelt equations.

In some of the mentioned works, the construction of Gaussian beam solutions is an effective method to study the recovery of coefficients. Gaussian beams were originally introduced in \cite{babich1981complex} and were later used for the first time in inverse problems in \cite{belishev1992boundary}. An advantage of Gaussian beams is that they allow the existence of conjugate points on the Riemannian manifold. We also refer readers to \cite{feizmohammadi2021recovery} for a comprehensive representation in the recovery of a time-dependent magnetic vector-valued potential and an electric scalar-valued potential on a Riemannian manifold from the knowledge of the DN map. The authors used Gaussian beams to reduce the inverse problem to the inversion of the light ray transform of the unknown
coefficients. By constructing Gaussian beams with reflections on the boundary, the work \cite{liu2025partial} proved that the nonlinearity of order three or higher in a semi-linear wave equation on a Lorentzian manifold with boundary can be determined up to natural obstructions by the partial DN map, where the measurement set is arbitrarily small. 

Let us also mention that Gaussian beams together with suitable (higher-order) linearization method have been extended by e.g., the works \cite{feizmohammadi2022recovery,uhlmann2021inverse} and the references therein for inverse nonlinear (elastic) wave equations. However, there are few results available in the literature concerning inverse problems of (J)MGT equations. 
We note that the construction of Gaussian beam solutions for the MGT equation is more complicated compared to the classical second order (in time) wave equation. Moreover, the unknown critical parameter $\Upsilon(x)$ will appear in the constructed Gaussian beam solutions, which creates difficulties in analyzing the geodesic ray transforms of the coefficients.

The works \cite{ArancibiaLecarosMercadoZamorano+2022+659+675,LiuTriggiani+2013+825+869} studied the inverse problem of determining the time-independent coefficient $\alpha(x)$ appearing in the MGT equation. A Lipschitz stability of recovering  $\alpha(x)$ was obtained by Carleman estimates. In the past three years, these results were extended to inverse local and nonlocal (J)MGT equations addressing the stable and unique determination of time-dependent coefficients, see \cite{fu2024inverse,fu2026calderon,Fu03052026}. 
Recently, the works \cite{qiu2026gauge} and \cite{qiu2026inverse} respectively studied inverse boundary value problems for JMGT equations with sources on a simple Riemannian manifold, and inverse JMGT equations  with Westervelt and Kuznetsov nonlinearities on a manifold that satisfies certain geometric conditions. We note that the two papers mainly focus on the unique determination of coefficients for JMGT equations. We also mention the works \cite{kaltenbacher2025acoustic,kaltenbacher2025imagingnonlinearitycoefficientsound} that considered the uniqueness and stability of reconstruction of two coefficients (sound speed and nonlinearity parameter) in the JMGT equation.

In this paper, we are concerned with the inverse problem of simultaneously recovering the coefficients from the knowledge of the associated DN map in a stable way. The inverse problem for the JMGT equation is considerably more involved than its second-order counterparts due to the third-order derivative structure, the simultaneous presence of several unknown coefficients, and the nonlinear dependence of the Gaussian beam amplitudes on the critical parameter \(\Upsilon=\alpha-c^2/b\). In particular, the contributions of \(\alpha\), \(\lambda\), and \(q\) appear at different asymptotic orders, which leads to a successive recovery procedure, while the appearance of \(\Upsilon\) in the transport equations produces weighted geodesic X-ray transforms whose weights depend on the unknown coefficients. Moreover, stability estimates require all Gaussian beams, remainder terms, boundary traces, and concentration estimates to hold uniformly over the full family of geodesics, including grazing ones. The recovery of the nonlinear coefficient \(\xi\) further requires higher-order linearization together with suitable geometric optics solutions. These features constitute the main difficulties in obtaining simultaneous stability for the multiple coefficients. To the authors' best knowledge, this is the first work addressing the stable determination of multiple space-varying coefficients for the JMGT equation.

The rest of this paper is organized as follows. In Section \ref{sec:geodesic-ray}, we introduce the geodesic ray and light ray transforms and the foliation condition. Some examples that verify such condition will be also given in this section. Section \ref{sec:well-posedness} focuses on the well-posedness of the forward linear and nonlinear problems. In Section \ref{sec:Gaussian-beams}, we construct Gaussian beam solutions for the linearized equation. In Section \ref{sec:lin}, we present the linearization method and establish some Alessandrini type identities for the unknown coefficients. In Section \ref{sec-proof}, we give the proof of the main theorems in this present paper. Some conclusions are given in the last Section \ref{sec:conclusion}.

\section{Geodesic ray transform and the foliation condition}\label{sec:geodesic-ray}

In this section, we will introduce some geometric notions and some ray transforms.
Recall that $(\Omega,g)$ is a compact Riemannian manifold with smooth boundary $\Gamma$. Let $S\Omega$ denote the unit tangent bundle of $\Omega$. For any $(x_0,\zeta)\in S\Omega$, let $\gamma_{x_0,\zeta}(\cdot)$ denote the unit speed geodesic starting at point $x_0$, in the direction $\zeta\in\Omega_{x_0}$, where $\Omega_{x_0}$ denotes the tangent space at $x_0$. We now define the set
\begin{equation}
\partial_-S\Omega\vcentcolon=\{(x_0,\zeta)\in S\Omega:\, x_0\in\Gamma, \langle\zeta,\nu_g(x_0)\rangle_g<0\},
\end{equation}
where $\nu_g$ is the outward normal unit vector field of $(\Omega,g)$ along the boundary $\Gamma$.
For any $(x_0,\zeta)\in S\Omega$, we define $\tau_{\rm exit}(x_0,\zeta)$ by
\begin{equation}
\tau_{\rm exit}(x_0,\zeta)\vcentcolon=\inf\{t>0:\, \gamma_{x_0,\zeta}(t)\in \Gamma, \gamma'_{x_0,\zeta}(t)\notin \Gamma_{\gamma_{x_0,\zeta}(t)}\}.
\end{equation}
It is clear that $\tau_{\rm exit}(x_0,\zeta)\le{\rm diam}_g(\Omega)<\infty$ if $(\Omega,g)$ is non-trapping. 
 Henceforth, for the sake of brevity, we use the term maximal geodesic to refer to the geodesics $\gamma_{x_0,\zeta}(\cdot)$  (or $\gamma(\cdot)$ in short) with $(x_0,\zeta)\in\partial_-S\Omega$, over their maximal interval $(0,\tau_{\rm exit})$. 
We next recall the definition of geodesic ray transform along a geodesic $\gamma$ in $\Omega$.
\begin{definition}
Let $(x_0,\zeta)\in\partial_-S\Omega$ and $\gamma_{x_0,\zeta}:(0,\tau_{\rm exit})\to \Omega$ be a unit speed geodesic. The geodesic ray transform of a scalar function $h\in C(\Omega)$ is defined by
\begin{equation}
\mathcal I_{x_0,\zeta}(h)=\mathcal I_\gamma(h)(x_0,\zeta)\vcentcolon=\int_0^{\tau_{\rm exit}}h(\gamma(t))dt.
\end{equation}
\end{definition}

For the geodesic $\gamma$ given as above, we can find the corresponding null geodesic $\vartheta=(t,\gamma(t))$ on the Lorentz manifold $(\mathbb R\times\Omega,\bar g)$, where $\bar g=-dt^2+g$. We parametrize maximal null geodesics through $\vartheta(t)=(s+t,\gamma(t))$ for $s\in\mathbb R$ and $t\in (0,\tau_{\rm exit})$. The light ray transform $\mathcal J_\vartheta$ of a scalar function $\tilde h\in C(\Omega_T)$ along $\vartheta$ is defined by
\begin{equation}
\mathcal J_\vartheta(\tilde h)(s,x_0,\zeta)\vcentcolon=\int_0^{\tau_{\rm exit}}\tilde h(\vartheta(t))dt,\quad (s,x_0,\zeta)\in\mathbb R\times\partial_-S\Omega.
\end{equation}
We note that the above transforms can also act on tensors and $m$-forms for $m\ge 1$.

We recall the foliation condition ensures that a manifold can be foliated by strictly convex hypersurfaces with respect to the metric. We note that the foliation condition allows the existence of conjugate points, and implies the non-trapping property of $(\Omega,g)$, see e.g., \cite{uhlmann2016inverse}. By \cite{paternain2019geodesic}, we know that the foliation condition is satisfied if there is a smooth strictly convex function on $(\Omega,g)$. Thus, if there is a strictly convex function, then the injectivity of the geodesic transform $\mathcal I_\gamma$ on $\partial_-S\Omega$ holds. Moreover, the inversion of $\mathcal I_\gamma$ under the foliation condition is stable (see \cite[Section 1 and Theorem 4.1]{uhlmann2016inverse}). Such results were extended to weighted geodesic ray transforms and light ray transforms related to the stationary Lorentz manifold $(\mathbb R\times\Omega,\bar g)$,  see \cite{paternain2019geodesic,feizmohammadi2021light,lassas2020light,oksanen2025interplay,SU04}  for more details.

We note that these ray transforms have essential applications in the recovery of coefficients (including the metrics) appearing in geometric elliptic and wave equations, see an incomplete list \cite{paternain2023geometric,feizmohammadi2021recovery,stefanov2018inverse,vasy2021light,uhlmann2021inverse} and the references therein.



We proceed to introduce a sufficient condition that guarantees the existence of strictly convex functions on $(\Omega,g)$.
Generally, let $\tilde g=b^{-1}g_0$ be the conformal metric of a given metric $g_0$, and let $D_0$ and $\widetilde D$ respectively be the Levi-Civita connections in the metric $g_0$ and $\tilde g$.  Then by the fundamental theorem of the Levi-Civita connection \cite[Theorem 4.3.1]{jost2005riemannian},  for any real-valued vector fields $X,Y$ on $\Omega$, we have
\begin{equation}\label{relation-connection}
\widetilde D_XY=D_{0X}Y-\frac{1}{2b}(Xb)Y-\frac{1}{2b}(Yb)X+\frac{\langle X,Y\rangle_{g_0}}{2b}D_0b.
\end{equation}
We show the following lemma:
\begin{lemma}
Let $\psi\in C^2(\overline\Omega)$ be a strictly  convex and positive function in the Riemannian manifold $(\Omega,g_0)$. That is, 
there is a constant $\varrho>0$ such that
\begin{equation}
D_0^2\psi(X,X)\ge \varrho|X|^2_{g_0},\quad \forall X\in\Omega_x,\, x\in\Omega.
\end{equation}
If the conformal factor $b$ satisfies the condition
\begin{equation}\label{condition-rho}
\langle D_0\psi,D_0\ln b\rangle_{g_0}\le 2\varrho-\varrho_0,
\end{equation}
for some constant $\varrho_0>0$, then there exists a constant $\rho_0>0$ depending on $\varrho_0,b$ and $g$, such that for any $\rho\ge \rho_0$, $\widetilde\psi:=e^{\rho\psi}$ is a strictly convex function on $(\Omega,\tilde g)$.
\end{lemma}
\begin{proof}
By \eqref{relation-connection}, for any real-valued vector field $X$, we can compute 
\begin{equation}\label{conformal-LC-connection}
\begin{split}
&\widetilde D^2\widetilde \psi(X,X)=\widetilde D^2e^{\rho\psi}(X,X)=\tilde g(\widetilde D_X\widetilde De^{\rho\psi},X)\\
&=\frac{\rho e^{\rho\psi}}{b}[g_0(\widetilde D_X\widetilde D\psi,X)+\rho b(X\psi)^2]\\
&=\frac{\rho e^{\rho\psi}}{b}[(Xb)(X\psi)+bg_0(\widetilde D_XD_0\psi,X)]+\rho^2e^{\rho\psi}(X\psi)^2\\
&=\frac{\rho e^{\rho\psi}}{b}[(Xb)(X\psi)+bg_0(D_{0X}D_0\psi-(2b)^{-1}(Xb)D_0\psi\\
&\quad -(2b)^{-1}\langle D_0\psi,D_0b\rangle_{g_0}X +(2b)^{-1}(X\psi)D_0b,X)]+\rho^2e^{\rho\psi}(X\psi)^2\\
&=\rho e^{\rho\psi}D_0^2\psi(X,X)+\frac{\rho e^{\rho\psi}}{b}[(X\psi)(Xb)-1/2D_0\psi(b)|X|_{g_0}^2]+\rho^2e^{\rho\psi}(X\psi)^2\\
&\ge \frac{\rho e^{\rho\psi}}{2}[2\varrho-D_0\psi(\ln b)]|X|_{g_0}^2+\frac{\rho e^{\rho\psi}}{b}[\rho b(X\psi)^2+(X\psi)(Xb)].
\end{split}
\end{equation}
Observe that
\begin{equation}
\rho b(X\psi)^2+(X\psi)(Xb)=\frac 14[2(\rho b)^{\frac12}(X\psi)+(\rho  b)^{-\frac12}(Xb)]^2-\frac{1}{4\rho b}(Xb)^2.\nonumber
\end{equation}
Hence, by \eqref{conformal-LC-connection}, we have
\begin{equation}
\widetilde D^2\widetilde \psi(X,X)\ge \frac{be^{\rho\psi}}{2}\Bigl[(2\varrho-D_0\psi(\ln b))\rho-\frac{1}{2b^2}\max_{\overline\Omega}|D_0b|_{g_0}^2\Bigr]|X|_{\tilde g}^2.\nonumber
\end{equation}
Recall that $b\in\mathcal U_K(M_0)$ is smooth on $\overline\Omega$. If there is a constant $\varrho_0>0$ such that
\begin{equation}
D_0\psi(\ln b)=b^{-1}\langle D_0\psi,D_0b\rangle_{g_0}\le 2\varrho-\varrho_0,\quad {\rm in}\,\,\Omega,\nonumber
\end{equation}
then, by taking
\begin{equation}
\rho_0\vcentcolon=2[(\min_{\overline\Omega}b)^{-1}+\varrho_0^{-1}\max_{\overline\Omega}(b^{-3}|D_0b|_{g_0}^2)],\nonumber
\end{equation}
we can conclude that
\begin{equation}
\widetilde D^2\widetilde\psi (X,X)\ge \varrho_0|X|_{\widetilde g}^2,\quad \forall\rho\ge\rho_0,\nonumber
\end{equation}
as desired.  This completes the proof.  
\end{proof}
\begin{remark}
By \eqref{conformal-LC-connection}, we can actually obtain
\begin{equation}
\widetilde D^2\widetilde\psi=\rho e^{\rho\psi}D_0^2\psi+\frac{\rho e^{\rho\psi}}{b}[{\rm sym}(D_0\psi\otimes D_0b)-1/2D_0\psi(b)g_0]+\rho^2e^{\rho\psi}D_0\psi\otimes D_0\psi,\nonumber
\end{equation}
where ${\rm sym}(D_0\psi\otimes D_0b)=\frac12(D_0\psi\otimes D_0b+D_0b\otimes D_0\psi)$.
\end{remark}
\begin{remark}
For our case, $g_0=ds^2$, the Euclidean metric, and $g=b^{-1}ds^2$, $D_0\psi=\nabla\psi$, we can simply choose $\psi=\frac12|x-x_0|^2$ with $x_0\in\R^n\backslash{\overline\Omega}$ and $\varrho=1$. The condition \eqref{condition-rho} implies 
\begin{equation}\label{cond:b-x-x0}
\nabla\ln b\cdot(x-x_0)\le 2-\varrho_0,\quad {\rm in}\, \Omega.
\end{equation}
For this case, $e^{\rho|x-x_0|^2}$ is a strictly convex function on $(\Omega,g)$ provided that $\rho$ is sufficiently large such that $\rho\ge\rho_0=2[(\min_{\overline\Omega}b)^{-1}+\varrho_0^{-1}\max_{\overline\Omega}(b^{-1}|\nabla b|^2)]$. 
Particularly, if $b$ is a constant, then the condition \eqref{cond:b-x-x0} clearly holds. 
\end{remark}
\begin{remark}
Let us recall the widely used Herglotz condition $\partial_r\bigl(rb^{-\frac12}(r)\bigr)>0$ for $b$ implying that $r\partial_r\ln b(r)<2$. Here $r=|x|$, and $b=b(r)$ is assumed to be an isotropic radial function in $\Omega$. Let $x_0=0_{\mathbb R^n}$ be the origin of $\mathbb R^n$.  In inequality \eqref{cond:b-x-x0}, if $b=b(r)$, then it becomes $r\partial_r\ln b(r)\le 2-\varrho_0<2$. Hence, the condition \eqref{cond:b-x-x0} clearly implies the Herglotz condition. We also refer to e.g., \cite{greene1976c,paternain2019geodesic,yao2011modeling} for  some discussions of the (local or global) existence of strictly convex functions on Riemannian manifolds.
\end{remark}

\section{Well-posedness of the forward problem}\label{sec:well-posedness}
In this section, we study the well-posedness of the forward linear MGT equation (with lower-order perturbations) and the nonlinear problem \eqref{eq:linear-JMGT-Sec1-intro}.  Denote
$$ L_{\alpha,b,c,\lambda,q}:=\partial_t^{3}+\alpha\partial_t^{2}-b\Delta\partial_t-c^{2}\Delta+\lambda\partial_t+q.$$
We first consider the linear MGT equation 
\begin{equation}\label{eq:full}
\begin{cases}
L_{\alpha,b,c,\lambda,q}u=h & {\rm in}\,\, \Omega_T,\\
u=f & {\rm on}\,\, \Gamma_T,\\
u(0)=\partial_tu(0)=\partial_t^2u(0)=0 & {\rm in}\,\, \Omega.
\end{cases}
\end{equation}
On an interval \(I=(t_*,t_*+\tau)\) for some $t_*\ge 0$ and $\tau>0$, the same notation denotes the corresponding
translated spaces. By Sobolev multiplication, \(E^m(\Omega_T)\) is a
Banach algebra when \(m>n+1\).

\noindent {\bf Well-posedness of the linear equation.}
We introduce the compatibility conditions for the linear equation.
\begin{definition}[Compatibility conditions of order \(m+1\)]
\label{def:compat-high}
Let \(m\ge1\) be an integer, \(h\in E^m(\Omega_T)\), and
\(f\in\mathcal A^m(\Gamma_T)\). Define
\begin{equation}\label{eq:compat-high}
\begin{aligned}
U_0&=U_1=U_2=0,\\
U_k&=-\alpha U_{k-1}+b\Delta U_{k-2}+c^2\Delta U_{k-3}
-\lambda U_{k-2}-qU_{k-3}
+\partial_t^{k-3}h(0,\cdot),\\
&\hspace{73mm}3\le k\le m+2.
\end{aligned}
\end{equation}
We say that \((f,h)\) satisfies the compatibility conditions up to order
\(m+1\) if
\begin{equation}\label{eq:compat-traces}
\partial_t^kf(0,\cdot)=U_k|_\Gamma,\qquad 0\le k\le m+1.
\end{equation}
\end{definition}

For \(h=0\), all the jets in \eqref{eq:compat-high} vanish, so
\eqref{eq:compat-traces} is precisely \(f\in\mathcal A_0^m(\Gamma_T)\).
For a general source, \(\lambda\) and \(q\) occur in the recursion and no boundary vanishing conditions on either coefficient are needed.
We give the following estimate.
\begin{lemma}\label{lem:P0-source-energy}
Let \(m\ge1\) be an integer and
\[
P_0=\partial_t^3+\alpha\partial_t^2-b\Delta\partial_t-c^2\Delta,
\qquad
\alpha,b,c\in\mathcal U_{m+1}(M_0),\qquad b,c\ge m_0>0.
\]
Let \(0\le t_*<t_*+\tau\le T\), and set \(I=(t_*,t_*+\tau)\).
Suppose that \(G\in E^m(I\times\Omega)\), \(g\in\mathcal A^m(I\times\Gamma)\), and
\[
w_j\in H^{m+2-j}(\Omega),\qquad 0\le j\le2.
\]
Set \(W_j=w_j\) for \(0\le j\le2\), and define
\begin{equation}\label{eq:P0-restart-jets}
W_k=-\alpha W_{k-1}+b\Delta W_{k-2}+c^2\Delta W_{k-3}
+\partial_t^{k-3}G(t_*),\qquad 3\le k\le m+2.
\end{equation}
If
\begin{equation}\label{eq:P0-restart-compatibility}
\partial_t^kg(t_*)=W_k|_\Gamma,\qquad 0\le k\le m+1,
\end{equation}
then
\[
P_0w=G,\qquad w|_{I\times\Gamma}=g,\qquad
\partial_t^jw(t_*)=w_j,\quad 0\le j\le2,
\]
has a unique solution \(w\in E^{m+2}(I\times\Omega)\), with
\(\partial_\nu w\in\mathcal N_{m+1}(I\times\Gamma)\), and
\begin{equation}\label{eq:P0-reference-estimate}
\begin{aligned}
&\|w\|_{E^{m+2}(I\times\Omega)}
+\|\partial_\nu w\|_{\mathcal N_{m+1}(I\times\Gamma)}\\
&\le C\Bigl(\sum_{j=0}^2\|w_j\|_{H^{m+2-j}(\Omega)}+\|G\|_{E^m(I\times\Omega)}+\|g\|_{H^{m+2}(I\times\Gamma)}\Bigr).
\end{aligned}
\end{equation}
The constant is uniform for the stated coefficient class.
\end{lemma}

\begin{proof}
This is precisely \cite[Lemma 2.2]{qiu2026inverse}, after translating the left
endpoint of the time interval to zero.  
\end{proof}

\begin{proposition}[Well-posedness of the linear problem \eqref{eq:full}]
\label{prop:wellposed}
Let \(m\ge 0\) be an integer. Assume
\[
\alpha,b,c,\lambda,q\in\mathcal U_{m+1}(M_0),\, h\in E^m(\Omega_T),\, f\in\mathcal A^m(\Gamma_T)
\qquad b,c\ge m_0>0.
\]
and assume that the compatibility conditions
\eqref{eq:compat-high}--\eqref{eq:compat-traces} hold. Then the linear problem \eqref{eq:full} admits a unique solution
\[
u\in E^{m+2}(\Omega_T),\qquad
\partial_\nu u\in\mathcal N_{m+1}(\Gamma_T),
\]
satisfying 
\begin{equation}\label{eq:est-high}
\|u\|_{E^{m+2}(\Omega_T)}
+\|\partial_\nu u\|_{\mathcal N_{m+1}(\Gamma_T)}
\le C\left(
\|h\|_{E^m(\Omega_T)}
+\|f\|_{H^{m+2}(\Gamma_T)}
\right),
\end{equation}
where the positive constant $C$ depends only on \(\Omega,T,m,m_0\), and \(M_0\). 
\end{proposition}

\begin{proof}
Let $B=\lambda\partial_t+q.$ 
We first consider the time interval \(I=(t_*,t_*+\tau)\). Fix the boundary and initial data along with all higher-order derivatives up to order $m+1$ that are determined recursively by the equation. Let \(\mathcal X_I\subset E^{m+2}(I\times\Omega)\) denote the closed set of functions satisfying the prescribed boundary data and the initial compatibility conditions at the left endpoint of \(I\). The higher-order initial time derivatives up to order (m+1) are determined recursively by the equation.
For \(v\in\mathcal X_I\), let \(\Phi(v)=w\) be the solution, guaranteed by
Lemma~\ref{lem:P0-source-energy}, of the linear equation
\begin{equation}\label{eq:linear-fixed-point-map}
P_0w=h-Bv,\qquad w|_{I\times\Gamma}=f,
\end{equation}
with the prescribed initial values at \(t=t_*\). Since
the coefficients are independent of time,
\[
\partial_t^{k-3}(Bv)(t_*)
=\lambda\,\partial_t^{k-2}v(t_*)
+q\,\partial_t^{k-3}v(t_*),\qquad 3\le k\le m+2.
\]
Consequently, the recursion \eqref{eq:P0-restart-jets} for \(w\) is
exactly the full recursion \eqref{eq:compat-high} (or its translated
version at a restart time). Thus \(\Phi(\mathcal X_I)\subset\mathcal X_I\).

If \(v_1,v_2\in\mathcal X_I\) and \(z=v_1-v_2\), then
\(\partial_t^kz(t_*)=0\) for \(0\le k\le m+1\). For \(0\le j\le m\),
the fundamental theorem of calculus and the multiplier estimate give
\[
\sup_{t\in I}\|\partial_t^j(Bz)(t)\|_{H^{m-j}(\Omega)}
\le CM_0\tau\|z\|_{E^{m+2}(I\times\Omega)}.
\]
Hence
\begin{equation}\label{eq:B-short-estimate}
\|Bz\|_{E^m(I\times\Omega)}
\le CM_0\tau\|z\|_{E^{m+2}(I\times\Omega)}.
\end{equation}
Applying \eqref{eq:P0-reference-estimate} yields
\[
\|\Phi(v_1)-\Phi(v_2)\|_{E^{m+2}(I\times\Omega)}
\le CM_0\tau\|v_1-v_2\|_{E^{m+2}(I\times\Omega)}.
\]
Choosing a uniform \(\tau_0>0\) so that \(CM_0\tau_0<1\), Banach's
fixed-point theorem gives a unique solution on every interval of length
at most \(\tau_0\).

At the right endpoint of each time interval, the solution already
constructed on the preceding interval, together with its time
derivatives up to the required order, is taken as the initial data for
the next interval. The same contraction argument applies, since the
difference of two successive iterates has vanishing initial traces up
to order $m+1$. Repeating this procedure a finite number of times
yields a solution on the whole interval $[0,T]$. Moreover, the standard \emph{a priori} estimate for the fixed point, combined
with the bounded lifting associated with the compatible initial and
boundary data, gives, on each such interval $I$, the estimate
\[
\|u\|_{E^{m+2}(I\times\Omega)} \le C\Bigl(\sum_{j=0}^2 \|\partial_t^j u(t_*)\|_{H^{m+2-j}(\Omega)}+\|h\|_{E^m(I\times\Omega)}+\|f\|_{H^{m+2}(I\times\Gamma)}\Bigr).
\]
A finite induction over the subdivision of \([0,T]\) yields
\[
\|u\|_{E^{m+2}(\Omega_T)}\le C\left(\|h\|_{E^m(\Omega_T)}+\|f\|_{H^{m+2}(\Gamma_T)}\right).
\]
This construction also proves uniqueness of solutions.

Finally \(G=h-Bu\) belongs to \(E^m(\Omega_T)\), and
\[
\|G\|_{E^m(\Omega_T)}
\le C\bigl(\|h\|_{E^m(\Omega_T)}
+M_0\|u\|_{E^{m+2}(\Omega_T)}\bigr).
\]
Moreover, its the compatibility recursion ( related to $P_0$) is precisely \eqref{eq:compat-high}. Applying Lemma~\ref{lem:P0-source-energy} on \((0,T)\) gives \eqref{eq:est-high}. 
\end{proof}

\begin{remark} Denote by
\begin{equation}\label{eq:adjoint-L-abclamq}
L_{\alpha,b,c,\lambda,q}^*v\vcentcolon=-\partial_t^3 v + \alpha(x)\partial_t^2 v 
+ \partial_t \Delta \bigl(b(x)v\bigr) 
- \Delta \bigl(c(x)^2 v\bigr) 
- \lambda(x)\partial_t v + q(x)v
\end{equation}
the adjoint operator of $L_{\alpha,b,c,\lambda,q}$ with respect to $L^2$.
Under the assumptions of Proposition~\ref{prop:wellposed}, assume
that \(b,c\in\mathcal U_{m+2}(M_0)\). Let
\(G\in E^m(\Omega_T)\) and \(g\in\mathcal A^m(\Gamma_T)\) satisfy the
terminal compatibility conditions. Then the adjoint-backward problem
\[
\begin{cases}
L_{\alpha,b,c,\lambda,q}^*y=G&\text{in }\Omega_T,\\
y=g&\text{on }\Gamma_T,\\
y(T)=\partial_ty(T)=\partial_t^2y(T)=0&\text{in }\Omega
\end{cases}
\]
has a unique solution \(y\in E^{m+2}(\Omega_T)\), with
\(\partial_\nu y\in\mathcal N_{m+1}(\Gamma_T)\), and
\[
\|y\|_{E^{m+2}(\Omega_T)}
+\|\partial_\nu y\|_{\mathcal N_{m+1}(\Gamma_T)}
\le C\left(
\|G\|_{E^m(\Omega_T)}
+\|g\|_{H^{m+2}(\Gamma_T)}
\right).
\]
\end{remark}

\begin{remark}
We note that the regularity conditions on the coefficients $\alpha,b,c,\lambda,q$ are not sharp. Such conditions are only for the convenience of discussions.
\end{remark}

\noindent {\bf Local well‑posedness of the nonlinear JMGT equation.}
\label{sec:Westervelt}
We now consider the following nonlinear equation with zero initial conditions
\begin{equation}\label{eq:Westervelt}
\begin{cases}
L_{\alpha,b,c,\lambda,q}u=\partial_t^2(\xi u^2)
&\text{in }\Omega_T,\\
u=f&\text{on }\Gamma_T,\\
u(0)=\partial_tu(0)=\partial_t^2u(0)=0&\text{in }\Omega.
\end{cases}
\end{equation}

\begin{proposition}[local well-posedness of \eqref{eq:Westervelt}]
\label{prop:Westervelt}
Let \(m>n+1\) be an integer, and assume that
\[
\alpha,b,c,\lambda,q,\xi\in\mathcal U_{m+1}(M_0),\quad f\in\mathcal A_0^m(\Gamma_T),\,\,
\quad b,c\ge m_0>0.
\]
There are constants
\(\delta_{\rm wp},r_0>0\), depending only on $\Omega,T,M_0,m_0$ such that if
\[
\|f\|_{H^{m+2}(\Gamma_T)}\le\delta_{\rm wp},
\]
then \eqref{eq:Westervelt} has a unique solution in the ball
\(\{u\in E^{m+2}(\Omega_T):\|u\|_{E^{m+2}}\le r_0\}\). Moreover, $\partial_\nu u\in\mathcal N_{m+1}(\Gamma_T)$ 
and
\begin{equation}\label{eq:Westervelt-est}
\|u\|_{E^{m+2}(\Omega_T)}
+\|\partial_\nu u\|_{\mathcal N_{m+1}(\Gamma_T)}
\le C\|f\|_{H^{m+2}(\Gamma_T)}.
\end{equation}
\end{proposition}
The above proposition can be proved by the Banach fixed point theorem, see for example \cite[Proposition 1.1]{qiu2026inverse} for more discussions. Therefore, we omit the details.



\section{Construction of Gaussian beam solutions}\label{sec:Gaussian-beams}
In this section, we will construct solutions to the linear equation
\begin{equation}
\begin{cases}
L_{\alpha,b,c,\lambda,q}v=0 & {\rm in}\,\, \Omega_T,\\
v=f & {\rm on}\,\, \Gamma_T,\\
v(0)=\partial_tv(0)=\partial_t^2v(0)=0  & {\rm in}\,\, \Omega,
\end{cases}
\end{equation}
of the form 
\begin{equation}
v(x,t)=e^{i\sigma\varphi(x,t)}a_{\sigma}(x,t)+r_\sigma(x,t),
\end{equation}
near a null geodesic, where the functions $\varphi$ and $a_\sigma$ respectively refer to the phase and amplitude function that will be constructed later, $r_\sigma$ is the associated remainder term. We note that the principal term $v_\sigma\vcentcolon=e^{i\sigma\varphi(x,t)}a_{\sigma}(x,t)$ is called the Gaussian beam concentrating on a null geodesic.

Recalling the Riemannian metric $g=b^{-1}ds^2$ and the related geometric setting introduced in Section \ref{subsection-1.2}, we can rewrite
\begin{equation}
\begin{split}
L_{\alpha,b,c,\lambda,q}&=\partial_t^3+\alpha\partial_t^2-\Delta_g\partial_t-\beta\Delta_g+\lambda\partial_t\\
&\quad +\frac{2-n}{2}\langle D\ln b,D\partial_t\cdot\rangle_g+\frac{(2-n)\beta}{2}\langle D\ln b, D\cdot\rangle_g+q,
\end{split}
\end{equation}
and denote 
\begin{equation}
P_{\alpha,\beta,\lambda}:=\partial_t^3+\alpha\partial_t^2-\Delta_g\partial_t-\beta\Delta_g+\lambda\partial_t.
\end{equation}
Clearly, if $n=2$, then $L_{\alpha,b,c,\lambda,q}-q=P_{\alpha,\beta,\lambda}$. 

For the construction of Gaussian beams, we consider the equation $L_{\alpha,b,c,\lambda,q}v=0$ in an extended domain $\mathcal M\vcentcolon=[0,T]\times \widetilde\Omega$, where $\Omega\subset\subset\widetilde\Omega$ is the smooth extension of $\Omega$. Accordingly, by extending the function $b$ smoothly ensuring $b>0$ in $\widetilde\Omega$, we can obtain a new metric (still denoted by $g$) in $\widetilde\Omega$. We also extend smoothly the coefficients $\alpha,\lambda,q$ and $\xi$ to $\widetilde\Omega$ in order to construct Gaussian beams in $\mathcal M$.
Now we consider $(\mathcal M,\bar g)$ as a Lorentzian manifold with the metric $\bar g=-dt^2+ g$. Let $\vartheta(t)=(t,\gamma(t))$ be a null geodesic on $(\mathcal M,\bar g)$, where $\gamma(t)$ is a unit-speed geodesic on the Riemannian manifold $(\widetilde\Omega,g)$. For the purpose of constructing Gaussian beams near the null geodesic $\vartheta$, we next introduce the Fermi coordinates, which are centered around $\vartheta$ where the metric $\overline g$ has a simple form along $\vartheta$. 

We note that the authors of the present paper (D. Qiu and T. Zhou), together with Y. Ye and X. Xu, have constructed Gaussian beams for the linear MGT equation without considering the lower‑order term $\lambda\partial_tu+qu$ in their recent work \cite{qiu2026inverse}. Here, for the sake of completeness and the readers' convenience, we present the detailed construction of such Gaussian beam solutions for the linear equation $L_{\alpha,b,c,\lambda,q}v=0$.


\subsection{Fermi coordinates}
We assume that $\vartheta(t)$ passes through a point $(t_0,x_0),$ where $t_0\in(0,T)$ and $\gamma(t_0)=x_0\in \Om$, and $\vartheta$ joins two points $(t_-,\gamma(t_-))$ and $(t_+,\gamma(t_+))$ with $t_{\pm}\in (0,T)$ and $\gamma(t_{\pm})\in \Gamma.$
Recalling that $T>{\rm diam}_g\Omega$, we extend $\vartheta$ to $\mathcal M$ such that $\gamma(t)$ is well defined on $[t_--\epsilon,t_++\epsilon]\subset (0,T)$
for some sufficiently small constant $\epsilon>0$.

We conclude the following lemma, which can be found in \cite[Lemma 1]{feizmohammadi2022recovery} (see also \cite[Lemma 3.1]{feizmohammadi2021recovery}). 
\begin{lemma}[Fermi coordinates]
\label{lem:Fermi-coord}
Let $\vartheta:(t_--\epsilon,t_++\epsilon)\rightarrow \mathcal M$ be a null geodesic on $(\mathcal M,\bar g)$. Then, for some $0<2\epsilon'<\epsilon/\sqrt2$, there exists a coordinate neighborhood $(U,\Phi)$ of $\vartheta([t_--\epsilon'/\sqrt2,t_++\epsilon'/\sqrt2]),$ with the coordinates denoted by $(z_0=s,z'=(z_1,\cdots,z_n))$ such that
$$\begin{gathered}
V=\Phi(U)=\bigl(\sqrt2(t_--t_0)-2\epsilon',\sqrt2(t_+-t_0)+2\epsilon'\bigr)\times B(0,\delta),\\
\Phi(\vartheta(t))=(\sqrt2(t-t_0),0,\cdots,0),\qquad t=t_0+\frac{s-z_1}{\sqrt2}.
\end{gathered}$$
where $B(0,\delta)$ denotes a ball in $\R^n$ with a sufficiently small radius $\delta.$
Moreover, in this coordinate system, the metric $\bar g$ can be written as
\begin{equation}
\bar g(s,z')|_\vartheta=2dsdz_1+\sum_{k=2}^ndz_k^2,\quad \frac{\partial \bar g_{jk}}{\partial z_i}\Big|_{\vartheta}=0,\quad i,j,k=0,1,\cdots,n.
\end{equation}
\end{lemma}
Let $a_0=\sqrt2(t_--t_0)-\epsilon', b_0=\sqrt2(t_+-t_0)+\epsilon'$. We define the tubular neighborhood $\mathcal V$ of $\vartheta$ by
$$\mathcal V=\Phi^{-1}\bigl([a_0,b_0]\times\overline{B(0,\delta')}\bigr)$$ with $\delta/2<\delta'<\delta$ sufficiently small such that the set $\mathcal V$ does not
intersect the bottom $\{0\}\times \Om$ and top $\{T\}\times \Om.$
Here and below, $\mathcal V_\delta$ denotes this tubular neighborhood, and $\vartheta$ is parametrized by $s$. We take $2\epsilon'<\epsilon/\sqrt2$ so that the closed tube lies in the coordinate neighborhood. The radii $\delta,\delta'$ in this construction are fixed geometric constants, independent of the measurement error denoted by $\delta$ in the main theorems.

We will focus on solving $L_{\alpha,b,c,\lambda,q}v=0$ in the tubular region $\mathcal V_\delta$, around the central geodesic $\vartheta$ where the Fermi coordinates $(s,z')$ are well-defined. In other words, we shall construct Gaussian beam via the WKB ansatz 
\begin{equation}
v_\sigma(s,z')=e^{i\sigma\varphi(s,z')}a_\sigma(s,z'),
\end{equation}
which approximately solves the linear equation $L_{\alpha,b,c,\lambda,q}v_\sigma=0$ in Fermi coordinates.
Here the phase function $\varphi$ near $\vartheta$ is given by
\begin{equation}
\varphi(s,z')=\sum_{k=0}^N\varphi_k(s,z'),\quad \varphi_k(s,z')=\sum_{j_1,\ldots,j_k=1}^n\Phi^k_{j_1\cdots j_k}(s)z_{j_1}\cdots z_{j_k},
\end{equation}
where $\Phi^k_{j_1\cdots j_k}(s)\in\mathbb C$ are smooth and symmetric coefficients for $k\ge 2$. In order to choose suitable ansatz for the amplitude function $a_\sigma$, we make the following observation.

Set $\widetilde b=b^{1-n/2}$, $\beta=c^2/b$, $\Upsilon=\alpha-\beta$ and $L=L_{\alpha,b,c,\lambda,q}$.
A direct computation gives
\begin{equation}
\begin{aligned}
e^{-i\sigma\varphi}L(e^{i\sigma\varphi}a_\sigma)
={}&i\sigma^3\varphi_t a_\sigma(\mathcal S\varphi)
+\sigma^2\mathcal T_1(a_\sigma,\varphi) +i\sigma\mathcal T_2(a_\sigma,\varphi)+La_\sigma,
\end{aligned}
\end{equation}
where $\square_{\bar g}=-\partial_t^2+\Delta_g$, and
\begin{equation}
\mathcal S\varphi
=\langle d\varphi,d\varphi\rangle_{\bar g}
=\langle D\varphi,D\varphi\rangle_g-\varphi_t^2,
\end{equation}
\begin{equation}
\begin{aligned}
\mathcal T_1(a_\sigma,\varphi)
={}&2\varphi_t\langle da_\sigma,d\varphi\rangle_{\bar g}
+(\mathcal S\varphi)a_{\sigma t}
+2a_\sigma\langle d\varphi,d\varphi_t\rangle_{\bar g}\\
&+a_\sigma\varphi_t
 \bigl (\square_{\bar g}\varphi
 -\langle d\varphi,d\ln\widetilde b\rangle_{\bar g}\bigr)+\bigl(\beta(\mathcal S\varphi)-\Upsilon\varphi_t^2\bigr)a_\sigma,
\end{aligned}
\end{equation}
\begin{equation}
\begin{aligned}
\mathcal T_2(a_\sigma,\varphi)
={}&3\varphi_t a_{\sigma tt}
+3\varphi_{tt}a_{\sigma t}
+\varphi_{ttt}a_\sigma+\alpha(2\varphi_t a_{\sigma t}+\varphi_{tt}a_\sigma)
+\lambda\varphi_t a_\sigma\\
&-\Delta_g(\varphi_t a_\sigma)
+\langle D\ln\widetilde b,D(\varphi_t a_\sigma)\rangle_g-2\langle D\varphi,Da_{\sigma t}+\beta Da_\sigma\rangle_g\\
&-\bigl(\Delta_g\varphi
-\langle D\ln\widetilde b,D\varphi\rangle_g\bigr)
(a_{\sigma t}+\beta a_\sigma).
\end{aligned}
\end{equation}

For an integer $N\ge2$, we take the ansatz
\begin{equation}\label{eq:with-cut-off}
a_\sigma(s,z')
=\chi\Bigl(\frac{|z'|}{\delta}\Bigr)
\sum_{k=0}^{N}\sigma^{-k}e_k(s,z'),
\qquad
e_k(s,z')=\sum_{j=0}^{N}e_{k,j}(s,z'),
\end{equation}
where every $e_k$ is independent of $\sigma$, and $e_{k,j}$ is
homogeneous of degree $j$ in $z'$.
Here $\chi\in C_c^\infty(\mathbb R)$ is a nonnegative cut-off function,
$\chi(\tau)=1$ for $|\tau|\le1/4$, and
$\chi(\tau)=0$ for $|\tau|\ge1/2$.
Set $e_{-1}=e_{-2}=0$. Collecting powers of $\sigma$
in the region where $\chi=1$ yields the exact identity
\begin{equation}
\begin{aligned}
e^{-i\sigma\varphi}L(e^{i\sigma\varphi}a_\sigma)&=i\sigma^3\varphi_t(\mathcal S\varphi)
  \sum_{k=0}^{N}\sigma^{-k}e_k+\sum_{k=0}^{N}\sigma^{2-k}
 \bigl(\mathcal T_1(e_k,\varphi)
 +i\mathcal T_2(e_{k-1},\varphi)+Le_{k-2}\bigr)\\
&\quad+\sigma^{1-N}
 \bigl(i\mathcal T_2(e_N,\varphi)+Le_{N-1}\bigr)
 +\sigma^{-N}Le_N.
\end{aligned}
\end{equation}
We construct the phase to sufficiently high transverse order so that
\begin{equation}\label{eq:eikonal-N-th}
\partial_{z'}^\theta(\mathcal S\varphi)(s,0)=0,
\qquad |\theta|\le N,
\end{equation}
and determine the amplitudes successively from
\begin{equation}\label{eq:trans-e-k-1}
\partial_{z'}^\theta\mathcal T_1(e_0,\varphi)(s,0)=0,
\qquad |\theta|\le N,
\end{equation}
and
\begin{equation}\label{eq:trans-e-k-2}
\begin{aligned}
\partial_{z'}^\theta
\bigl(\mathcal T_1(e_k,\varphi)
+i\mathcal T_2(e_{k-1},\varphi)+Le_{k-2}\bigr)(s,0)&=0, \quad \text{ for all }|\theta|\le N,\ k=1,\ldots,N,
\end{aligned}
\end{equation}
for $s\in[a_0,b_0]$.
Along \(\vartheta\), the coefficient of \(\partial_s\) in the transport
operator is nonvanishing, while the coefficients of the transverse
derivatives vanish. After dividing the transport equation by the
coefficient of \(\partial_s\) and comparing the terms of the same
degree in the transverse Taylor expansion, one obtains, at each degree,
a linear ordinary differential equation in \(s\). Its right-hand side
depends only on lower-degree terms of the same amplitude and on the
amplitude polynomials determined at the preceding steps. Hence the
coefficients can be constructed successively, and the finite transport
recursion is well defined.


\subsection{Construction of the phase function}
Let us proceed to construct the phase function satisfying the eikonal equation \eqref{eq:eikonal-N-th} and the conditions
\begin{equation}\label{cond:for-Im-phase}
\Im\varphi\ge 0,\quad \Im\varphi|_\vartheta=0,\quad \Im\varphi(s,z')\ge C|z'|^2,\quad \forall (s,z')\in\mathcal V_\delta.
\end{equation}
We begin with solving the equation \eqref{eq:eikonal-N-th} with $|\theta|=0$. For simplicity, we use the notion $\partial_k$ to denote $\partial_{z_k}$ for $k=0,1,\cdots,n$. When $|\theta|=0$, the  equation becomes
$$\sum\limits_{k,l=0}^n\bar g^{kl}\partial_k\varphi\partial_l\varphi\Big|_\vartheta=0.$$
By Lemma \ref{lem:Fermi-coord}, on $\vartheta$, this reduces to
\begin{equation}\label{var0}
2\partial_0\varphi\partial_1\varphi+\sum\limits_{k=2}^n(\partial_k\varphi)^2\Big|_\vartheta=0.
\end{equation}
Similarly, for $|\theta|=1,$ we have
\begin{equation}\label{var1}
\sum\limits_{k,l=0}^n\bar g^{kl}\partial^2_{jk}\varphi\partial_l\varphi\Big|_\vartheta=0\quad {\rm for}\quad 1\le j\le n.
\end{equation}
Clearly, equations \eqref{var0} and \eqref{var1} are satisfied  by respectively choosing
$$\varphi_0=0,\quad \varphi_1=z_1.$$
For the case $|\theta|=2,$ we take
\begin{equation}\label{var2}
\varphi_2(s,z')=\sum\limits_{i,j=1}^nH_{ij}(s)z_iz_j,
\end{equation}
where $(H_{ij})_{1\le i,j\le n}$ is a symmetric and complex-valued matrix such that $\Im H$ is positive definite. It follows from Lemma \ref{lem:Fermi-coord} and \eqref{var2} that
\begin{equation}\label{var22}
\sum\limits_{k,l=0}^n(\partial_{ij}^2\bar g^{kl}\partial_k\varphi\partial_l\varphi+2\bar g^{kl}\partial^3_{kij}\varphi\partial_l\varphi+2\bar g^{kl}\partial^2_{ki}\varphi\partial_{lj}^2\varphi+4\partial_i\bar g^{kl}\partial^2_{jk}\varphi\partial_l\varphi)|_{\vartheta}=0.
\end{equation}
By the choices of $\varphi_0, \varphi_1$ and $\varphi_2,$ (\ref{var22}) implies that
$$(\partial^2_{ij}\bar g^{11}+2\bar g^{10}\partial^3_{0ij}\varphi+2\sum\limits_{k=2}^n\partial^2_{ki}\varphi\partial^2_{kj}\varphi)|_{\vartheta}=0.$$
We finally obtain the following Riccati equation for $H(s)$,
\begin{equation}\label{eq:riccati}
\frac{d}{ds}H+HAH+B=0,\quad H(s_0)=H_0,\quad {\rm with}\quad \Im H_0>0\ {\rm and}\ s\in [a_0,b_0],
\end{equation}
where $s_0=a_0$, $B=(B_{ij})_{i,j=1}^n$, $B_{ij}(s)=\frac14\partial_{ij}^2\bar g^{11}(s,0)$,  and the components of $A=(A_{ij})$ satisfy
$$\left\{ \begin{array}{l}
A_{11}=0,\\
A_{ii}=2,\quad i=2,\cdots,n,\\
A_{ij}=0,\quad {\rm otherwise}.
\end{array} \right.$$
For the above Riccati equation, we have
\begin{lemma}\label{lem-solve-riccati}\cite[Lemma 2.56]{kachalov2001inverse}
The Riccati equation \eqref{eq:riccati} admits a unique solution $H$, which is
symmetric and $\Im(H(s))>0$ for all $s\in (a_0,b_0).$ We have $H(s)=Z(s)Y^{-1}(s),$ where the matrix
valued functions $Z(s), Y(s)$ solve the first order linear system
$$\frac{dZ}{ds}=-BY,\quad \frac{dY}{ds}=AZ,\quad {\rm subject\ to}\quad Y(a_0)=I,\quad Z(a_0)=H_0.$$
Moreover, the matrix $Y(s)$ is non-degenerate on $(a_0,b_0)$, and there holds
\begin{equation}\label{iden-HY-matrix}
{\rm det}(\Im H(s))\cdot |{\rm det}(Y(s))|^2={\rm det}(\Im H_0).
\end{equation}
\end{lemma}
By \cite[Lemma 3.3]{feizmohammadi2021recovery}, the solution $H\in C^3([a_0,b_0];\mathbb C^{n\times n})$ and $\varphi_2\in C^3(\mathcal V_\delta)$ in the Fermi coordinates.
The ODEs inherit the coefficient regularity; all derivative orders below are chosen within the available a priori regularity.
For the case $|\theta|=3,4\cdots,$ the polynomials $\varphi_j$ of higher degree obey some linear non-homogeneous ODEs, and thus they can be constructed analogously. We refer to, e.g., \cite[Section 4.3]{wendels2025stable} for more details.

\subsection{Construction of the amplitude function}
We will construct the amplitude function for the Gaussian beam by solving the equations \eqref{eq:trans-e-k-1}--\eqref{eq:trans-e-k-2}. We start with constructing the leading term $e_0$. 

Recalling the Fermi coordinates and the relation $\partial_t\varphi=(\partial_s\varphi-\partial_{z_1}\varphi)/\sqrt2$, and using the condition \eqref{eq:eikonal-N-th},  we have, along the null geodesic $\vartheta$, $\varphi_t=-\frac{\sqrt 2}{2}$, and
\begin{equation}
\langle d\varphi, d\varphi_t\rangle_{\bar g}=\frac12\partial_t(\mathcal S\varphi)=\frac{1}{2\sqrt2}(\partial_s-\partial_{z_1})(\mathcal S\varphi)(s,0,\cdots,0)=0,\quad s\in [a_0,b_0].
\end{equation}
Moreover, we have $\varphi_t\ne 0$ in $\overline{\mathcal V}$ (see \cite[Lemma 3.6]{feizmohammadi2021recovery}). Hence to construct $e_0$, by \eqref{eq:trans-e-k-1}--\eqref{eq:trans-e-k-2}, it suffices to solve the equations
\begin{equation}
\begin{split}
\frac{\partial^{|\theta|}}{\partial {z'}^\theta}[\mathcal T_1(e_0^1,\varphi)](s,0,\cdots,0)&=0,\\
\frac{\partial^{|\theta|}}{\partial {z'}^\theta}[\mathcal T_1(e_0^2,\varphi)+i\mathcal T_2(e_0^1,\varphi)](s,0\cdots,0)&=0,
\end{split}
\end{equation}
where $s\in [a_0,b_0]$ and $\mathcal T_0(\cdot,\varphi)=(2\varphi_t)^{-1}\mathcal T_1(\cdot,\varphi)$. 

To retain the notation used below, $e_0^1=e_0$, $e_0^2=e_1$, $e_{0,j}^1=e_{0,j}$ and $e_{0,j}^2=e_{1,j}$ denote the first two levels of the single sequence in \eqref{eq:with-cut-off}.
Let us now consider the case $j=0$ and construct the terms $e_{0,0}^1$ and $e_{0,0}^2$ by solving the following equations on $\vartheta$,
\begin{equation}\label{eq:trans-eqns-N=0}
\begin{split}
\mathcal T_0(e_{0,0}^1,\varphi)&=\langle de_{0,0}^1,d\varphi\rangle_{\bar g}+\frac12\left[\square_{\bar g}\varphi-\langle d\varphi, d\ln \widetilde b\rangle_{\bar g}+\frac{\sqrt 2}{2}\Upsilon\right]e^1_{0,0}=0,\\
\mathcal T_1(e_{0,0}^2,\varphi)&=-i\mathcal T_2(e_0^1,\varphi)\big|_\vartheta.
\end{split}
\end{equation}
The second equation is solved after all the terms in $e_0^1$ have been determined.
Observe that in the Fermi coordinates,
\begin{equation}
\square_{\bar g}\varphi=\sum_{k,l=0}^n\bar g^{kl}\partial^2_{kl}\varphi=\sum_{k=2}^n\partial_{kk}^2\varphi={\rm tr}(AH)\quad {\rm on}\,\, \vartheta.
\end{equation}
Furthermore, 
\begin{equation}
{\rm tr}(AH)={\rm tr}(Y_sY^{-1})={\rm tr}(\partial_s\log Y)=\partial_s(\log\det Y).
\end{equation}
We have, on $\vartheta$
\begin{equation}
\square_{\bar g}\varphi-\langle d\varphi, d\ln \widetilde b\rangle_{\bar g}=\partial_s(\ln \det Y(s))-\partial_s(\ln\widetilde b)=\partial_s[\ln(\widetilde b^{-1}\det Y(s))].
\end{equation}
Hence when $j=0$, the first equation of \eqref{eq:trans-eqns-N=0} becomes
\begin{equation}
\partial_se^1_{0,0}(s)+\frac12\Bigl\{\partial_s[\ln(\widetilde b^{-1}\det Y(s))]+\frac{\sqrt 2}{2}\Upsilon\Bigr\}e_{0,0}^1(s)=0,\quad s\in [a_0,b_0].
\end{equation}
The solution of the above ODE is given by
\begin{equation}
e^1_{0,0}(s)=\widetilde b^{\frac12}(s)(\det Y(s))^{-\frac12}e^{-\frac{\sqrt 2}{4}\int_{a_0}^s\Upsilon(\tau,0)d\tau},\quad s\in [a_0,b_0].
\end{equation}
Taking
\begin{equation}
2h(s)=\partial_s[\ln(\widetilde b^{-1}\det Y(s))]+\frac{\sqrt 2}{2}\Upsilon(s),\quad s\in [a_0,b_0],
\end{equation}
we can solve the second non-homogeneous equation of \eqref{eq:trans-eqns-N=0} to have
\begin{equation}
e_{0,0}^2(s)=e^{-\int_{a_0}^sh(\tau)d\tau}-i\int_{a_0}^se^{-\int_\tau^sh(l)dl}\frac{\mathcal T_2(e_0^1,\varphi)(\tau,0)}{2\varphi_t(\tau,0)}d\tau,\quad s\in [a_0,b_0],
\end{equation}
where we have set $e_{0,0}^2(a_0)=1$. The subsequent terms $e_{0,k}^l$ with $k=1,\cdots, N$ and $l=1,2$ can be constructed by solving linear first order ODEs. In fact, to determine $e^1_{0,k}$, it suffices to solve the equation
\begin{equation}
\partial_se^1_{0,k}+\langle AH(s)z',\nabla_{z'}e^1_{0,k}\rangle+h(s)e^1_{0,k}=\mathcal E_k,\quad s\in [a_0,b_0],
\end{equation}
where $\mathcal E_k$ is a homogeneous polynomial of degree $k$ in the $z'$-coordinates with the
coefficients only depending on $\{e_{0,j}^1\}_{j=0}^{k-1}$ and $\{\varphi_j\}_{j=0}^{\min\{N,k+2\}}$, together with the known metric coefficients. Hence we can obtain $e^1_{0,k}$ and thus $e^1_0$, since the non-homogeneous linear ODE system admits a unique solution for each $k\ge 1$, when the initial condition is imposed. 

Similarly, we can obtain $e_0^2$ by solving the first equation \eqref{eq:trans-e-k-2}. To determine the subsequent terms $e_k$, $k=2,\cdots,N$, we need to solve the equation \eqref{eq:trans-e-k-2}, but this can be accomplished analogously to the above argument. Therefore we omit the details for the sake of brevity. We also refer to e.g., \cite[Section 4.2.2]{feizmohammadi2022recovery} and \cite[Section 4.4]{wendels2025stable} for the construction of amplitude functions for second-order wave equations. 

\subsection{Construction of the remainder function}
Recalling that the remainder term $r_\sigma=v-v_\sigma$, we see that 
\begin{equation}\label{eq:lin-for-r-sigma}
\begin{cases}
L_{\alpha,b,c,\lambda,q}r_\sigma=-L_{\alpha,b,c,\lambda,q}v_\sigma & {\rm in}\,\, \Omega_T\\
r_\sigma=0 & {\rm on}\,\, \Gamma_T,\\
r_\sigma(0)=\partial_tr_\sigma(0)=\partial_t^2r_\sigma(0)=0 & {\rm in}\,\, \Omega.
\end{cases}
\end{equation}
By Proposition \ref{prop:wellposed}, and using the Sobolev embedding theorem, the above problem \eqref{eq:lin-for-r-sigma} admits a unique solution $r_\sigma\in H^{s+2}(\Omega_T)$ for integers $s\ge1$, such that
\begin{equation}\label{est:Hs+2-Hs+1}
\|r_\sigma\|_{H^{s+2}(\Omega_T)}\lesssim \|L_{\alpha,b,c,\lambda,q}v_\sigma\|_{E^s(\Omega_T)}\lesssim\|L_{\alpha,b,c,\lambda,q}v_\sigma\|_{H^{s+1}(\Omega_T)}.
\end{equation}
By some similar arguments to \cite[Lemma 2]{feizmohammadi2022recovery} (see also \cite[Section 3.6]{feizmohammadi2021recovery}) 
we can prove the following lemma implying that the remainder term $r_\sigma$ vanishes when $N$ is large enough and $\sigma\to\infty$.
\begin{lemma}
Let $\vartheta$ be a null geodesic in $(\Omega_T,\bar g)$ such that the end points of $\vartheta$ are outside $\Omega_T$ in the sense that $\vartheta(a_0),\vartheta(b_0)\notin\overline\Omega_T$. Let $v_\sigma$ be an approximate Gaussian beam of order $N$ around $\vartheta$ for the equation $L_{\alpha,b,c,\lambda,q}v=0$. Then the constructed Gaussian beam $v_\sigma$ and the remainder term $r_\sigma=v-v_\sigma$ satisfy
\begin{equation}\label{est:remainder-r-sigma}
\|v_\sigma\|_{H^k(\Omega_T)}\lesssim \sigma^{k-\frac n4},\quad
\|r_\sigma\|_{H^{k+1}(\Omega_T)}\lesssim\sigma^{-K},
\end{equation}
where $K=\frac N2+\frac n4-k-\frac 52>0$, $k\in\N$, $k\ge2$, and $N\ge\max\{2,k\}$. 
Moreover, the boundary Dirichlet data $f_{\sigma}\vcentcolon=v_\sigma|_{\Gamma_T}$ satisfies 
\begin{equation}
\|f_{\sigma}\|_{C^s(\Gamma_T)}\lesssim \sigma^s,\quad \|f_{\sigma}\|_{H^l(\Gamma_T)}\lesssim \sigma^{l+\frac12-\frac n4}
\end{equation}
for any $s\ge 0$ and $l>0$ within the available coefficient regularity. The implicit constants are independent of $\sigma$, for the fixed ray and tube.
\end{lemma}
\begin{proof}
Let $\theta$ be a $n+1$-dimensional multi-index. By the equations \eqref{eq:eikonal-N-th}, \eqref{eq:trans-e-k-1} and \eqref{eq:trans-e-k-2}, we can deduce that
\begin{equation}
\begin{split}
|\partial^\theta_z (L_{\alpha,b,c,\lambda,q}v_\sigma)|\lesssim |e^{i\sigma\varphi}|\Bigl(\sum_{j=0}^{|\theta|}\sigma^{|\theta|-j+3}|z'|^{N+1-j}+\sigma^{|\theta|+1-N}\Bigr),\quad |\theta|\le N.
\end{split}
\end{equation}
The terms containing derivatives of $\chi$ are exponentially small in every fixed Sobolev norm.
Let $C$ be the constant given in \eqref{cond:for-Im-phase}, and let
\begin{equation}
I(|z'|^m)\vcentcolon=\int_{|z'|\le\delta}|z'|^me^{-C_0\sigma|z'|^2}dz',\quad C_0=2C,\, m\in \mathbb N_0.
\end{equation}
Writing $r=|z'|$ and using the change of variable $\rho^2=\sigma r^2$, we have
\begin{equation}
\begin{aligned}
I(|z'|^m)&=|\mathbb S^{n-1}|\int_0^\delta r^{m+n-1}e^{-C_0\sigma r^2}\,dr\\
&=|\mathbb S^{n-1}|\sigma^{-\frac{m+n}{2}}\int_0^{\delta\sqrt\sigma}\rho^{m+n-1}e^{-C_0\rho^2}\,d\rho\lesssim\sigma^{-\frac{m+n}{2}}.
\end{aligned}
\end{equation}
By the third condition of \eqref{cond:for-Im-phase} for the phase function $\varphi$, we have 
\begin{equation}
|e^{i\sigma\varphi}|=e^{-\sigma\Im\varphi}\le e^{-C\sigma|z'|^2}.
\end{equation}
Hence, we can obtain
\begin{equation}
\begin{split}
\|\partial^\theta_z (L_{\alpha,b,c,\lambda,q}v_\sigma)\|^2_{L^2(\Omega_T)}&\lesssim \sum_{j=0}^{|\theta|}\sigma^{2(|\theta|-j+3)}I(|z'|^{2(N+1-j)})+\sigma^{2(|\theta|+1-N)}I(1)\\
&\lesssim \sigma^{2(|\theta|+3)-N-1-n/2},
\end{split}
\end{equation}
yielding that
\begin{equation}
\|L_{\alpha,b,c,\lambda,q}v_\sigma\|_{H^k(\Omega_T)}\lesssim \sigma^{-\frac N2-\frac n4+k+\frac 52},\quad k\in\N.
\end{equation}
By the equation \eqref{eq:lin-for-r-sigma}, it follows from the estimate \eqref{est:Hs+2-Hs+1}, with $s=k-1$, that
\begin{equation}
\|r_\sigma\|_{H^{k+1}(\Omega_T)}\lesssim \|L_{\alpha,b,c,\lambda,q}v_\sigma\|_{H^{k}(\Omega_T)}\lesssim \sigma^{-K},\quad k\ge2,
\end{equation}
We can argue similarly to have
\begin{equation}
\|v_\sigma\|_{H^k(\Omega_T)}\lesssim\sigma^{k-\frac n4}.
\end{equation}
Finally, by the boundary trace theorem, we have
\begin{equation}\label{est:f-sigma-BT}
\|f_\sigma\|_{H^l(\Gamma_T)}\lesssim \|v_\sigma\|_{H^{l+1/2}(\Omega_T)}\lesssim \sigma^{l+\frac12-\frac n4},\quad \forall l>0.
\end{equation}
This completes the proof.
\end{proof}

Recalling the adjoint operator $L^*_{\alpha,b,c,\lambda,q}$ in \eqref{eq:adjoint-L-abclamq}, we have
\begin{equation}
\begin{split}
L^*_{\alpha,b,c,\lambda,q}&=-\partial_t^3+\alpha\partial_t^2+b\Delta\partial_t-c^2\Delta-\lambda\partial_t\\
&\quad +(\Delta b)\partial_t+2\langle\nabla b,\nabla\partial_t\cdot\rangle-2\langle\nabla c^2,\nabla\cdot\rangle+(q-\Delta c^2)\\
&=-\partial_t^3+\alpha\partial_t^2+\Delta_g\partial_t-\beta\Delta_g+(\Delta b-\lambda)\partial_t+\langle D\ln\hat b,D\partial_t\cdot\rangle_g\\
&\quad +\beta\langle D\ln\tilde b-2D\ln c^2,D\cdot\rangle_g+(q-\Delta c^2),
\end{split}
\end{equation}
where $\hat b=b^{1+\frac n2}$ and $\tilde b=b^{1-\frac n2}$. We proceed to construct Gaussian beams for the equation $L^*_{\alpha,b,c,\lambda,q}y=0$ that will be used in the next section.

Noting that the principal part of $L_{\alpha,b,c,\lambda,q}$ and $-L^*_{\alpha,b,c,\lambda,q}$ coincides, following the arguments of constructing Gaussian beams for $L_{\alpha,b,c,\lambda,q}v=0$, we can construct Gaussian beam solutions for the adjoint equation $L^*_{\alpha,b,c,\lambda,q}y=0$ in a small tubular region around a null geodesic $\vartheta$ of the form $$y=e^{i\sigma\phi(x,t)}\eta_\sigma(x,t)+R_\sigma(x,t).$$
More precisely, for sufficiently regular functions $\phi$ and $\eta_\sigma$, we can compute
\begin{equation}
\begin{split}
-e^{-i\sigma\phi}L^*_{\alpha,b,c,\lambda,q}(e^{i\sigma\phi}\eta_\sigma)&=i\sigma^3\phi_t\eta_\sigma(\mathcal S\phi)+\sigma^2\widetilde{\mathcal T}_1(\eta_\sigma,\phi)\\
&\quad +i\sigma\widetilde{\mathcal T}_2(\eta_\sigma,\phi)-L^*_{\alpha,b,c,\lambda,q}\eta_\sigma,
\end{split}
\end{equation}
where $\mathcal S\phi=\langle d\phi,d\phi\rangle_{\bar g}$, and 
\begin{equation}
\begin{split}
\widetilde{\mathcal T}_1(\eta_\sigma,\phi)&=2\langle d\eta_\sigma,d\phi\rangle_{\bar g}\phi_t+(\mathcal S\phi)\eta_{\sigma t}+2\langle d\phi,d\phi_t\rangle_{\bar g}\eta_\sigma\\
&\quad +[\square_{\bar g}\phi+\langle d\phi,d\ln\hat b\rangle_{\bar g}](\eta_\sigma\phi_t)+[\Upsilon\phi_t^2-(\mathcal S\phi)\beta]\eta_\sigma,
\end{split}
\end{equation}
\begin{equation}
\begin{split}
\widetilde{\mathcal T}_2(\eta_\sigma,\phi)
&=3\phi_t\eta_{\sigma tt}+3\phi_{tt}\eta_{\sigma t}+\phi_{ttt}\eta_\sigma-\alpha(2\phi_t\eta_{\sigma t}+\phi_{tt}\eta_\sigma)
+(\lambda-\Delta b)\phi_t\eta_\sigma\\
&\quad-\Delta_g(\phi_t\eta_\sigma)
-\langle D\ln\hat b,D(\phi_t\eta_\sigma)\rangle_g-2\langle D\phi,D\eta_{\sigma t}-\beta D\eta_\sigma\rangle_g\\
&\quad-\bigl[\Delta_g\phi+\langle D\ln\hat b,D\phi\rangle_g\bigr]
(\eta_{\sigma t}-\beta\eta_\sigma)
+2\eta_\sigma\langle D\beta,D\phi\rangle_g.
\end{split}
\end{equation}
The eikonal equation for $\phi$ is the same as that for $\varphi$, and thus we can take $\phi=\varphi$. 
For the amplitude function $\eta_\sigma$, we can take 
\begin{equation}
\eta_\sigma(s,z')=\sum_{k=0}^N\sigma^{-k}\chi\bigl(\frac{|z'|}{\delta}\bigr)\theta_k(s,z'),
\end{equation}
where $\chi(\tau)$ is the cut-off function given in \eqref{eq:with-cut-off}, and for each $k\in\{0,1,\cdots,N\}$,
\begin{equation}
\theta_k(s,z')=\sum_{j=0}^N\theta_{k,j}(s,z'),\qquad \theta_0^1=\theta_0,\quad\theta_{0,0}^1=\theta_{0,0}.
\end{equation}
The coefficients satisfy
\[
\begin{aligned}
\partial_{z'}^\mu\widetilde{\mathcal T}_1(\theta_0,\phi)(s,0)&=0,\\
\partial_{z'}^\mu\bigl[\widetilde{\mathcal T}_1(\theta_1,\phi)
+i\widetilde{\mathcal T}_2(\theta_0,\phi)\bigr](s,0)&=0,\\
\partial_{z'}^\mu\bigl[\widetilde{\mathcal T}_1(\theta_k,\phi)
+i\widetilde{\mathcal T}_2(\theta_{k-1},\phi)
-L^*_{\alpha,b,c,\lambda,q}\theta_{k-2}\bigr](s,0)&=0,
\quad k=2,\ldots,N,
\end{aligned}
\]
for $|\mu|\le N$. Similarly, we can solve the transport equation
\begin{equation}
\widetilde{\mathcal T}_1(\theta^1_{0,0},\phi)(s,0\cdots,0)=0,\quad s\in [a_0,b_0]
\end{equation}
on $\vartheta$, to obtain the first principal term of $\theta^1_0$
\begin{equation}
\theta^1_{0,0}(s)=\hat b^{-\frac12}(s)(\det Y(s))^{-\frac12}e^{\frac{\sqrt 2}{4}\int_{a_0}^s\Upsilon(\tau,0)d\tau},\quad s\in [a_0,b_0].
\end{equation}
We note that $\theta_{0,0}^1(s)$ has a slight difference from $e_{0,0}^1(s)$. The terms $\theta_{k,j}$ and thus $\theta_k$ can be constructed in a similar way.
Moreover, $y_\sigma=e^{i\sigma\phi}\eta_\sigma$ and the remainder $R_\sigma$ of $y$ enjoy the estimate \eqref{est:remainder-r-sigma}. Indeed, applying the same remainder calculation to $-L_{\alpha,b,c,\lambda,q}^*$  gives
\begin{equation}
\begin{cases}
L_{\alpha,b,c,\lambda,q}^*R_\sigma=-L^*y_\sigma,\\
R_\sigma|_{\Gamma_T}=0,\\
R_\sigma(T)=\partial_tR_\sigma(T)=\partial_t^2R_\sigma(T)=0.
\end{cases}
\end{equation}
We also prove that the remainder term satisfies
\begin{equation}
\|R_\sigma\|_{H^{k+1}(\Omega_T)}\lesssim\|L^*y_\sigma\|_{H^k(\Omega_T)}\lesssim\sigma^{-K}.
\end{equation}



\section{Linearization method and Alessandrini type integral identities}\label{sec:lin}
In this section, we use the classical linearization method to analyze the asymptotic expansion of solutions with respect to the small Dirichlet data. Some useful Alessandrini type integral identities for $\alpha,\lambda,q$ and $\xi$ will be also given by introducing the adjoint problem of $L_{\alpha,b,c,\lambda,q}u=0$ together with the simple integration by parts formula.
\subsection{Linearization of the nonlinear problem}
Let $f_1,f_2\in\mathcal A_0^s(\Gamma_T)$ and let
$\varepsilon_1,\varepsilon_2\in\mathbb R\setminus\{0\}$. We use
\begin{equation}
D_\varepsilon u_{\varepsilon_j f_j}
\vcentcolon=\frac{u_{\varepsilon_jf_j}}{\varepsilon_j},
\qquad D^2_{\varepsilon_1,\varepsilon_2}u
\vcentcolon=
\frac{
u_{\varepsilon_1f_1+\varepsilon_2f_2}
-u_{\varepsilon_1f_1}
-u_{\varepsilon_2f_2}
}{\varepsilon_1\varepsilon_2}
\end{equation}
to denote the first- and second-order finite difference of $u$, respectively.
Set
\[
f_\varepsilon
\vcentcolon=\varepsilon_1f_1+\varepsilon_2f_2,
\qquad
\rho_\varepsilon
\vcentcolon=
|\varepsilon_1|\|f_1\|_{H^{s+2}(\Gamma_T)}
+|\varepsilon_2|\|f_2\|_{H^{s+2}(\Gamma_T)}.
\]
For the general definition of higher order finite difference operator and its properties, we refer to \cite[Appendix C]{lassas2022uniqueness} (see also \cite{lai2024partial,lassas2025stability}) for more details.
We next apply the linearization method to analyze the asymptotic expansion of solutions to the following nonlinear problem
\begin{equation}\label{eqn:non-sys-vare-f}
\begin{cases}
L_{\alpha,b,c,\lambda,q}u=\partial_t^2(\xi u^2)  & {\rm in}\,\, \Omega_T,\\
u=f_\varepsilon\vcentcolon=\varepsilon_1f_1+\varepsilon_2f_2  & {\rm on}\,\, \Gamma_T,\\
u(0)=\partial_tu(0)=\partial_t^2u(0)=0 & {\rm in}\,\, \Omega.
\end{cases}
\end{equation}
\begin{lemma}\label{lem:asy-nonlin-u}
Let $0<\epsilon_0\le\delta_{\rm wp}$, where $\delta_{\rm wp}$ is the constant in
Proposition~\ref{prop:Westervelt} for $m=s$. Assume
$\rho_\varepsilon\le\epsilon_0$ and
$\alpha,b,c,\lambda,q,\xi\in\mathcal U_K(M_0)$, $b,c\ge m_0>0$, with some sufficiently large $K>0$. Then the solution
$u_{\varepsilon f}\in E^{s+2}(\Omega_T)$ of
\eqref{eqn:non-sys-vare-f} has the expansion
\begin{equation}
u_{\varepsilon f}=\varepsilon_1v_1+\varepsilon_2v_2+Q=\varepsilon_1v_1+\varepsilon_2v_2+(\varepsilon_1^2w_{20}+2\varepsilon_1\varepsilon_2w_{11}+\varepsilon_2^2w_{02})+R,
\end{equation}
where for each $j=1,2$, $v_j$ solves the linear problem
\begin{equation}\label{eqn:linear-1st-v}
\begin{cases}
L_{\alpha,b,c,\lambda,q}v_j=0  & {\rm in}\,\, \Omega_T,\\
v_j=f_j  & {\rm on}\,\, \Gamma_T,\\
v_j(0)=\partial_tv_j(0)=\partial_t^2v_j(0)=0 & {\rm in}\,\, \Omega,
\end{cases}
\end{equation}
and for $k,l\in \{0,1,2\}$ satisfying $k+l=2,$ $w_{kl}$ is the solution to
\begin{equation}\label{eqn:linear-w-v1v2}
\begin{cases}
L_{\alpha,b,c,\lambda,q}w_{kl}=\partial_t^2(\xi v_1^kv_2^l)  & {\rm in}\,\, \Omega_T,\\
w_{kl}=0  & {\rm on}\,\, \Gamma_T,\\
w_{kl}(0)=\partial_tw_{kl}(0)=\partial_t^2w_{kl}(0)=0 & {\rm in}\,\, \Omega.
\end{cases}
\end{equation}
Moreover, the remainder terms satisfy
\begin{equation}\label{est:first-exp-Q}
\|Q\|_{E^{s+2}(\Omega_T)}
+\|\partial_\nu Q\|_{\mathcal N_{s+1}(\Gamma_T)}
\lesssim\rho_\varepsilon^2,
\end{equation}
\begin{equation}\label{est:second-R}
\|R\|_{E^{s+2}(\Omega_T)}
+\|\partial_\nu R\|_{\mathcal N_{s+1}(\Gamma_T)}
\lesssim\rho_\varepsilon^3.
\end{equation}
\end{lemma}
\begin{proof}
Since $f_1,f_2\in\mathcal A_0^s(\Gamma_T)$, we have $f_\varepsilon\in\mathcal A_0^s(\Gamma_T)$ satisfying $\|f_\varepsilon\|_{H^{s+2}(\Gamma_T)}
\le \rho_\varepsilon$. Thus, Proposition
\ref{prop:Westervelt}, with $m=s$, gives
\begin{equation}
\|u_{\varepsilon f}\|_{E^{s+2}(\Omega_T)}
\lesssim\rho_\varepsilon.
\end{equation}
The fixed-point construction in Proposition~\ref{prop:Westervelt}
shows that all initial jets of $u_{\varepsilon f}$ through order $s+2$ vanish and the same is true for $v_1$ and $v_2$.
Let $Q=u_{\varepsilon f}-\varepsilon_1v_1-\varepsilon_2v_2$. Noting that $v_j$ satisfy the linear equation \eqref{eqn:linear-1st-v}, we see that $Q$ satisfies
\begin{equation}
\begin{cases}
L_{\alpha,b,c,\lambda,q}Q=\partial_t^2(\xi u^2_{\varepsilon f}) & {\rm in}\,\, \Omega_T,\\
Q=0 & {\rm on}\,\, \Gamma_T,\\
Q(0)=\partial_tQ(0)=\partial_t^2Q(0)=0  & {\rm in}\,\, \Omega.
\end{cases}
\end{equation}
Consequently, $\partial_t^2(\xi u_{\varepsilon f}^2)\in E^s(\Omega_T)$
and its time derivatives through order $s-1$ vanish at $t=0$. Thus the
source and the zero boundary data satisfy the full compatibility
conditions. Applying the linear estimate \eqref{eq:est-high} yields
\begin{equation}
\begin{split}
\|Q\|_{E^{s+2}(\Omega_T)}
+\|\partial_\nu Q\|_{\mathcal N_{s+1}(\Gamma_T)}
&\lesssim
\|\partial_t^2(\xi u^2_{\varepsilon f})\|_{E^s(\Omega_T)}\\
&\lesssim
\|u_{\varepsilon f}\|^2_{E^{s+2}(\Omega_T)}
\lesssim\rho_\varepsilon^2.
\end{split}
\end{equation}

By equations \eqref{eqn:linear-1st-v} and \eqref{eqn:linear-w-v1v2}, we find that $R\vcentcolon=u_{\varepsilon f}-(\varepsilon_1v_1+\varepsilon_2v_2+\varepsilon_1^2w_{20}+2\varepsilon_1\varepsilon_2w_{11}+\varepsilon_2^2w_{02})$ satisfies
\begin{equation}
\begin{cases}
L_{\alpha,b,c,\lambda,q}R=\xi\partial_t^2[u_{\varepsilon f}^2-(\varepsilon_1v_1+\varepsilon_2v_2)^2] & {\rm in}\,\, \Omega_T,\\
R=0 & {\rm on}\,\,\Gamma_T,\\
R(0)=\partial_tR(0)=\partial_t^2R(0)=0 & {\rm in}\,\, \Omega.
\end{cases}
\end{equation}
Recalling that $u_{\varepsilon f},v_1,v_2\in E^{s+2}(\Omega_T)$ and $Q=u_{\varepsilon f}-\varepsilon_1v_1-\varepsilon_2v_2$, using the estimate \eqref{est:first-exp-Q}, we have 
\begin{equation}
\begin{split}
\|\xi\partial_t^2[u_{\varepsilon f}^2-(\varepsilon_1v_1+\varepsilon_2v_2)^2]\|_{E^s(\Omega_T)}\lesssim \|u_{\varepsilon f}+\varepsilon_1v_1+\varepsilon_2v_2\|_{E^{s+2}(\Omega_T)}\|Q\|_{E^{s+2}(\Omega_T)}\lesssim \rho_\varepsilon^3.
\end{split}
\end{equation}
The source in the $R$-equation has the same vanishing initial conditions, so its
zero boundary data satisfy the compatibility conditions as well. Hence Proposition \ref{prop:wellposed} now gives
$R\in E^{s+2}(\Omega_T)$ and the estimate \eqref{est:second-R}.
\end{proof}

\begin{remark}\label{rem:say-u-12}
In particular, if
$\varepsilon_1=\varepsilon_2=\varepsilon$ and $f_1=f_2=f$, then the uniqueness of solutions implies $v_1=v_2=v$ and $w_{20}=w_{11}=w_{02}=w$. Moreover, the expansion reads $u_{2\varepsilon f}=2\varepsilon v+4\varepsilon^2w+R$.

Apply Lemma~\ref{lem:asy-nonlin-u} with
$\varepsilon=\varepsilon_1$, $f_1=f$, and $f_2=0$, while keeping
$\varepsilon_2\ne0$. Then $v_2=w_{11}=w_{02}=0$, and we obtain
$u_{\varepsilon f}=\varepsilon v+\widetilde Q$, where
$\widetilde Q$ satisfies \eqref{est:first-exp-Q}. Hence 
\begin{equation}
D_{\varepsilon}u_{\varepsilon f}=v+\varepsilon^{-1}\widetilde Q,
\end{equation}
implying that
\begin{equation}
D_{\varepsilon}L_{\alpha,b,c,\lambda,q}u_{\varepsilon f}=L_{\alpha,b,c,\lambda,q}v+D_{\varepsilon}L_{\alpha,b,c,\lambda,q}\widetilde Q.
\end{equation}
Applying the one-parameter expansion separately in the directions
$f_1$ and $f_2$ gives
\begin{equation}
u_{\varepsilon_1f_1}
=\varepsilon_1v_1+\varepsilon_1^2w_{20}+R_1,
\qquad
u_{\varepsilon_2f_2}
=\varepsilon_2v_2+\varepsilon_2^2w_{02}+R_2.
\end{equation}
where $R_j$ is the remainder term of $u_{\varepsilon_jf_j},$ and $u_{\varepsilon_jf_j}$ is the unique solution to \eqref{eqn:non-sys-vare-f} with Dirichlet boundary data $\varepsilon_jf_j$ for each $j=1,2.$
Thus, we can rewrite
\begin{equation}
u_{\varepsilon_1f_1+\varepsilon_2f_2}
=u_{\varepsilon_1f_1}+u_{\varepsilon_2f_2}
+2\varepsilon_1\varepsilon_2w_{11}
+R-R_1-R_2.
\end{equation}
Let $\tilde R=R-R_1-R_2=\varepsilon_1\varepsilon_2D^2_{\varepsilon_1,\varepsilon_2}R$. Using the definition of the second order finite difference operator $D^2_{\varepsilon_1,\varepsilon_2}$ and the equation \eqref{eqn:linear-w-v1v2}, we have
\begin{equation}
D^2_{\varepsilon_1,\varepsilon_2}u
=2w_{11}+D^2_{\varepsilon_1,\varepsilon_2}R,
\end{equation}
yielding that
\begin{equation}\label{eqnpr}
\begin{split}
D^2_{\varepsilon_1,\varepsilon_2}
L_{\alpha,b,c,\lambda,q}u
&=2L_{\alpha,b,c,\lambda,q}w_{11}
+D^2_{\varepsilon_1,\varepsilon_2}
L_{\alpha,b,c,\lambda,q}R\\
&=2\partial_t^2(\xi v_1v_2)
+D^2_{\varepsilon_1,\varepsilon_2}
L_{\alpha,b,c,\lambda,q}R.
\end{split}
\end{equation}	
\end{remark}

\subsection{Alessandrini type identities}
For any sufficiently regular $z$, set
\begin{equation}\label{def:boundary-flux}
\mathcal Bz
\vcentcolon=
\left.(b\partial_\nu\partial_tz+c^2\partial_\nu z)\right|_{\Gamma_T}.
\end{equation}
For $j=1,2$, we consider the nonlinear problems related to the coefficients $\alpha_j,\lambda_j,q_j$ and $\xi_j$ with boundary Dirichlet data $\varepsilon f$, assuming $b$ and $c$ are known a priori. That is,
\begin{equation}
\begin{cases}
L_{\alpha_j,b,c,\lambda_j,q_j}u_j=\partial_t^2(\xi_j u_j^2) & {\rm in}\,\, \Omega_T,\\
u_j=\varepsilon f & {\rm on}\,\, \Gamma_T,\\
u_j(0)=\partial_tu_j(0)=\partial_t^2u_j(0)=0 & {\rm in}\,\, \Omega,
\end{cases}
\end{equation}
By Lemma \ref{lem:asy-nonlin-u} and Remark \ref{rem:say-u-12}, $u_j=u_{j,\varepsilon f}$ have the asymptotic expansion 
\begin{equation}
u_j=\varepsilon v_j+\widetilde Q_j,\quad j=1,2,
\end{equation}
where $v_j$ solve the equation $L_{\alpha_j,b,c,\lambda_j,q_j}v_j=0$ with the same Dirichlet data $f$, and the remainder terms $\widetilde Q_j$ satisfying the estimate \eqref{est:first-exp-Q} for $j=1,2$.
Let $v=v_1-v_2$, $\tilde\alpha=\alpha_2-\alpha_1$, $\tilde\lambda=\lambda_2-\lambda_1$, $\tilde q=q_2-q_1$.    
We find that
\begin{equation}\label{eqn:linear-1st-v1-v2}
\begin{cases}
L_{\alpha_1,b,c,\lambda_1,q_1}v=\tilde\alpha\partial_t^2v_2+\tilde\lambda\partial_tv_2+\tilde qv_2 & {\rm in}\,\, \Omega_T,\\
v=0 & {\rm on}\,\, \Gamma_T,\\
v(0)=\partial_tv(0)=\partial_t^2v(0)=0 & {\rm in}\,\, \Omega,
\end{cases}
\end{equation}
Since the initial jets of $v_2$ vanish through order $s+1$, the source in \eqref{eqn:linear-1st-v1-v2} belongs to $E^s(\Omega_T)$ and compatibility conditions hold.

For $h\in\mathcal A_T^s(\Gamma_T)$, we consider the adjoint problem
\begin{equation}\label{eqns:sys-backward-y}      
\begin{cases}
L_{\alpha_1,b,c,\lambda_1,q_1}^*y=0 & {\rm in}\,\, \Omega_T,\\
y=h & {\rm on}\,\, \Gamma_T,\\
y(T)=\partial_ty(T)=\partial_t^2y(T)=0 & {\rm in}\,\,\Omega,
\end{cases}
\end{equation}
Multiplying the first equation of \eqref{eqn:linear-1st-v1-v2} by $y$ and noting the initial and boundary conditions of the problems \eqref{eqn:linear-1st-v} and \eqref{eqns:sys-backward-y}, we can obtain the  Alessandrini type equality for $\tilde\alpha,\tilde\lambda,\tilde q$,
\begin{equation}\label{integral-identity-ab}
\begin{split}
I_{\tilde\alpha,\tilde\lambda,\tilde q}&\vcentcolon=\int_{\Omega_T}[-\tilde \alpha(\partial_tv_2)(\partial_ty)+\tilde\lambda(\partial_tv_2)y+\tilde qv_2y]dxdt\\
&=-\varepsilon^{-1}\int_{\Gamma_T}
[\Lambda_1(\varepsilon f)-\Lambda_2(\varepsilon f)]h\,d\Gamma dt+\varepsilon^{-1}\int_{\Gamma_T}
\mathcal B(\tilde Q_1-\tilde Q_2)h\,d\Gamma dt.
\end{split}
\end{equation}
Consequently,
\begin{equation}
\begin{split}
|I_{\tilde\alpha,\tilde\lambda,\tilde q}|&\lesssim |\varepsilon|^{-1}\bigl( \|\Lambda_1(\varepsilon f)-\Lambda_2(\varepsilon f)\|_{L^2(\Gamma_T)}+
\|\partial_\nu(\tilde Q_1-\tilde Q_2)\|_{H^1(0,T;L^2(\Gamma))}\bigr)
\|h\|_{L^2(\Gamma_T)}.
\end{split}
\end{equation}

We now consider the following nonlinear problem for $u_j=u_{j,\varepsilon_1f_1+\varepsilon_2f_2}$ with the Dirichlet data $\varepsilon_1f_1+\varepsilon_2f_2$ and the coefficients $\alpha_j,\lambda_j,q_j$ and $\xi_j$ for $j=1,2$,
\begin{equation}
\begin{cases}
L_{\alpha_j,b,c,\lambda_j,q_j}u_j=\partial_t^2(\xi_j u_j^2) & {\rm in}\,\, \Omega_T,\\
u_j=\varepsilon_1f_1+\varepsilon_2f_2 & {\rm on}\,\, \Gamma_T,\\
u_j(0)=\partial_tu_j(0)=\partial_t^2u_j(0)=0 & {\rm in}\,\, \Omega.
\end{cases}
\end{equation}
By Lemma \ref{lem:asy-nonlin-u}, we should study the problem \eqref{eqn:linear-w-v1v2} for $w_{j,11}$ where the coefficients $\alpha,\lambda,q,\xi$ are replaced by $\alpha_j,\lambda_j,q_j,\xi_j$ for $j=1,2$. We briefly write $w_j=w_{j,11}$ to have
\begin{equation}\label{sys:for-w-j}
\begin{cases}
L_{\alpha_j,b,c,\lambda_j,q_j}w_j=\partial_t^2(\xi_j v_{1,j}v_{2,j}) & {\rm in}\,\, \Omega_T,\\
w_j=0 & {\rm on}\,\, \Gamma_T,\\
w_j(0)=\partial_tw_j(0)=\partial_t^2w_j(0)=0 & {\rm in}\,\, \Omega,
\end{cases}
\end{equation}
where $v_{k,j}$ satisfy the equation $L_{\alpha_j,b,c,\lambda_j,q_j}v_{k,j}=0$ with respect to the Dirichlet data $f_k$ and $\alpha_j,\lambda_j,q_j$ for $j,k=1,2$. For simplicity, we choose $f=f_1=f_2$ in the subsequent analysis. Then by the uniqueness of solutions to the $v_{k,j}$-equation, we have $v_{1,j}=v_{2,j}\vcentcolon=v_j$ in $\Omega_T$ for $j=1,2$.
Let $w=w_1-w_2$, $\tilde\xi=\xi_2-\xi_1$. We have
\begin{equation}\label{eqn:sys-w-w1-w2}
\begin{cases}
L_{\alpha_1,b,c,\lambda_1,q_1}w=F(\tilde\alpha,\tilde\lambda,\tilde q,\tilde\xi) & {\rm in}\,\, \Omega_T,\\
w=0 & {\rm on}\,\, \Gamma_T,\\
w(0)=\partial_tw(0)=\partial_t^2w(0)=0 & {\rm in}\,\, \Omega,
\end{cases}
\end{equation}
where  
\begin{equation}
F(\tilde\alpha,\tilde\lambda,\tilde q,\tilde\xi)=\tilde\alpha\partial_t^2w_2+\tilde\lambda\partial_tw_2+\tilde qw_2-\tilde\xi\partial_t^2v_2^2+\xi_1\partial_t^2(v(v_1+v_2)).
\end{equation}
All solution factors in this expression have vanishing initial jets through the
orders supplied by Proposition~\ref{prop:wellposed}. Hence
$F(\widetilde\alpha,\widetilde\lambda,\widetilde q,\widetilde\xi)$
belongs to $E^s(\Omega_T)$ and satisfies the full zero-boundary
compatibility conditions for \eqref{eqn:sys-w-w1-w2}.

Let $\widetilde R_j$ denote the mixed remainder
$R_j(\varepsilon_1,\varepsilon_2)-R_j(\varepsilon_1,0)
-R_j(0,\varepsilon_2)$. Define
\[
D^2_{\varepsilon_1,\varepsilon_2}\Lambda_j\vcentcolon=\frac{\Lambda_j(\varepsilon_1f+\varepsilon_2f)-\Lambda_j(\varepsilon_1f)
-\Lambda_j(\varepsilon_2f)}{\varepsilon_1\varepsilon_2}.
\]
The expansion above gives the exact boundary relation
\begin{equation}\label{eq:boundary-second-difference}
\mathcal Bw=\mathcal B(w_1-w_2)=\frac12D^2_{\varepsilon_1,\varepsilon_2}(\Lambda_1-\Lambda_2)-\frac{1}{2\varepsilon_1\varepsilon_2}
\mathcal B(\widetilde R_1-\widetilde R_2).
\end{equation}
Multiplying the first equation of
\eqref{eqn:sys-w-w1-w2} by $y$ and integrating by parts gives
\begin{equation}\label{iden:xi-DtN}
\begin{split}
&2\int_{\Omega_T}\tilde\xi(\partial_ty)(\partial_tv_2)v_2\,dxdt\\
&=-\int_{\Gamma_T}\mathcal B(w)y\,d\Gamma dt
+\int_{\Omega_T}(\tilde\alpha\partial_tw_2\partial_ty-\tilde\lambda y\partial_tw_2-\tilde qw_2y)\,dxdt\\
&\quad-\int_{\Omega_T}y\xi_1[(v_1+v_2)\partial_t^2v+2(\partial_tv)\partial_t(v_1+v_2)+v\partial_t^2(v_1+v_2)]\,dxdt.
\end{split}
\end{equation}


\section{Proof of the main theorems}\label{sec-proof}
In this section, we give the proof of the main theorems.
\subsection{Stable determination of the linear coefficients $\alpha,\lambda$ and $q$}\label{Sec-proof-lin}

Since all coefficients are real-valued, the equation $L^*_{\alpha_1,b,c,\lambda_1,q_1}y=0$ implies
$L^*_{\alpha_1,b,c,\lambda_1,q_1}\bar y=0$. For a maximal null geodesic $\vartheta$ in $\mathcal M$, we respectively
construct Gaussian beam solutions for $L_{\alpha_2,b,c,\lambda_2,q_2}v_2=0$ and $L^*_{\alpha_1,b,c,\lambda_1,q_1}y=0$:
\begin{equation}
v_2=e^{i\sigma\varphi}a_\sigma+r_\sigma,
\qquad
f=e^{i\sigma\varphi}a_\sigma|_{\Gamma_T},
\end{equation}
and
\begin{equation}
y=e^{-i\sigma\bar\varphi}\bar\eta_\sigma+R_\sigma,
\qquad
h=e^{-i\sigma\bar\varphi}\bar\eta_\sigma|_{\Gamma_T}.
\end{equation}

Let $s_-=\sqrt2(t_--t_0)$ and $s_+=\sqrt2(t_+-t_0)$ be the Fermi-coordinate parameters of the entry and exit points of the
null geodesic through $\Omega_T$. The leading transport solutions are only determined up to nonzero multiplicative constants. We therefore normalize the leading amplitudes at $s=s_-$. Equivalently, in the formulas of Section~\ref{sec:Gaussian-beams}, the lower integration limit $a_0$ may be replaced by $s_-$. This merely
multiplies the Gaussian beams by uniformly bounded constants and does not influence any of their estimates. With this normalization, along the null geodesic,
\begin{equation}\label{eq:principal-amplitude-normalized}
\begin{split}
e_{0,0}^1(s)&=\widetilde b^{1/2}(s)(\det Y(s))^{-1/2}e^{-\frac{\sqrt2}{4}\int_{s_-}^{s}\Upsilon_2(\tau,0)\,d\tau},\\
\overline\theta_{0,0}^1(s)
&=\hat b^{-1/2}(s)(\overline{\det Y(s)})^{-1/2}e^{\frac{\sqrt2}{4}\int_{s_-}^{s}\Upsilon_1(\tau,0)\,d\tau},
\end{split}
\end{equation}
where $\widetilde b=b^{1-\frac n2}, \hat b=b^{1+\frac n2}.$
Since
$\Upsilon_2-\Upsilon_1=\alpha_2-\alpha_1=\tilde\alpha$
in $\Omega$, we have
\begin{equation}\label{eq:product-principal-amplitudes}
b^{n/2}e_{0,0}^1(s)\overline{\theta_{0,0}^1}(s)=|\det Y(s)|^{-1}e^{-\frac{\sqrt2}{4}\int_{s_-}^{s}
\tilde\alpha(\tau,0)\,d\tau}.
\end{equation}
Here we used $b^{n/2}\widetilde b^{1/2}\hat b^{-1/2}=1$.

We now insert the Gaussian beam solutions into \eqref{integral-identity-ab} to have
\begin{equation}\label{iden-alpha-lam-q-DtN}
\begin{split}
I_{\tilde\alpha,\tilde\lambda,\tilde q}\vcentcolon&=
I_{\tilde\alpha}+I_{\tilde\lambda}+I_{\tilde q}\\
&=-\int_{\Omega_T}\tilde\alpha\bigl[\bigl(\sigma^2|\varphi_t|^2a_\sigma\bar\eta_\sigma+i\sigma(\varphi_ta_\sigma\bar\eta_{\sigma t} -\bar\varphi_ta_{\sigma t}\bar\eta_\sigma)+a_{\sigma t}\bar\eta_{\sigma t}
\bigr)e^{-2\sigma\Im\varphi}\\
&\qquad\qquad\quad
+(i\sigma\varphi_ta_\sigma+a_{\sigma t})R_{\sigma t}e^{i\sigma\varphi}
+(-i\sigma\bar\varphi_t\bar\eta_\sigma+\bar\eta_{\sigma t})
r_{\sigma t}e^{-i\sigma\bar\varphi}+r_{\sigma t}R_{\sigma t}\bigr]dxdt\\
&\quad+\int_{\Omega_T}\tilde\lambda\bigl[(i\sigma\varphi_ta_\sigma\bar\eta_\sigma+a_{\sigma t}\bar\eta_\sigma)e^{-2\sigma\Im\varphi}+(i\sigma\varphi_ta_\sigma+a_{\sigma t})R_\sigma e^{i\sigma\varphi}\\
&\qquad\qquad\quad
+\bar\eta_\sigma r_{\sigma t}e^{-i\sigma\bar\varphi}+r_{\sigma t}R_\sigma\bigr]dxdt\\
&\quad+\int_{\Omega_T}\tilde q\bigl[a_\sigma\bar\eta_\sigma e^{-2\sigma\Im\varphi}+a_\sigma R_\sigma e^{i\sigma\varphi}
+\bar\eta_\sigma r_\sigma e^{-i\sigma\bar\varphi}+r_\sigma R_\sigma\bigr]dxdt.
\end{split}
\end{equation}
Let $e_0^1(s,z')$ and $\theta_0^1(s,z')$ denote the full leading amplitudes appearing in $a_\sigma$ and $\eta_\sigma$, and put
\begin{equation}\label{eq:full-leading-product-6-1}
A_0(s,z')\vcentcolon=e_0^1(s,z')\overline{\theta_0^1}(s,z').
\end{equation}
In particular, on $\vartheta$,
\[
A_0(s,0)=e_{0,0}^1(s)\overline{\theta_{0,0}^1}(s).
\]
Taking $N$, and hence $K$, sufficiently large, the estimate \eqref{est:remainder-r-sigma}, together with the corresponding estimate for $R_\sigma$ and $r_\sigma$, give
\begin{equation}\label{eq:remainder-orders-6-1}
\begin{split}
I_{\tilde\alpha}&=-\sigma^2\int_{\Omega_T}\tilde\alpha|\varphi_t|^2A_0(s,z')
e^{-2\sigma\Im\varphi}\,dxdt+\mathcal O(\sigma^{1-\frac n2}),\\
I_{\tilde\lambda}&=i\sigma\int_{\Omega_T}\tilde\lambda\varphi_tA_0(s,z')
e^{-2\sigma\Im\varphi}\,dxdt+\mathcal O(\sigma^{-\frac n2}),\\
I_{\tilde q}&=\int_{\Omega_T}\tilde qA_0(s,z')e^{-2\sigma\Im\varphi}\,dxdt+\mathcal O(\sigma^{-1-\frac n2}).
\end{split}
\end{equation}
Before using the estimates of the geodesic ray globally, we are required to make the dependence on the
chosen geodesic explicit. 

For $(p_0,\zeta)\in\partial_-S\Omega$, let
$\gamma_{p_0,\zeta}$ be the corresponding unit-speed maximal geodesic
and let $\tau(p_0,\zeta)$ be its exit time. We write
\[
p_+=\gamma_{p_0,\zeta}(\tau(p_0,\zeta)),\qquad\zeta_+=\dot\gamma_{p_0,\zeta}(\tau(p_0,\zeta)),
\]
so that $(p_+,\zeta_+)\in\partial_+S\Omega$.
Set
\begin{equation}\label{eq:incidence-angles}
\iota_-(p_0,\zeta)\vcentcolon=-\langle\zeta,\nu(p_0)\rangle_g,\qquad \iota_+(p_0,\zeta)\vcentcolon=
\langle\zeta_+,\nu(p_+)\rangle_g.
\end{equation}
For $0<\kappa<1$, define
\begin{equation}\label{eq:non-grazing-rays}
\mathcal M_\kappa\vcentcolon=\left\{(p_0,\zeta)\in\partial_-S\Omega:\iota_-(p_0,\zeta)\ge\kappa,\quad
\iota_+(p_0,\zeta)\ge\kappa\right\},
\end{equation}
and let
\[
\mathcal G_\kappa\vcentcolon=\partial_-S\Omega\setminus\mathcal M_\kappa.
\]

For each maximal geodesic $\gamma=\gamma_{p_0,\zeta}, (p_0,\zeta)\in\partial_-S\Omega,$
let $\vartheta_\gamma$ denote the associated null geodesic in the
extended spacetime. To emphasize the dependence on $\gamma$, we write $(U_\gamma,\Phi_\gamma)$
for the Fermi coordinate neighborhood associated with
$\vartheta_\gamma$ in Lemma~\ref{lem:Fermi-coord}. Thus
\[
\Phi_\gamma: U_\gamma\rightarrow \Phi_\gamma(U_\gamma)\subset\mathbb R^{1+n}
\]
is the corresponding Fermi coordinate map, and $(s,z')=\Phi_\gamma(t,x).$
We denote its inverse by $\mathcal F_\gamma\vcentcolon=\Phi_\gamma^{-1},$
so that
\[
\mathcal F_\gamma(s,z')=(t(s,z'),x(s,z')).
\]
In particular, along the central null geodesic,
\begin{equation}\label{eq:Fermi-center-gamma}
\Phi_\gamma(\vartheta_\gamma(t))
=
\bigl(\sqrt2(t-t_{0,\gamma}),0\bigr),
\qquad
t=t_{0,\gamma}+\frac{s-z_1}{\sqrt2}.
\end{equation}
Let $[s_{-,\gamma},s_{+,\gamma}]\vcentcolon=[\sqrt2(t_{-,\gamma}-t_{0,\gamma}),\sqrt2(t_{+,\gamma}-t_{0,\gamma})]$ be the longitudinal Fermi interval corresponding to the segment of $\vartheta_\gamma$ in $\overline\Omega_T$. We denote by $J_\gamma$ a prolonged longitudinal interval containing $[s_{-,\gamma},s_{+,\gamma}]$, obtained by extending the
spatial geodesic beyond both endpoints. Thus, in the notation of Section~\ref{sec:Gaussian-beams}, $J_\gamma$ is the
geodesic-dependent counterpart of the interval $[a_0,b_0]$.
For each null geodesic $\vartheta_\gamma$, let
$\chi_\gamma\in C_c^\infty(J_\gamma)$ be such that
\begin{equation}\label{cond-cut-off}
\chi_\gamma=1 \quad\text{on a neighborhood of the longitudinal interval corresponding to }\overline\Omega_T.
\end{equation}
The support of $\partial_s\chi_\gamma$ is chosen in the
prolonged end regions lying outside $\overline\Omega_T$.
We then replace the local Gaussian-beam ansatz by
\[
v_\sigma(s,z')=\chi_\gamma(s)\chi\Bigl(\frac{|z'|}{\delta}\Bigr)e^{i\sigma\varphi(s,z')}\sum_{k=0}^N\sigma^{-k}e_k(s,z').
\]
Since $\operatorname{supp}\partial_s\chi_\gamma\cap\overline\Omega_T=\emptyset,$ the longitudinal cut-off produces no additional residual term in $\Omega_T$.

We prove the following lemma.
\begin{lemma}\label{lem:uniform-ray-family}
Assume that $(\Omega,g)$ has strictly convex boundary and is non-trapping.
For the fixed extensions of $\Omega$, the coefficients, and the finite
approximation orders of the Gaussian beams employed in
Section~\ref{sec:Gaussian-beams}, there exist
\[
\kappa_0\in(0,1),\qquad
\ell_*>0,\qquad
\rho_*>0,\qquad
\epsilon_*>0,\qquad
\sigma_*>1,
\]
together with constants $C\ge1$, $c>0$, and $N_0\in\mathbb N$, depending
only on the \emph{a priori} data and the prescribed finite approximation
orders, such that the following assertions hold.

\medskip
\noindent
{\rm (i)}
For every $(p_0,\zeta)\in\partial_-S\Omega$, the Fermi coordinates of
Lemma~\ref{lem:Fermi-coord} may be chosen on a prolonged tubular
neighborhood of the form
\[
\mathcal T_\gamma(\rho_*)
=
\mathcal F_\gamma\bigl(J_\gamma\times B(0,\rho_*)\bigr),
\qquad
\mathcal F_\gamma=\Phi_\gamma^{-1},
\]
where the corresponding spatial geodesic is prolonged by the fixed
length $\ell_*$ beyond both physical endpoints. The tube may be chosen
so that
\begin{equation}\label{eq:uniform-tube-separation}
\inf_{(t,x)\in\mathcal T_\gamma(\rho_*)}
\min\{t,T-t\}
\ge
\epsilon_*.
\end{equation}
Moreover, there exists a cut-off function
$\chi_\gamma\in C_c^\infty(J_\gamma)$ satisfying
\eqref{cond-cut-off} and such that
\begin{equation}\label{eq:uniform-cutoff-separation}
\operatorname{dist}\!\left(
\mathcal F_\gamma\bigl(
(\operatorname{supp}\partial_s\chi_\gamma)
\times B(0,\rho_*)\bigr),
\overline\Omega_T
\right)
\ge
\epsilon_*.
\end{equation}

The tubes $\mathcal T_\gamma(\rho_*)$ admit coverings by at most $N_0$
Fermi coordinate charts of fixed longitudinal size. Let
$\mathcal J_{\gamma,j}$ denote the associated Fermi volume densities,
and let $\{\psi_{\gamma,j}\}$ be a partition of unity subordinate to
such a covering. For every derivative order $m$ occurring below,
\begin{equation}\label{eq:uniform-Fermi-bounds}
\sup_{(p_0,\zeta)\in\partial_-S\Omega}
\max_{1\le j\le N_0}
\bigl(
\|\Phi_{\gamma,j}\|_{C^m}
+
\|\Phi_{\gamma,j}^{-1}\|_{C^m}
+
\|\mathcal J_{\gamma,j}\|_{C^m}
+
\|\psi_{\gamma,j}\|_{C^m}
\bigr)
\le
C_m.
\end{equation}
Consequently, all constants arising in the estimates for the
finite-order Gaussian beams constructed in
Section~\ref{sec:Gaussian-beams} may be chosen uniformly with respect
to $(p_0,\zeta)$.

\medskip
\noindent
{\rm (ii)}
For every $0<\kappa<\kappa_0$, one has
\begin{equation}\label{eq:grazing-length}
\tau(p_0,\zeta)
\le
C\kappa,
\qquad
(p_0,\zeta)\in\mathcal G_\kappa,
\end{equation}
whereas
\begin{equation}\label{eq:good-ray-lower-length}
\tau(p_0,\zeta)
\ge
c\kappa,
\qquad
(p_0,\zeta)\in\mathcal M_\kappa.
\end{equation}
In addition,
\begin{equation}\label{eq:grazing-measure}
{\rm meas}_{\partial_-S\Omega}(\mathcal G_\kappa)
\le
C\kappa,
\end{equation}
where ${\rm meas}_{\partial_-S\Omega}$ denotes the natural Riemannian
measure on the boundary unit sphere bundle.

\medskip
\noindent
{\rm (iii)}
Let $(p_0,\zeta)\in\mathcal M_\kappa$, set
$\gamma=\gamma_{p_0,\zeta}$, and let
\[
\mathcal F_\gamma:
J_\gamma\times B(0,\rho_*)
\longrightarrow
\mathcal T_\gamma(\rho_*)
\]
be the inverse Fermi coordinate map introduced in {\rm (i)}. Let
$A_\gamma$ be any principal density arising from the leading integrals
in \eqref{eq:remainder-orders-6-1} after pullback by
$\mathcal F_\gamma$, including the Fermi volume density and the fixed
cutoffs. Its finite-order derivatives satisfy
\[
\sup_{(p_0,\zeta)\in\partial_-S\Omega}
\|A_\gamma\|_{C^m(J_\gamma\times B(0,\rho_*))}
\le
C_m.
\]
Define
\begin{equation}\label{eq:physical-G-def}
\mathfrak G_{\sigma,\gamma}
\vcentcolon=
\sigma^{n/2}
\int_{\mathcal F_\gamma^{-1}(\Omega_T)}
A_\gamma(s,z')
e^{-2\sigma\Im\varphi(s,z')}
\,ds\,dz',
\end{equation}
and
\begin{equation}\label{eq:lead-G-def}
\mathfrak G_{\gamma}^{\,\mathrm{lead}}
\vcentcolon=
\left(\frac{\pi}{2}\right)^{n/2}
\int_{s_-}^{s_+}
\frac{A_\gamma(s,0)}
{\sqrt{\det\Im H(s)}}
\,ds.
\end{equation}
If
\begin{equation}\label{eq:sigma-kappa-condition}
\sigma
\ge
\max\{\sigma_*,C\kappa^{-8}\},
\end{equation}
then
\begin{equation}\label{eq:uniform-SP-error}
\left|
\mathfrak G_{\sigma,\gamma}
-
\mathfrak G_{\gamma}^{\,\mathrm{lead}}
\right|
\le
C\bigl(
\sigma^{-1}
+
\kappa^{-1}\sigma^{-1/4}
\bigr),
\end{equation}
where $C$ is independent of $(p_0,\zeta)$, $\kappa$, and $\sigma$.
\end{lemma}
The proof is given in the appendix of this paper.

A direct consequence of this lemma is the $L^2$ uniformity needed in the subsequent analysis. Suppose that
for a coefficient $F$ and some $D>0$ the fixed-ray argument gives, for
$(p_0,\zeta)\in\mathcal M_\kappa$,
\begin{equation}\label{eq:generic-good-ray}
|\mathcal I_{p_0,\zeta}(F)|
\lesssim
\delta^{1/D}
+
\kappa^{-1}\delta^{1/(4D)}.
\end{equation}
For the grazing family, the \emph{a priori} $L^\infty$ bound and
\eqref{eq:grazing-length} give
\[
|\mathcal I_{p_0,\zeta}(F)|
\lesssim\kappa,
\qquad
(p_0,\zeta)\in\mathcal G_\kappa.
\]
Together with \eqref{eq:grazing-measure}, this yields
\begin{equation}\label{eq:generic-L2-ray-est}
\|\mathcal I(F)\|_{L^2(\partial_-S\Omega)}
\lesssim
\delta^{1/D}
+
\kappa^{-1}\delta^{1/(4D)}
+
\kappa^{3/2}.
\end{equation}
Choosing $\kappa=\delta^{1/(10D)}$, we can obtain
\begin{equation}\label{eq:generic-L2-ray-holder}
\|\mathcal I(F)\|_{L^2(\partial_-S\Omega)}\lesssim \delta^{\frac{3}{20D}}.
\end{equation}
Since $\kappa=\delta^{1/(10D)}$ and $\sigma\sim\delta^{-1/D}$, the condition \eqref{eq:sigma-kappa-condition} is satisfied for all
sufficiently small $\delta$. After decreasing the common
$\delta_0>0$ once, the same choice also guarantees
$\sigma\ge\sigma_*$ from part~(i). Hence both the large-frequency
threshold and the small-data threshold are uniform over the complete
ray family.

We now carry out the transverse Gaussian concentration argument on
$\mathcal M_\kappa$. For each of the three principal moments in
\eqref{eq:remainder-orders-6-1}, we first apply
Lemma~\ref{lem:uniform-ray-family} to the complete integral, with the full transverse density $A_0(s,z')$ given in
\eqref{eq:full-leading-product-6-1}. Only once this concentration step has been performed do we restrict the leading amplitudes to the central ray and invoke \eqref{eq:product-principal-amplitudes}. Using
\[
dtdx=b^{n/2}J_{\bar g}(s,z')\,ds\,dz',
\qquad
J_{\bar g}(s,0)=1,
\]
together with \eqref{eq:product-principal-amplitudes} and
\eqref{iden-HY-matrix}, and recalling that
$|\varphi_t|^2=\frac12$ and
$\varphi_t=-\frac{\sqrt2}{2}$ on $\vartheta$, we obtain nonzero
coefficients $C_j=C_j(p_0,\zeta)$, $j=1,2,3$, satisfying
\begin{equation}\label{eq:uniform-leading-Cj}
0<c_*\le |C_j(p_0,\zeta)|\le C_*<\infty,
\qquad
(p_0,\zeta)\in\mathcal M_\kappa,
\end{equation}
where $c_*,C_*$ are independent of the chosen geodesic and of
$\kappa,\sigma,\delta$. Consequently, the three principal moments
satisfy
\begin{equation}\label{iden-int-I-alp}
1-e^{-\frac{\sqrt2}{4}\int_{s_-}^{s_+}\tilde\alpha(s,0)\,ds}=
C_1\sigma^{\frac n2-2}I_{\tilde\alpha}+\mathcal O\bigl(\sigma^{-1}+\kappa^{-1}\sigma^{-1/4}\bigr),
\end{equation}
\begin{equation}\label{eq:weighted-I-lambda}
\int_{s_-}^{s_+}\tilde\lambda(s,0)e^{-\frac{\sqrt2}{4}\int_{s_-}^{s}\tilde\alpha(\tau,0)\,d\tau}ds
=C_2\sigma^{\frac n2-1}I_{\tilde\lambda}+\mathcal O\bigl(\sigma^{-1}+\kappa^{-1}\sigma^{-1/4}\bigr),
\end{equation}
and
\begin{equation}\label{eq:weighted-I-q}
\int_{s_-}^{s_+}\tilde q(s,0)e^{-\frac{\sqrt2}{4}\int_{s_-}^{s}
\tilde\alpha(\tau,0)\,d\tau}ds=C_3\sigma^{\frac n2}I_{\tilde q}+\mathcal O\bigl(\sigma^{-1}+\kappa^{-1}\sigma^{-1/4}\bigr).
\end{equation}

To justify the uniform nondegeneracy in
\eqref{eq:uniform-leading-Cj}, observe first that, in each local Fermi
chart, the determinant identity
\[
\det(\Im H(s))\,|\det Y(s)|^2=\det(\Im H_0)
\]
cancels the potentially degenerating $Y$-factor arising in the leading transverse Gaussian integral. Moreover, transitions between local orthonormal Fermi frames are given by uniformly bounded linear transformations whose inverses are uniformly bounded as well. The local leading contributions are multiplied by $\chi_\mu(s)$, and summation over $\mu$, together with $\sum_\mu\chi_\mu(s)=1$, preserves the principal ray integral. Finally, the remaining metric and Jacobian normalization factors are smooth on a fixed compact set and are therefore uniformly bounded above and bounded away from zero. These observations establish \eqref{eq:uniform-leading-Cj}. The precise values of the coefficients $C_j$ are not essential for the subsequent stability argument.

We first observe that
\[
\int_{s_-}^{s_+}\tilde\alpha(s,0)e^{-\frac{\sqrt2}{4}\int_{s_-}^{s}\tilde\alpha(\tau,0)\,d\tau}ds=2\sqrt2
\left[1-e^{-\frac{\sqrt2}{4}\int_{s_-}^{s_+}\tilde\alpha(s,0)\,ds}\right].
\]
Indeed, this identity follows directly by differentiating the
exponential factor with respect to \(s\).

Along the central geodesic,
\[
s=\sqrt2(t-t_0),
\qquad
ds=\sqrt2\,dt.
\]
Consequently, up to fixed universal multiplicative constants, the
unweighted integrals with respect to the longitudinal variable \(s\)
coincide with the geodesic ray transforms \(\mathcal I_{p_0,\zeta}\) introduced in
Section~\ref{sec:geodesic-ray}. Since $\tilde\alpha$ is uniformly bounded and the lengths of all maximal geodesics are uniformly bounded, the elementary estimates for the exponential function yield
\begin{equation}\label{ineq:1-ex}
\Bigl|1-e^{-\frac{\sqrt2}{4}\int_{s_-}^{s}\tilde\alpha(\tau,0)\,d\tau}\Bigr|\lesssim \Bigl|
\int_{s_-}^{s}\tilde\alpha(\tau,0)\,d\tau\Bigr|,
\end{equation}
and, conversely,
\begin{equation}\label{eq:reverse-exp-alpha}
\Bigl|\int_{s_-}^{s_+}\tilde\alpha(s,0)\,ds\Bigr|\lesssim\Bigl|1-e^{-\frac{\sqrt2}{4}\int_{s_-}^{s_+}\tilde\alpha(s,0)\,ds}
\Bigr|.
\end{equation}

Before proceeding with the recovery of the coefficients, we record
three direct estimates that will be used in controlling the
lower-order terms arising in the successive subtraction arguments.
Using the definitions in \eqref{iden-alpha-lam-q-DtN},
the Cauchy-Schwarz inequality, the \emph{a priori}
\(L^\infty\)-bounds for the coefficient differences, and the
Gaussian-beam estimates
\[
\|\partial_t v_2\|_{L^2(\Omega_T)}+\|\partial_t y\|_{L^2(\Omega_T)}\lesssim\sigma^{1-\frac n4},
\qquad
\|v_2\|_{L^2(\Omega_T)}+\|y\|_{L^2(\Omega_T)}\lesssim\sigma^{-\frac n4},
\]
we obtain
\begin{equation}\label{eq:direct-I-component-bounds}
\begin{split}
|I_{\tilde\alpha}|\lesssim\sigma^{2-\frac n2}\|\tilde\alpha\|_{L^\infty(\Omega)},\quad
|I_{\tilde\lambda}|\lesssim\sigma^{1-\frac n2}\|\tilde\lambda\|_{L^\infty(\Omega)},\quad
|I_{\tilde q}|\lesssim\sigma^{-\frac n2}\|\tilde q\|_{L^\infty(\Omega)}.
\end{split}
\end{equation}
These bounds follow directly from the exact Gaussian-beam solutions
and therefore involve no additional error arising from the transverse
Gaussian concentration argument.

We next recover the three coefficients successively.

\medskip
\noindent\textbf{Stability estimate for $\tilde\alpha$.}
By \eqref{eq:direct-I-component-bounds} and the \emph{a priori} bounds
\[
|I_{\tilde\lambda}|\lesssim\sigma^{1-\frac n2},
\quad |I_{\tilde q}|\lesssim \sigma^{-\frac n2}.
\]
the relation
\[
I_{\tilde\alpha}=I_{\tilde\alpha,\tilde\lambda,\tilde q}-I_{\tilde\lambda}-I_{\tilde q},
\]
and \eqref{iden-int-I-alp}--\eqref{eq:reverse-exp-alpha}, we can obtain
\begin{equation}\label{ineq:exp-alpha-1}
|\mathcal I_{p_0,\zeta}(\tilde\alpha)|
\lesssim
\sigma^{\frac n2-2}
|I_{\tilde\alpha,\tilde\lambda,\tilde q}|
+
\sigma^{-1}
+
\kappa^{-1}\sigma^{-1/4}.
\end{equation}
The Gaussian-beam trace estimates, together with the analogous estimates for the adjoint beam, yield the following convenient
non-sharp bounds:
\begin{equation}\label{eq:boundary-GB-bounds-6-1}
\|h\|_{L^2(\Gamma_T)} \lesssim\sigma^{1/2},\quad\|f\|_{H^{s+2}(\Gamma_T)}+\|f\|_{\mathcal A^s(\Gamma_T)}\lesssim
\sigma^{s+\frac52}.
\end{equation}
Indeed, the estimate in \(H^{s+2}(\Gamma_T)\) follows from
\eqref{est:f-sigma-BT}. The remaining components of the
\(\mathcal A^s(\Gamma_T)\)-norm are controlled by the standard
time-trace embedding
\[
H^{s+2}(\Gamma_T)\hookrightarrow C^k([0,T];H^{s+\frac32-k}(\Gamma)),\quad 0\le k\le s+1.
\]
By \eqref{iden-alpha-lam-q-DtN}, \eqref{eq:direct-I-component-bounds} and \eqref{eq:boundary-GB-bounds-6-1}, we can get
\begin{equation}\label{eq:I-total-bound-6-1}
|I_{\tilde\alpha,\tilde\lambda,\tilde q}|\lesssim \sigma^{1/2}\bigl(\delta\varepsilon^{-1}
+\varepsilon\|f\|_{H^{s+2}(\Gamma_T)}^2\bigr).
\end{equation}
Substituting this estimate into \eqref{ineq:exp-alpha-1} gives
\begin{equation}\label{eq:alpha-ray-before-optimization}
|\mathcal I_{p_0,\zeta}(\tilde\alpha)|\lesssim\sigma^{\frac{n-3}{2}}\delta\varepsilon^{-1}+\varepsilon\sigma^{\frac{n+4s+7}{2}}+\sigma^{-1}+\kappa^{-1}\sigma^{-1/4}.
\end{equation}

We choose
\begin{equation}\label{eq:choice-alpha-parameters}
\sigma=c_1\delta^{-\frac1{n+2s+4}},\quad \varepsilon=c_2\delta^{\frac{n+4s+9}{2(n+2s+4)}},
\end{equation}
where \(c_1,c_2>0\) are fixed independently of \(\delta\). Then, for
\((p_0,\zeta)\in\mathcal M_\kappa\),
\begin{equation}\label{ineq:J-alpha}
|\mathcal I_{p_0,\zeta}(\tilde\alpha)|\lesssim \delta^{\frac1{n+2s+4}}+\kappa^{-1} \delta^{\frac1{4(n+2s+4)}}.
\end{equation}
Applying \eqref{eq:generic-L2-ray-holder} with
\(D=n+2s+4\), we obtain
\begin{equation}\label{eq:I-alpha-L2-ray}
\|\mathcal I(\tilde\alpha)\|_{L^2(\partial_-S\Omega)}
\lesssim
\delta^{\frac{3}{20(n+2s+4)}}.
\end{equation}
Moreover,
\[
\varepsilon\|f\|_{\mathcal A^s(\Gamma_T)}\lesssim\varepsilon\sigma^{s+\frac52}\lesssim\delta^{1/2}.
\]
Consequently, for sufficiently small \(\delta\), the Dirichlet boundary data \(\varepsilon f\in \mathcal U_{\Gamma_T}^{\epsilon_0,s}\).

We now apply the stability estimate for the geodesic X-ray transform. In the case $O=M$, \cite[Theorem 5.1]{bohr2021stability} applied to the scalar weight $W\equiv1$ gives an exponent $\mu\in(0,1)$, depending only on the fixed \emph{a priori} geometry, such that
\begin{equation}\label{eq:Bohr-Xray-L2}
\|F\|_{L^2(\Omega)} \lesssim \|F\|_{C^2(\overline\Omega)}^{1-\mu}\|\mathcal I(F)\|_{L^2(\partial_-S\Omega)}^{\mu}.
\end{equation}
The $L^2$ norm on the ray space in
\cite[Theorem 5.1]{bohr2021stability} is taken with respect to the natural Riemannian volume form on the corresponding subset of $\partial_+SM$, which is equivalent, up to orientation, to the ray-space norm used above.
We absorb the fixed grazing loss into the X-ray exponent by setting $\mathfrak m \vcentcolon=\frac{3}{20}\mu \in(0,1)$.
Since
$\|\tilde\alpha\|_{C^2(\overline\Omega)}\lesssim M_0$,
\eqref{eq:I-alpha-L2-ray} yields
\begin{equation}\label{est:alpha-stability-L2}
\|\tilde\alpha\|_{L^2(\Omega)}
\lesssim
\delta^{\frac{\mathfrak m}{n+2s+4}}.
\end{equation}
From this point on, $\mathfrak m$ denotes this fixed effective H\"older exponent. It depends only on the \emph{a priori} geometry and coefficient class, and not on $\delta$ or on an individual geodesic.

We need an auxiliary $L^\infty$ bound that is useful in the next step. Choose $l>\frac n2$ within the a priori regularity and $\theta_1\in(0,1)$ such that $(1-\theta_1)l>\frac n2.$ By interpolation and Sobolev embedding, one has
\begin{equation}\label{est:alpha-aux-Linf}
\|\tilde\alpha\|_{L^\infty(\Omega)} \lesssim \|\tilde\alpha\|_{L^2(\Omega)}^{\theta_1}
\|\tilde\alpha\|_{H^l(\Omega)}^{1-\theta_1} \lesssim \delta^{\frac{\mathfrak m\theta_1}{n+2s+4}}.
\end{equation}

\medskip
\noindent\textbf{Stability estimate for $\tilde\lambda$.}
By \eqref{eq:weighted-I-lambda} and \eqref{ineq:1-ex}, we have
\begin{equation}\label{ide-int-I-lam}
\begin{split}
\Bigl|\int_{s_-}^{s_+}\tilde\lambda(s,0)\,ds\Bigr|\lesssim\sigma^{\frac n2-1}|I_{\tilde\lambda}|+\int_{s_-}^{s_+}\Bigl|
\int_{s_-}^{s}\tilde\alpha(\tau,0)\,d\tau\Bigr|ds+\sigma^{-1}+\kappa^{-1}\sigma^{-1/4}.
\end{split}
\end{equation}
Using
\[
I_{\tilde\lambda}
=
I_{\tilde\alpha,\tilde\lambda,\tilde q}
-I_{\tilde\alpha}-I_{\tilde q},
\]
together with the estimates \eqref{eq:direct-I-component-bounds}, we obtain
\begin{equation}\label{ineq:tilde-lambda}
|\mathcal I_{p_0,\zeta}(\tilde\lambda)|
\lesssim
\sigma^{\frac n2-1}
|I_{\tilde\alpha,\tilde\lambda,\tilde q}|
+
\sigma\|\tilde\alpha\|_{L^\infty(\Omega)}
+
\sigma^{-1}+\kappa^{-1}\sigma^{-1/4}.
\end{equation}
Combining this with \eqref{eq:I-total-bound-6-1} gives
\begin{equation}\label{est:J-lambda}
|\mathcal I_{p_0,\zeta}(\tilde\lambda)|\lesssim\sigma^{\frac{n-1}{2}}\delta\varepsilon^{-1}+
\varepsilon\sigma^{\frac{n+4s+9}{2}}+\sigma\|\tilde\alpha\|_{L^\infty(\Omega)}+\sigma^{-1}+\kappa^{-1}\sigma^{-1/4}.
\end{equation}
Take
\begin{equation}\label{constant-k1}
k_1=(\mathfrak m\theta_1)^{-1}\left[8+(2-\mathfrak m\theta_1)(n+2s)\right],
\end{equation}
and choose
\begin{equation}\label{eq:choice-lambda-parameters}
\sigma=c_3\delta^{-\frac1{n+2s+k_1}},\quad\varepsilon=c_4\delta^{\frac{n+4s+k_1+5}{2(n+2s+k_1)}}.
\end{equation}
Then \eqref{est:alpha-aux-Linf} implies
\begin{equation}\label{eq:lambda-ray-optimized}
\begin{split}
|\mathcal I_{p_0,\zeta}(\tilde\lambda)|&\lesssim\delta^{\frac{k_1-4}{2(n+2s+k_1)}}+\delta^{\frac{\mathfrak m\theta_1}{n+2s+4}
-\frac1{n+2s+k_1}}+\delta^{\frac1{n+2s+k_1}}+\kappa^{-1}\delta^{\frac1{4(n+2s+k_1)}}\\
&\lesssim\delta^{\frac1{n+2s+k_1}}+\kappa^{-1}\delta^{\frac1{4(n+2s+k_1)}},\quad (p_0,\zeta)\in\mathcal M_\kappa.
\end{split}
\end{equation}
Indeed, \eqref{constant-k1} gives
\[
\frac{\mathfrak m\theta_1}{n+2s+4}-\frac1{n+2s+k_1}=\frac1{n+2s+k_1},
\]
and $k_1>6$. Furthermore,
\[
\varepsilon\|f\|_{\mathcal A^s(\Gamma_T)}\lesssim\delta^{1/2}.
\]
Applying \eqref{eq:generic-L2-ray-holder} with
$D=n+2s+k_1$ and then \eqref{eq:Bohr-Xray-L2} to
$\tilde\lambda$, we obtain
\begin{equation}\label{est:lambda-stability-L2}
\|\tilde\lambda\|_{L^2(\Omega)}
\lesssim
\delta^{\frac{\mathfrak m}{n+2s+k_1}}.
\end{equation}
Choosing $\theta_2\in(0,1)$ such that $(1-\theta_2)l>\frac n2$, again by interpolation, we can get
\begin{equation}\label{est:lambda-aux-Linf}
\|\tilde\lambda\|_{L^\infty(\Omega)}
\lesssim
\delta^{\frac{\mathfrak m\theta_2}{n+2s+k_1}}.
\end{equation}

\medskip
\noindent\textbf{Stability estimate for $\tilde q$.}
By \eqref{eq:weighted-I-q} and \eqref{ineq:1-ex},
\begin{equation}\label{ide-int-I-q}
\begin{split}
\Bigl|\int_{s_-}^{s_+}\tilde q(s,0)\,ds\Bigr|\lesssim\sigma^{\frac n2}|I_{\tilde q}|+\int_{s_-}^{s_+}
\Bigl|\int_{s_-}^{s}\tilde\alpha(\tau,0)\,d\tau\Bigr|ds+\sigma^{-1}+\kappa^{-1}\sigma^{-1/4}.
\end{split}
\end{equation}
Since
\[
I_{\tilde q}=I_{\tilde\alpha,\tilde\lambda,\tilde q}-I_{\tilde\alpha}-I_{\tilde\lambda},
\]
the estimates \eqref{eq:direct-I-component-bounds} imply
\begin{equation}\label{est:q-Iq}
|\mathcal I_{p_0,\zeta}(\tilde q)|\lesssim \sigma^{\frac n2}|I_{\tilde\alpha,\tilde\lambda,\tilde q}|+
\sigma^2\|\tilde\alpha\|_{L^\infty(\Omega)}+\sigma\|\tilde\lambda\|_{L^\infty(\Omega)}+\sigma^{-1}+\kappa^{-1}\sigma^{-1/4}.
\end{equation}
Using \eqref{eq:I-total-bound-6-1}, we therefore obtain
\begin{equation}\label{eq:q-ray-before-optimization}
\begin{split}
|\mathcal I_{p_0,\zeta}(\tilde q)|\lesssim\sigma^{\frac{n+1}{2}}\delta\varepsilon^{-1}+\varepsilon\sigma^{\frac{n+4s+11}{2}}
+\sigma^2\|\tilde\alpha\|_{L^\infty(\Omega)}+\sigma\|\tilde\lambda\|_{L^\infty(\Omega)}+\sigma^{-1}+\kappa^{-1}\sigma^{-1/4}.
\end{split}
\end{equation}
In particular, the first power is
$\sigma^{(n+1)/2}$, since
$\sigma^{n/2}$ in \eqref{est:q-Iq} is multiplied by the boundary
factor
$\|h\|_{L^2(\Gamma_T)}\lesssim\sigma^{1/2}$.

We set
\begin{equation}\label{constant-k2}
k_2=\max\left\{(\mathfrak m\theta_1)^{-1}\left[12+(3-\mathfrak m\theta_1)(n+2s)\right],\,(\mathfrak m\theta_2)^{-1}
\left[2k_1+(2-\mathfrak m\theta_2)(n+2s)\right]\right\},
\end{equation}
and choose
\begin{equation}\label{eq:choice-q-parameters}
\sigma=c_5\delta^{-\frac1{n+2s+k_2}},\quad\varepsilon=c_6\delta^{\frac{n+4s+k_2+5}{2(n+2s+k_2)}}.
\end{equation}
Then \eqref{est:alpha-aux-Linf} and
\eqref{est:lambda-aux-Linf} give
\begin{equation}\label{eq:q-ray-optimized}
\begin{split}
|\mathcal I_{p_0,\zeta}(\tilde q)|\lesssim&\delta^{\frac{k_2-6}{2(n+2s+k_2)}}+\delta^{\frac{\mathfrak m\theta_1}{n+2s+4}-\frac2{n+2s+k_2}}+\delta^{\frac{\mathfrak m\theta_2}{n+2s+k_1}-\frac1{n+2s+k_2}}+\delta^{\frac1{n+2s+k_2}}+\kappa^{-1}\delta^{\frac1{4(n+2s+k_2)}}\\
&\lesssim\delta^{\frac1{n+2s+k_2}}+\kappa^{-1}\delta^{\frac1{4(n+2s+k_2)}},
\quad
(p_0,\zeta)\in\mathcal M_\kappa.
\end{split}
\end{equation}
Indeed, the first quantity in the maximum
\eqref{constant-k2} gives
\[
\frac{\mathfrak m\theta_1}{n+2s+4}
-\frac2{n+2s+k_2}
\ge
\frac1{n+2s+k_2},
\]
and the second gives
\[
\frac{\mathfrak m\theta_2}{n+2s+k_1}
-\frac1{n+2s+k_2}
\ge
\frac1{n+2s+k_2}.
\]
Also $k_2>8$, so
\[
\frac{k_2-6}{2(n+2s+k_2)}
\ge
\frac1{n+2s+k_2}.
\]
Moreover,
\[
\varepsilon\|f\|_{\mathcal A^s(\Gamma_T)}\lesssim\delta^{1/2}.
\]
Applying \eqref{eq:generic-L2-ray-holder} with $D=n+2s+k_2$ and then \eqref{eq:Bohr-Xray-L2} yields
\begin{equation}\label{est:q-stability-L2}
\|\tilde q\|_{L^2(\Omega)}\lesssim\delta^{\frac{\mathfrak m}{n+2s+k_2}}.
\end{equation}

Finally, since $k_2>k_1>4$, the estimates \eqref{est:alpha-stability-L2}, \eqref{est:lambda-stability-L2}, and
\eqref{est:q-stability-L2} give the intermediate $L^2$ stability estimate
\begin{equation}\label{est:for-alp-lam-q-DN-L2}
\|\alpha_2-\alpha_1\|_{L^2(\Omega)}+\|\lambda_2-\lambda_1\|_{L^2(\Omega)}+\|q_2-q_1\|_{L^2(\Omega)}\lesssim \delta^{\frac{\mathfrak m}{n+2s+k_2}}.
\end{equation}

We are also required to establish the auxiliary $L^\infty$ estimate used in the subsequent recovery of the nonlinear coefficient. To this end, we choose
$\theta_3\in(0,1)$ such that $(1-\theta_3)l>\frac n2.$
By interpolation between \eqref{est:q-stability-L2} and the fixed \emph{a priori} $H^l(\Omega)$ bound, we can obtain
\begin{equation}\label{est:q-stability-L-inf}
\|\tilde q\|_{L^\infty(\Omega)}\lesssim\delta^{\frac{\mathfrak m\theta_3}{n+2s+k_2}}.
\end{equation}
Let $\theta_0\vcentcolon=\min\{\theta_1,\theta_2,\theta_3\}.$
Since $k_2>k_1>4$, the estimates \eqref{est:alpha-aux-Linf}, \eqref{est:lambda-aux-Linf}, and \eqref{est:q-stability-L-inf} imply
\begin{equation}\label{est:for-alp-lam-q-DN}
\|\alpha_2-\alpha_1\|_{L^\infty(\Omega)}+\|\lambda_2-\lambda_1\|_{L^\infty(\Omega)}+\|q_2-q_1\|_{L^\infty(\Omega)}
\lesssim \delta^{\mu_0},
\end{equation}
where
\begin{equation}\label{eq:mu0-linear-61}
\mu_0\vcentcolon=\frac{\mathfrak m\theta_0}{n+2s+k_2}\in(0,1).
\end{equation}
The exponent and the implicit constant depend only on the fixed \emph{a priori} data. Therefore, the proof of Theorem~\ref{thm:stable-alp-lam-q} is complete. \hfill $\square$

\subsection{Stable determination of the nonlinear coefficient $\xi$}
In this section, we focus on giving the proof of the stable determination of the nonlinear coefficient $\xi$. Instead of using Gaussian beam solutions in the geometric setting, we apply suitable GO solutions under the condition that $b$ is a known positive constant.

\subsubsection{Construction of GO solutions}
We set
\begin{equation}
\phi_\omega(t,x)
\vcentcolon=
t+\frac{x\cdot\omega}{\sqrt b},
\qquad
X_\omega
=
\partial_t-\sqrt b\,\omega\cdot\nabla,
\qquad
\omega\in\mathbb S^{n-1}.
\end{equation}
Recall the notion $\Upsilon_2=\alpha_2-\frac{c^2}{b}.$ A direct computation gives
\begin{equation}\label{eqn:compu-L-e-phi}
e^{-i\sigma\phi_\omega}
L_{\alpha_2,b,c,\lambda_2,q_2}
e^{i\sigma\phi_\omega}
=
\sigma^2P_2+i\sigma P_1
+L_{\alpha_2,b,c,\lambda_2,q_2},
\end{equation}
where
\begin{equation}
\begin{split}
P_2&\vcentcolon=-2X_\omega-\Upsilon_2,\\
P_1&\vcentcolon=3\partial_t^2+2\alpha_2\partial_t-b\Delta-2\sqrt b\,\omega\cdot\nabla\partial_t-\frac{2c^2}{\sqrt b}\omega\cdot\nabla+\lambda_2.
\end{split}
\end{equation}
For $N\in\mathbb N_{\ge1}$, we construct GO solutions of
$L_{\alpha_2,b,c,\lambda_2,q_2}v_2=0$ of the form
\begin{equation}
v_2(t,x)=e^{i\sigma\phi_\omega(t,x)}a_\sigma(t,x)+r_\sigma(t,x),\quad a_\sigma=a_0+\sigma^{-1}a_1+\cdots+\sigma^{-N}a_N.
\end{equation}
We use the fixed smooth extensions introduced above and keep the same
notation for the extended coefficients. In particular, when
$\tilde\alpha$ occurs outside $\Omega$, it denotes the difference of
the two fixed smooth extensions; no zero extension of
$\tilde\alpha$ is used.

The leading transport equation is
\begin{equation}
X_\omega a_0+\frac{\Upsilon_2}{2}a_0=0.
\end{equation}
We choose $ \Psi\in C_c^\infty(\mathbb R^n)$ such that $\int_{\mathbb R^n}\Psi^3\,dx=1$.
Thus we can take
\begin{equation}\label{eq:a0-GO-xi}
a_0(t,x)=\Psi_\ell(x+\sqrt bt\,\omega-p)e^{-\frac12\int_0^t\Upsilon_2(x+\sqrt b(t-\tau)\omega)\,d\tau},
\end{equation}
where $\Psi_\ell(z)\vcentcolon=\ell^{-\frac n3}\Psi\left(\frac z\ell\right)$.
The subsequent amplitudes satisfy
\begin{equation}\label{eq:recur-a-k-to-N}
X_\omega a_k+\frac{\Upsilon_2}{2}a_k=\frac{i}{2}P_1a_{k-1}+\frac12L_{\alpha_2,b,c,\lambda_2,q_2}a_{k-2},\quad
1\le k\le N,
\end{equation}
with the convention that $a_{-1}=0$. In particular, for $k=1$ we can construct
\begin{equation}
\begin{split}
a_1(t,x)&=\Psi_\ell(x+\sqrt bt\,\omega-p)e^{-\frac12\int_0^t\Upsilon_2(x+\sqrt b(t-\tau)\omega)\,d\tau}\\
&\quad+\frac{i}{2}\int_0^t(P_1a_0)(\tau,x+\sqrt b(t-\tau)\omega)e^{-\frac12\int_\tau^t\Upsilon_2(x+\sqrt b(t-s)\omega)\,ds}
d\tau.
\end{split}
\end{equation}
The remaining amplitude functions $a_k$, $2\le k\le N$, can be obtained recursively by solving the ODEs \eqref{eq:recur-a-k-to-N}.

Let
\[
\partial_-S_E\Omega\vcentcolon=\left\{(x,\theta)\in\partial\Omega\times\mathbb S^{n-1}:
\theta\cdot\nu(x)<0\right\}.
\]
We fix $(x_-,\theta)\in\partial_-S_E\Omega$, where $\theta$ denotes the direction of propagation through $\Omega$.
Let $\tau_E(x_-,\theta)$ be the length of the line issued from $x_-$ in the direction $\theta$. We set
\[
\omega=-\theta,\quad t_-=\frac12\Bigl(T-\frac{\tau_E(x_-,\theta)}{\sqrt b}\Bigr),\quad
t_+=\frac12\Bigl(T+\frac{\tau_E(x_-,\theta)}{\sqrt b}\Bigr),
\]
and choose
\begin{equation}\label{eq:p-choice-xi}
p=x_-+\sqrt b\,t_-\,\omega=x_- -\sqrt b\,t_-\,\theta.
\end{equation}
With this choice, we know that
\[
c_{x_-,\theta}(t)
\vcentcolon=
p-\sqrt bt\,\omega
=
p+\sqrt bt\,\theta,
\]
satisfies $c_{x_-,\theta}(t_-)=x_-$ and $c_{x_-,\theta}(t_+)=x_+,$ where $x_+$ denotes the exit point. Since $\tau_E(x_-,\theta)\le {\rm diam}\,\Omega,$ we can obtain
\begin{equation}\label{eq:uniform-time-margin-GO-xi}
t_-\ge\frac{\sqrt b\,T-{\rm diam}\,\Omega}{2\sqrt b},\quad T-t_+\ge\frac{\sqrt b\,T-{\rm diam}\,\Omega}{2\sqrt b}.
\end{equation}

Lemma~\ref{lem:uniform-GO-xi} below establishes that $c_{x_-,\theta}(0)$ and $c_{x_-,\theta}(T)$ remain uniformly
separated from $\overline\Omega$. In particular, there exists $\ell_*>0$, independent of $(x_-,\theta)$, such that, for every
$0<\ell\le\ell_*$,
\begin{equation}\label{eq:uniform-endpoint-support-GO-xi}
\supp\Psi_\ell(\,\cdot-p\,)\cap\overline\Omega=\varnothing,
\qquad
\supp\Psi_\ell(\,\cdot+\sqrt bT\,\omega-p\,)
\cap\overline\Omega=\varnothing.
\end{equation}
Consequently, the required initial compatibility conditions for the
forward GO solution, as well as the corresponding terminal
compatibility conditions for the adjoint construction, are satisfied
uniformly over the entire incoming ray family.

Using \eqref{eqn:compu-L-e-phi} and
\eqref{eq:recur-a-k-to-N}, the remainder satisfies
\begin{equation}\label{eq:sys-remain-GO}
\begin{cases}
L_{\alpha_2,b,c,\lambda_2,q_2}r_\sigma
=-e^{i\sigma t}F_\sigma
&\text{in }\Omega_T,\\
r_\sigma=0
&\text{on }\Gamma_T,\\
r_\sigma(0)=\partial_tr_\sigma(0)
=\partial_t^2r_\sigma(0)=0
&\text{in }\Omega,
\end{cases}
\end{equation}
where
\begin{equation}
\begin{split}
F_\sigma=e^{i\sigma\frac{x\cdot\omega}{\sqrt b}}\bigl[\sigma^{1-N}\big(iP_1a_N+L_{\alpha_2,b,c,\lambda_2,q_2}a_{N-1}
\big)+\sigma^{-N}L_{\alpha_2,b,c,\lambda_2,q_2}a_N\bigr].
\end{split}
\end{equation}
The higher-order compatibility conditions follow from the support choice above. By the energy estimate for
\eqref{eq:sys-remain-GO}, for every integer $s\ge2$ within the available coefficient regularity,
\begin{equation}\label{est:remainder-r-sig-GO}
\|r_\sigma\|_{E^{s+2}(\Omega_T)}
\lesssim
\sigma^{s+1-N}
\|\Psi_\ell\|_{H^{2N+s+2}(\mathbb R^n)}
\lesssim
\sigma^{s+1-N}
\ell^{-2N-s-2+\frac n6}.
\end{equation}

Similarly, for the adjoint equation
$L^*_{\alpha_1,b,c,\lambda_1,q_1}y=0$ we can construct
\begin{equation}
y(t,x)=e^{-2i\sigma\phi_\omega(t,x)}d_\sigma(t,x)+R_\sigma(t,x),\quad
d_\sigma=d_0+(2\sigma)^{-1}d_1+\cdots+(2\sigma)^{-N}d_N.
\end{equation}
The leading amplitude satisfies
\begin{equation}
X_\omega d_0-\frac{\Upsilon_1}{2}d_0=0.
\end{equation}
We can choose
\begin{equation}\label{eq:d0-GO-xi}
d_0(t,x)=\Psi_\ell(x+\sqrt bt\,\omega-p)e^{\frac12\int_0^t\Upsilon_1(x+\sqrt b(t-\tau)\omega)\,d\tau}.
\end{equation}
The amplitudes $d_k$, $1\le k\le N$, are obtained from the corresponding inhomogeneous transport equations. The transport recursion preserves the translating support variable $x+\sqrt bt\,\omega-p$. Hence, with the uniform choice
$0<\ell\le\ell_*$ in \eqref{eq:uniform-endpoint-support-GO-xi}, all terminal conditions required
for the backward problem vanish in $\Omega$. The remainder $R_\sigma$ satisfies
\[
R_\sigma(T)=\partial_tR_\sigma(T)=\partial_t^2R_\sigma(T)=0 \quad\text{in }\Omega
\]
and obeys the same type of estimate as \eqref{est:remainder-r-sig-GO}, with an implicit constant independent of $\sigma$ and $\ell$. We next give the proof of Theorem \ref{thm:stable-alp-lam-q-xi}.
\medskip

\subsubsection{Proof of Theorem \ref{thm:stable-alp-lam-q-xi}}

We use the integer $s$ fixed in the formulation of the inverse
problem. In particular, $s>n+2$, so that all Sobolev embeddings invoked
below are valid. Accordingly, we fix $N=s+1.$
We set
\begin{equation}
f=e^{i\sigma\phi_\omega}a_\sigma\big|_{\Gamma_T},\quad h=e^{-2i\sigma\phi_\omega}d_\sigma\big|_{\Gamma_T}.
\end{equation}
Let
\begin{equation}
I_{\tilde\xi}\vcentcolon=\int_{\Omega_T}\tilde\xi(\partial_t y)(\partial_t v_2)v_2\,dx\,dt.
\end{equation}
Since $\partial_t\phi_{\omega}=1$, the principal contribution to the product is
\[
(\partial_t y)(\partial_t v_2)v_2=2\sigma^2a_0^2d_0+\text{lower-order terms}.
\]
We define
\begin{equation}\label{eq:def-J-ell}
J_\ell
\vcentcolon=
\int_{\Omega_T}
\tilde\xi
\Psi_\ell^3(x+\sqrt bt\,\omega-p)
e^{\tilde h(t,x)}
\,dx\,dt,
\end{equation}
where
\begin{equation}
\tilde h(t,x)
\vcentcolon=
-\frac12\int_0^t
\tilde\alpha(x+\sqrt b(t-\tau)\omega)\,d\tau
-\frac12\int_0^t
\Upsilon_2(x+\sqrt b(t-\tau)\omega)\,d\tau.
\end{equation}
By the definitions of $a_0$ and $d_0$, one has
\[
J_\ell
=
\int_{\Omega_T}
\tilde\xi a_0^2d_0\,dx\,dt.
\]

We next estimate the error generated by substituting the GO solutions
into $I_{\tilde\xi}$. The transport recursion yields, for
$|\beta|\le2$ and $0\le k\le N$,
\[
\|\partial_{t,x}^{\beta}a_k\|_{L^\infty(\Omega_T)}
+
\|\partial_{t,x}^{\beta}d_k\|_{L^\infty(\Omega_T)}
\lesssim
\ell^{-\frac n3-|\beta|-2k}.
\]
We work under the parameter condition
\begin{equation}\label{eq:sigma-ell-regime}
\sigma\ge \ell^{-(6N+6)}.
\end{equation}
In particular, $\sigma^{-1}\ell^{-2}\le1,$ and hence the full amplitude expansions satisfy
\[
\|\partial_{t,x}^{\beta}a_\sigma\|_{L^\infty(\Omega_T)}+\|\partial_{t,x}^{\beta}d_\sigma\|_{L^\infty(\Omega_T)}
\lesssim \ell^{-\frac n3-|\beta|},\quad |\beta|\le2.
\]
Since $N=s+1$, estimate \eqref{est:remainder-r-sig-GO} gives
\[
\|r_\sigma\|_{E^{s+2}(\Omega_T)}
+
\|R_\sigma\|_{E^{s+2}(\Omega_T)}
\lesssim
\ell^{-3s-4+\frac n6}.
\]

We now expand $(\partial_t y)(\partial_t v_2)v_2$. The contribution containing the three principal oscillatory factors is precisely $2\sigma^2a_0^2d_0.$ Every remaining term that contains neither $r_\sigma$ nor
$R_\sigma$ is of strictly lower order in $\sigma$. Consequently, after normalization by $\sigma^2$, the sum of all such contributions is bounded by
\[
C\sigma^{-1}\ell^{-(6N+5)}.
\]
For the terms containing remainders, the same estimate can also be achieved. Hence,
for $\sigma>1$ and $0<\ell<1$, it holds that
\begin{equation}\label{est:J-ell-I-xi}
|J_\ell|\lesssim\sigma^{-2}|I_{\tilde\xi}|+\sigma^{-1}\ell^{-(6N+5)}.
\end{equation}
The choice of $\sigma$ and $\ell$ made below satisfies
\eqref{eq:sigma-ell-regime}, up to a fixed multiplicative constant.

We next estimate the boundary term in \eqref{iden:xi-DtN}. We take $\varepsilon_1=\varepsilon_2=\varepsilon$ and $f_1=f_2=f.$
By \eqref{est:second-R} and \eqref{eq:boundary-second-difference}, we can also obtain
\begin{equation}\label{est:Jbv-xi}
\begin{split}
\Bigl|\int_{\Gamma_T}\mathcal B(w)h\,d\Gamma\,dt\Bigr|\lesssim\varepsilon^{-2}\bigl(\delta+\|\varepsilon f\|_{H^{s+2}(\Gamma_T)}^3\bigr)\|h\|_{L^2(\Gamma_T)}.
\end{split}
\end{equation}
It follows from \eqref{iden:xi-DtN} and
\eqref{est:Jbv-xi} that
\begin{equation}\label{est-mid-J-xi}
\begin{split}
|I_{\tilde\xi}|&\lesssim\bigl(\varepsilon^{-2}\delta+\varepsilon\|f\|_{H^{s+2}(\Gamma_T)}^3\bigr)\|h\|_{L^2(\Gamma_T)}+\|\tilde\alpha\|_{L^\infty(\Omega)}
\|\partial_tw_2\|_{L^2(\Omega_T)}
\|\partial_ty\|_{L^2(\Omega_T)}\\
&\quad+\|\tilde\lambda\|_{L^\infty(\Omega)}\|\partial_tw_2\|_{L^2(\Omega_T)}\|y\|_{L^2(\Omega_T)}+\|\tilde q\|_{L^\infty(\Omega)}\|w_2\|_{L^2(\Omega_T)}\|y\|_{L^2(\Omega_T)}\\
&\quad+\|\xi_1\|_{L^\infty(\Omega)}\|y\|_{L^\infty(\Omega_T)}\sum_{j=1}^2\|v_j\|_{H^2(0,T;L^2(\Omega))}
\|v\|_{H^2(0,T;L^2(\Omega))}.
\end{split}
\end{equation}
The boundary data satisfy the rough estimates
\begin{equation}\label{eq:GO-boundary-xi}
\|f\|_{H^k(\Gamma_T)}+\|h\|_{H^k(\Gamma_T)}\lesssim\sigma^{k+\frac12}\ell^{-k-1-2N+\frac n6}, \quad k\ge0.
\end{equation}
Moreover, we have
\begin{equation}\label{eq:A-s-GO-xi}
\|f\|_{\mathcal A^s(\Gamma_T)}\lesssim\sigma^{s+\frac52}\ell^{-s-3-2N+\frac n6}.
\end{equation}
By the energy estimates for the forward and backward linear problems, we can get
\begin{equation}
\|\partial_ty\|_{L^2(\Omega_T)}+\|y\|_{L^2(\Omega_T)}\lesssim\sigma^{\frac52}\ell^{-3-2N+\frac n6},
\end{equation}
\begin{equation}
\|v_j\|_{H^2(0,T;L^2(\Omega))}\lesssim\sigma^{\frac52}\ell^{-3-2N+\frac n6},\quad j=1,2.
\end{equation}
Denote
\begin{equation}
\mathcal N(\tilde\alpha,\tilde\lambda,\tilde q)\vcentcolon=\|\tilde\alpha\|_{L^\infty(\Omega)}+\|\tilde\lambda\|_{L^\infty(\Omega)}+
\|\tilde q\|_{L^\infty(\Omega)}.
\end{equation}
Furthermore, from \eqref{eqn:linear-1st-v1-v2},
\begin{equation}
\begin{split}
\|v\|_{H^2(0,T;L^2(\Omega))}\lesssim\|\tilde\alpha\partial_t^2v_2+\tilde\lambda\partial_tv_2+\tilde qv_2\|_{L^2(\Omega_T)}
\lesssim\sigma^{\frac52}\ell^{-3-2N+\frac n6}\mathcal N(\tilde\alpha,\tilde\lambda,\tilde q),
\end{split}
\end{equation}

For $w_2$, we recall the equation
\[
L_{\alpha_2,b,c,\lambda_2,q_2}w_2=\partial_t^2(\xi_2v_2^2).
\]
Consequently,
\begin{equation}
\begin{split}
\|w_2\|_{L^2(\Omega_T)}+\|\partial_tw_2\|_{L^2(\Omega_T)}\lesssim
\|\partial_t^2(\xi_2v_2^2)\|_{L^2(\Omega_T)}\lesssim
\|v_2\|_{H^{n+1}(\Omega_T)}\|v_2\|_{H^2(0,T;L^2(\Omega))}.
\end{split}
\end{equation}
Since $n\ge3$, $n+1>1+\frac{n+1}{2},$ and hence $H^{n+1}(\Omega_T)\hookrightarrow W^{1,\infty}(\Omega_T).$
This is the embedding used to control both $v_2$ and
$\partial_tv_2$ in the product
$\partial_t^2(\xi_2v_2^2)$. 
The boundary estimates give
\begin{equation}
\|y\|_{L^\infty(\Omega_T)}\lesssim\|y\|_{H^n(\Omega_T)}\lesssim\sigma^{n+\frac12}\ell^{-\frac{5n}{6}-1-2N},
\end{equation}
\begin{equation}
\|v_1\|_{H^{n+1}(\Omega_T)}+\|v_2\|_{H^{n+1}(\Omega_T)}\lesssim\sigma^{n+\frac32}\ell^{-\frac{5n}{6}-2-2N}.
\end{equation}
By the estimate \eqref{est:for-alp-lam-q-DN} obtained in Section~\ref{Sec-proof-lin},
\begin{equation}\label{eq:N-linear-coeff-xi}
\mathcal N(\tilde\alpha,\tilde\lambda,\tilde q)\lesssim \delta^{\frac{\mathfrak m\theta_0}{n+2s+k_2}}.
\end{equation}
Collecting the above bounds in \eqref{est-mid-J-xi}, and then dividing
by $\sigma^2$, yields
\begin{equation}\label{est:J-xi-delta-sigma}
\begin{split}
\sigma^{-2}|I_{\tilde\xi}|\lesssim\varepsilon^{-2}\sigma^{-\frac32}\delta\,\ell^{-1-2N}+\varepsilon
\sigma^{3s+6}\ell^{-3s-10-8N}+\sigma^{n+\frac92}\ell^{-8-6N-n}\delta^{\frac{\mathfrak m\theta_0}{n+2s+k_2}}.
\end{split}
\end{equation}
We note that the powers of $\sigma$ and $\ell$ used here are not sharp.

For later analysis, we are also required to prove the following lemma.

\begin{lemma}[Uniformity of the GO construction]\label{lem:uniform-GO-xi}
Assume that $\Omega\subset\mathbb R^n$ is bounded with smooth strictly
convex boundary and
\[
\sqrt b\,T>{\rm diam}\,\Omega.
\]
Let $p$ be chosen as in \eqref{eq:p-choice-xi}. Then there
exist
\[
d_*>0,\quad \ell_*>0,\quad C\ge1,
\]
depending only on $\Omega$, $b$, $T$, $\supp\Psi$, the fixed coefficient extensions, the \emph{a priori} coefficient bounds, and the order $N$, such that the following statements hold uniformly for all $(x_-,\theta)\in\partial_-S_E\Omega$.

{\rm (i)} It holds that
\begin{equation}\label{eq:endpoint-center-separation-xi}
{\rm dist}\bigl(c_{x_-,\theta}(0),\overline\Omega\bigr) \ge d_*,
\quad
{\rm dist}\bigl(c_{x_-,\theta}(T),\overline\Omega\bigr)\ge d_*.
\end{equation}
Consequently, \eqref{eq:uniform-endpoint-support-GO-xi} holds for every
$0<\ell\le\ell_*$.

{\rm (ii)} For every fixed finite order of derivatives involved in the construction,
\[
\|\partial_{t,x}^{\beta}a_k\|_{L^\infty(\Omega_T)}+\|\partial_{t,x}^{\beta}d_k\|_{L^\infty(\Omega_T)}
\le C_{\beta,k} \ell^{-\frac n3-|\beta|-2k},
\]
where $C_{\beta,k}$ is independent of $(x_-,\theta)$. The residual estimates for the truncated forward and adjoint GO ansatzes are uniform with respect to $(x_-,\theta)$ as well.

{\rm (iii)} The remainder estimates and the boundary trace estimates \eqref{est:remainder-r-sig-GO},
\eqref{eq:GO-boundary-xi}, and \eqref{eq:A-s-GO-xi}
hold with constants independent of $(x_-,\theta)$. In particular, the
constants in \eqref{est:J-ell-I-xi} and
\eqref{est:J-xi-delta-sigma} are uniform over
$\partial_-S_E\Omega$.
\end{lemma}

The proof of the above lemma is given in the appendix.

We next identify the weighted ray transform contained in $J_\ell$.
Set
\[
z=x+\sqrt bt\,\omega-p.
\]
Since $\tilde\xi=0$ on $\Gamma$, its zero extension to $\mathbb R^n$
is globally Lipschitz. In the attenuation factors we use the fixed extensions introduced in the GO construction above. In particular, in $\mathbb R^n\backslash\Omega$, $\tilde\alpha$ means the difference of the fixed extensions of $\alpha_2$ and $\alpha_1$. Then
\begin{equation}\label{eq:J-ell-change-var}
\begin{split}
J_\ell=\int_0^T\int_{\mathbb R^n}\tilde\xi(p-\sqrt bt\,\omega+z)\Psi_\ell^3(z)e^{-\frac12\int_0^t\Upsilon_2(p-\sqrt b\tau\,\omega+z)\,d\tau}e^{-\frac12\int_0^t\tilde\alpha(p-\sqrt b\tau\,\omega+z)\,d\tau}\,dzdt.
\end{split}
\end{equation}
Using the fact that $\int_{\mathbb R^n}\Psi_\ell^3\,dz=1$, we can define
\begin{equation}\label{eq:J0-xi}
\begin{split}
J_0\vcentcolon=\int_0^T\tilde\xi(p-\sqrt bt\,\omega)e^{-\frac12\int_0^t\Upsilon_2(p-\sqrt b\tau\,\omega)\,d\tau}
e^{-\frac12\int_0^t\tilde\alpha(p-\sqrt b\tau\,\omega)\,d\tau}dt.
\end{split}
\end{equation}
Denote by
\begin{equation}
\begin{split}
F_{\tilde\xi}(t,z)\vcentcolon=&\tilde\xi(p-\sqrt bt\,\omega+z)e^{-\frac12\int_0^t\Upsilon_2(p-\sqrt b\tau\,\omega+z)\,d\tau}
e^{-\frac12\int_0^t\tilde\alpha(p-\sqrt b\tau\,\omega+z)\,d\tau}.
\end{split}
\end{equation}
The zero extension of $\tilde\xi$ is in $W^{1,\infty}(\mathbb R^n)$. Therefore
\[
|F_{\tilde\xi}(t,z)-F_{\tilde\xi}(t,0)|
\lesssim |z|.
\]
It then follows that
\begin{equation}\label{eq:Jell-J0}
\begin{split}
|J_\ell-J_0|\le \int_0^T\int_{\mathbb R^n} |\Psi_\ell(z)|^3|F_{\tilde\xi}(t,z)-F_{\tilde\xi}(t,0)|\,dzdt\lesssim \ell.
\end{split}
\end{equation}
The implicit constant in \eqref{eq:Jell-J0} is uniform over the oriented line family, since it follows only from the fixed Lipschitz bound of the zero extension of $\tilde\xi$ and the fixed smooth extensions in the attenuation factors.

Let
\begin{equation}
t_-
\vcentcolon=
\inf\{t\in[0,T]:p-\sqrt bt\,\omega\in\overline\Omega\},
\qquad
t_+
\vcentcolon=
\sup\{t\in[0,T]:p-\sqrt bt\,\omega\in\overline\Omega\}.
\end{equation}
Since $\Omega$ is convex, the intersection of the line with $\overline\Omega$ is the segment corresponding to $[t_-,t_+]$.
Because the zero extension of $\tilde\xi$ vanishes outside $\Omega$, $J_0$ is effectively integrated only over $[t_-,t_+]$.
Set
\begin{equation}\label{eq:me-xi}
m_e(p,\omega)\vcentcolon=e^{-\frac12\int_0^{t_-}\big(\Upsilon_2+\tilde\alpha\big)(p-\sqrt b\tau\,\omega)\,d\tau}.
\end{equation}
By the uniform bounds on the chosen extensions,
\[
|m_e(p,\omega)|
+
|m_e(p,\omega)|^{-1}
\lesssim1
\]
uniformly in the relevant rays. Hence
\begin{equation}\label{eq:J0-normalized}
\begin{split}
m_e^{-1}(p,\omega)J_0=\int_{t_-}^{t_+}\tilde\xi(p-\sqrt bt\,\omega)e^{-\frac12\int_{t_-}^t\Upsilon_2(p-\sqrt b\tau\,\omega)\,d\tau}e^{-\frac12\int_{t_-}^t\tilde\alpha(p-\sqrt b\tau\,\omega)\,d\tau}dt.
\end{split}
\end{equation}
Denote by
\begin{equation}\label{eq:J1-xi}
J_1\vcentcolon=\int_{t_-}^{t_+}\tilde\xi(p-\sqrt bt\,\omega)e^{-\frac12\int_{t_-}^t\Upsilon_2(p-\sqrt b\tau\,\omega)\,d\tau} dt.
\end{equation}
We can get
\begin{equation}\label{eq:J1-J0}
|J_1|\lesssim |J_0|+\|\tilde\alpha\|_{L^\infty(\Omega)}.
\end{equation}

Combining \eqref{est:J-ell-I-xi}, \eqref{est:J-xi-delta-sigma}, \eqref{eq:Jell-J0}, and \eqref{eq:J1-J0}, and absorbing 
$\|\tilde\alpha\|_{L^\infty}$ into \eqref{eq:N-linear-coeff-xi}, we can obtain
\begin{equation}\label{eq:J1-master-xi}
\begin{split}
|J_1|&\lesssim\varepsilon^{-2}\sigma^{-\frac32}\delta\ell^{-1-2N}+\varepsilon\sigma^{3s+6}\ell^{-3s-10-8N}\\
&\quad+\sigma^{n+\frac92}\ell^{-8-6N-n}\delta^{\frac{\mathfrak m\theta_0}{n+2s+k_2}}+\sigma^{-1}\ell^{-(6N+5)}+\ell.
\end{split}
\end{equation}
We now choose
\begin{equation}\label{eq:parameters-xi}
\sigma=c_7\delta^{-(6s+12)\eta},\quad\ell=\delta^\eta,\quad\varepsilon=c_8\delta^{\frac13+(6s^2+30s+35)\eta},
\end{equation}
where
\begin{equation}\label{eq:eta-xi}
\eta\vcentcolon=\min\Bigl\{\frac{1}{3(12s^2+53s+56)},\,\frac{\mathfrak m\theta_0}{(n+2s+k_2)(6sn+13n+33s+69)}\Bigr\}.
\end{equation}
Since $N=s+1$, the relation $\sigma\sim \ell^{-(6N+6)}$ holds, and thus \eqref{eq:sigma-ell-regime} is satisfied. After
decreasing the common $\delta_0>0$ once, we also have $\ell=\delta^\eta\le\ell_*,$
so the uniform endpoint support and compatibility conclusions of Lemma~\ref{lem:uniform-GO-xi} apply to every ray used in the
argument.

Substituting these parameters into \eqref{eq:J1-master-xi}, we have $|J_1|\lesssim\delta^\eta$ uniformly with respect to $p$ and $\omega$. Indeed, the first two terms have the exponent
\[
\frac13-(12s^2+53s+55)\eta\ge\eta,
\]
the fourth term is of order $\delta^\eta$, and the third one is controlled by the second choice in \eqref{eq:eta-xi}.
Moreover, using \eqref{eq:A-s-GO-xi},
\begin{equation}
\|\varepsilon f\|_{\mathcal A^s(\Gamma_T)}\lesssim\varepsilon\sigma^{s+\frac52}\ell^{-3s-5+\frac n6}\lesssim \delta^{\frac13+\frac{n\eta}{6}}.
\end{equation}
Hence, after decreasing $\delta_0$ if necessary, both
$\varepsilon f$ and $2\varepsilon f$ belong to
$\mathcal U_{\Gamma_T}^{\epsilon_0,s}$, as required in the
second finite-difference argument.

It remains to identify $J_1$ as a weighted X-ray transform in Euclidean spaces. We still keep the notation $\partial_-S\Omega$ for the corresponding incoming ray space. Set
\begin{equation}
x_-=p-\sqrt bt_-\,\omega,\quad\tilde\omega=-\omega,\quad l=\sqrt b(t-t_-),
\end{equation}
\begin{equation}
\gamma_{x_-,\tilde\omega}(l)=x_-+l\tilde\omega,\quad \tau(x_-,\tilde\omega)=\sqrt b(t_+-t_-).
\end{equation}
Then
\begin{equation}\label{eq:J1-weighted-Xray}
J_1=b^{-\frac12}\int_0^{\tau(x_-,\tilde\omega)}\tilde\xi(\gamma_{x_-,\tilde\omega}(l))e^{-\int_0^l
\widetilde\Upsilon_2(\gamma_{x_-,\tilde\omega}(r))\,dr}dl,
\end{equation}
where $\widetilde\Upsilon_2=(4b)^{-\frac12}\Upsilon_2=\frac{1}{2\sqrt b}\Upsilon_2.$
Take a smooth compactly supported
extension of $\widetilde\Upsilon_2$ to $\mathbb R^n$ and define the
scalar weight
\begin{equation}\label{eq:weight-W-xi}
W(x,\theta)\vcentcolon=e^{-\int_{-\infty}^0\widetilde\Upsilon_2(x+r\theta)\,dr},\quad
(x,\theta)\in\mathbb R^n\times\mathbb S^{n-1}.
\end{equation}
Then $W$ is smooth and nonvanishing, with the estimate
\[
\|W\|_{C^k}+\|W^{-1}\|_{L^\infty}\lesssim 1
\]
uniformly over the admissible coefficient class, provided $K$ is chosen sufficiently large. Moreover,
\[
W(\gamma_{x_-,\tilde\omega}(l),\tilde\omega)=W(x_-,\tilde\omega)e^{-\int_0^l\widetilde\Upsilon_2(\gamma_{x_-,\tilde\omega}(r))\,dr}.
\]
Since $W(x_-,\tilde\omega)$ and its inverse are uniformly bounded,
\eqref{eq:J1-weighted-Xray} and the uniform estimate for $J_1$ imply
\begin{equation}\label{eq:IW-xi}
\|I_W\tilde\xi\|_{L^\infty(\partial_-S\Omega)}
\lesssim
\delta^\eta.
\end{equation}
Consequently,
\[
\|I_W\tilde\xi\|_{L^2(\partial_-S\Omega)}
\lesssim
\delta^\eta.
\]

Applying \cite[Theorem 5.1]{bohr2021stability} to the weighted ray transform $I_W$, the uniform bounds on $W$ and $W^{-1}$ allow us to use the exponent $\mathfrak m$ fixed in Section~\ref{Sec-proof-lin}, after decreasing
it once if necessary. Hence, we can conclude that
\begin{equation}\label{eq:xi-stability-L2}
\|\tilde\xi\|_{L^2(\Omega)}
\lesssim
\delta^{\mathfrak m\eta}.
\end{equation}
Using the interpolation inequality between $L^2(\Omega)$ and the fixed
a priori $H^{s_*}(\Omega)$ bound, where $s_*>n/2$, there is
$\theta_4\in(0,1)$ such that
\begin{equation}\label{eq:xi-stability-Linf}
\|\tilde\xi\|_{L^\infty(\Omega)}
\lesssim
\delta^{\mathfrak m\eta\theta_4}.
\end{equation}
Combining this with \eqref{est:for-alp-lam-q-DN}, we finally obtain
\begin{equation}\label{est:a-lam-q-xi-proof}
\|\alpha_2-\alpha_1\|_{L^\infty(\Omega)}+\|\lambda_2-\lambda_1\|_{L^\infty(\Omega)}+\|q_2-q_1\|_{L^\infty(\Omega)}+
\|\xi_2-\xi_1\|_{L^\infty(\Omega)}\lesssim\delta^{\widetilde\mu_0},
\end{equation}
where
\begin{equation}
\widetilde\mu_0\vcentcolon=\min\Bigl\{\frac{\mathfrak m\theta_0}{n+2s+k_2},\mathfrak m\eta\theta_4\Bigr\}\in(0,1).
\end{equation}
Hence, the proof of Theorem~\ref{thm:stable-alp-lam-q-xi} is complete. \hfill $\square$


\section{Conclusions}\label{sec:conclusion}
In this paper, we are concerned with the stable determination of the space-varying coefficients $\alpha,\lambda,q,\xi$ appearing in the JMGT equation from the knowledge of the DN map $\Lambda_{\alpha,\lambda,q,\xi}$. The proof of the main theorems combine the linearization method with suitable Gaussian beam and GO solutions for linearized MGT equations. The classical foliation condition and the geodesic ray transforms play an essential role in establishing the stability result. 

Throughout this paper, we have assumed that $b$ is a known function, so it is more challenging to establish the stability result of determining $b$ for the (J)MGT equation. It is also of great interest to study the stable determination of the coefficients when only partial boundary measurements are available.

\appendix
\section{Proof of Lemmas \ref{lem:uniform-ray-family} and \ref{lem:uniform-GO-xi}}

\emph{Proof of Lemma \ref{lem:uniform-ray-family}.}\quad We first establish the exterior estimate used in {\rm (i)} and {\rm (iii)}. Set
\[
D_g\vcentcolon={\rm diam}_g(\Omega)<T,
\qquad
d_T\vcentcolon=T-D_g>0.
\]
Let $(\Omega_e,g_e)$ be a fixed complete extension of the metric extension used in Section~\ref{sec:Gaussian-beams}, with
$\overline\Omega\Subset\Omega_e$. The fixed coefficient extensions are taken
on the same spatial extension. Choose \(r_c>0\) sufficiently small so that the two-sided tubular
neighborhood
\[
\mathcal C_{r_c}\vcentcolon=\bigl\{x\in\Omega_e:\operatorname{dist}_{g_e}(x,\partial\Omega)<r_c\bigr\}
\]
is well defined and the nearest-point projection onto
\(\partial\Omega\) is smooth. Define the signed distance function
\(r\in C^\infty(\mathcal C_{r_c})\) by
\[
r(x)=\begin{cases}
\operatorname{dist}_{g_e}(x,\partial\Omega),& x\in \Omega\cap\mathcal C_{r_c},
\\[1mm]
-\operatorname{dist}_{g_e}(x,\partial\Omega),
& x\in (\Omega_e\setminus\overline\Omega)\cap\mathcal C_{r_c}.
\end{cases}
\]
Thus $r>0$ in $\Omega$, $r=0$ on $\partial\Omega$ and $r<0$ outside $\overline\Omega$. Moreover, $\nabla_{g_e}r=-\nu$ on $\partial\Omega.$

By the strict convexity of $\partial\Omega$, the smoothness of the signed distance function $r$ in the collar neighborhood, and the compactness of $\partial\Omega$, after decreasing $r_c>0$ if necessary, there exist constants $\eta_0>0$, $c_0>0$, and $B\ge1$ such that, whenever $|r|\le r_c$ and $|v|_{g_e}=1$, it holds that
\begin{equation}\label{eq:two-sided-collar}
|\nabla_{g_e}^2r(v,v)|\le B,
\qquad
\nabla_{g_e}^2r(v,v)\le-c_0
\quad\text{if }|dr(v)|\le\eta_0.
\end{equation}
Let $\gamma_{\rm out}(u)$ denote the unit-speed continuation of $\gamma$ from either physical endpoint in the direction pointing outside $\Omega$. At the exit endpoint, this is the continuation of the original orientation of $\gamma$, whereas at the incoming endpoint it is the continuation of the reversed geodesic. Put
\[
a=\iota_-\quad\text{or}\quad a=\iota_+,
\qquad
q(u)=r(\gamma_{\rm out}(u)).
\]
Then
\[
q(0)=0,\qquad q'(0)=-a,\qquad 0\le a\le1.
\]
Choose $\ell_*>0$ so small that
\begin{equation}\label{eq:ell-star-uniform}
2\ell_*<r_c,\qquad
2B\ell_*\le\frac{\eta_0}{4},\qquad
\ell_*<\frac{d_T}{8},\qquad
\ell_*<1.
\end{equation}
For $0\le u\le2\ell_*$, the continued geodesic remains in the
two-sided collar. If $a\le\eta_0/2$, then
\[
|q'(u)|\le a+2B\ell_*\le\frac{3\eta_0}{4},
\]
and hence the strict convexity estimate in
\eqref{eq:two-sided-collar} applies to this interval, giving
\[
q(u)\le-au-\frac{c_0}{2}u^2.
\]
If $a\ge\eta_0/2$, then
\[
q'(u)\le-a+2B\ell_*\le-\frac a2,
\]
so $q(u)\le-\frac a2u.$
Therefore there is a uniform constant $c_e>0$ such that
\begin{equation}\label{eq:uniform-exterior-clearance}
{\rm dist}_{g_e}(\gamma_{\rm out}(u),\Omega)=-q(u)\ge c_e(au+u^2),\qquad 0\le u\le2\ell_*.
\end{equation}
In particular, if $\ell_*/2\le u\le\ell_*$, then
\[
\operatorname{dist}_{g_e}(\gamma_{\rm out}(u),\Omega)\ge c_eu^2\ge\frac{c_e\ell_*^2}{4}.
\]
Hence the prolonged central geodesic is separated from $\Omega$ by a uniform positive distance on this interval, independently of the incidence parameter $a$, including in the grazing limit $a\to0$.

\medskip
\emph{Proof of {\rm (i)}.}
For a maximal unit-speed geodesic
$\gamma=\gamma_{p_0,\zeta}$, write $\tau=\tau(p_0,\zeta)\le D_g$ and translate the associated null geodesic in time so that $t_-=(T-\tau)/2,$ $t_+=(T+\tau)/2.$
Then
\begin{equation}\label{eq:uniform-time-margin}
t_-\ge \frac{d_T}{2},
\qquad
T-t_+\ge \frac{d_T}{2}.
\end{equation}
By \eqref{eq:ell-star-uniform}, the central prolongation of length
$\ell_*$ at both ends remains a fixed positive distance from the time
faces. For the Fermi-coordinate estimates, consider the compact parameter set
\[
\mathcal P=\overline{\partial_-S\Omega}\times[-2\ell_*,D_g+2\ell_*]
\]
together with the corresponding compact orthonormal frame bundle.
The completeness of $(\Omega_e,g_e)$ guarantees that every spatial geodesic in the family can be prolonged for the uniform amount needed in the construction. Since the corresponding family of initial data and orthonormal frames is compact, it can be covered by finitely many smooth frame trivializations. Moreover, the metric and all coefficients
entering the Gaussian beam construction are independent of $t$. Therefore, the ray-dependent translation in the time variable used to place each null geodesic inside $(0,T)$ does not affect the constants
in the Fermi coordinate construction, the Riccati and transport equations, or the Sobolev estimates.

Let $\mathcal F_\gamma=\Phi_\gamma^{-1}$ be the inverse Fermi coordinate map. Locally one may write
\[
t=t_0+\frac{s-z_1}{\sqrt2},\qquad x(s,0)=\gamma\left(\frac{s-s_-}{\sqrt2}\right).
\]
Since $D\mathcal F_\gamma(s,0)$ is invertible along the central
geodesic, the inverse-function theorem gives a local Fermi coordinate
chart near every point $(s,0)$. By the compactness of the parameter family and the quantitative inverse-function theorem, these local neighborhoods can be chosen with uniform sizes. More precisely, there exist $r_0>0$, $\rho_0>0$, and $C_K>0$, independent of $\gamma$, such that for every center point $s_0$, the map $\mathcal F_\gamma$ is a
diffeomorphism on
\[
(s_0-r_0,s_0+r_0)\times B(0,\rho_0),
\]
and
\begin{equation}\label{eq:uniform-Fermi-local}
\|\mathcal F_\gamma\|_{C^K}+\|\mathcal F_\gamma^{-1}\|_{C^K}\le C_K.
\end{equation}

We claim that there is a common $\rho_*>0$ for which $\mathcal F_\gamma$ is injective on every complete prolonged tube. Otherwise, for each $j\ge1$ one could choose a ray $\gamma_j$ and two distinct points
$(s_j,z'_j)$, $(\widetilde s_j,\widetilde z'_j)$ with
$|z'_j|,|\widetilde z'_j|<1/j$ such that
\[
\mathcal F_{\gamma_j}(s_j,z'_j)
=
\mathcal F_{\gamma_j}(\widetilde s_j,\widetilde z'_j).
\]
Equality of the time coordinates yields
\[
|s_j-\widetilde s_j|=|z_{1,j}-\widetilde z_{1,j}|\rightarrow0.
\]
For all sufficiently large $j$, the two points are contained in a common local Fermi neighborhood of uniform size on which
$\mathcal F_{\gamma_j}$ is injective, which is a contradiction.

By \eqref{eq:uniform-exterior-clearance}, the prolonged central
geodesic is separated from $\Omega$ by a uniform positive distance on
the end regions $\ell_*/2\le u\le\ell_*$. Together with the uniform
local bounds in \eqref{eq:uniform-Fermi-local}, this allows us, after
decreasing the transverse radius $\rho_*>0$ if necessary, to ensure
that the corresponding transverse end regions of the Fermi tube are
disjoint from $\overline{\Omega_T}$.
Moreover, the relation
\[
t=t_0+\frac{s-z_1}{\sqrt2},
\]
combined with \eqref{eq:uniform-time-margin}, allows $\rho_*$ to be
chosen uniformly so that the whole effective tube remains a positive
distance from the time faces $\{t=0\}$ and $\{t=T\}$. Hence the
longitudinal cutoff may be chosen equal to one on the portion of the
tube intersecting $\overline\Omega_T$, while
\[
\operatorname{dist}\left(\operatorname{supp}\partial_s\chi_\gamma,\overline\Omega_T\right)\ge \epsilon_*
\]
for some $\epsilon_*>0$ independent of $\gamma$. Finally, since the longitudinal length of the prolonged tubes is
uniformly bounded, each tube can be covered by at most $N_0$ local
Fermi coordinate neighborhoods of fixed size. A subordinate partition
of unity can therefore be chosen with the uniform bounds stated in
\eqref{eq:uniform-Fermi-bounds}.

Fix the same initial matrix $H_0$ with $\Im H_0>0$ for every ray.
The matrices $Y,Z$ in the Riccati linearization depend continuously on
the compact parameter set and satisfy $\det Y\neq0$. Consequently,
\begin{equation}\label{eq:uniform-Y}
\|Y\|+\|Y^{-1}\|+\|Z\|\le C.
\end{equation}
Using
\[
H=ZY^{-1},\quad \Im H=(Y^{-1})^*(\Im H_0)Y^{-1},
\]
we obtain
\begin{equation}\label{eq:uniform-ImH}
cI\le\Im H(s)\le CI,
\end{equation}
as well as
\begin{equation}\label{eq:uniform-det-HY}
\det\Im H(s)\,|\det Y(s)|^2
=
\det\Im H_0.
\end{equation}
The higher-order coefficients in the Taylor expansion of the phase,
as well as the successive transport coefficients of the amplitudes,
satisfy a finite system of linear ODEs along the central geodesic.
By the uniform bounds established above for the Fermi coordinates,
the metric coefficients, and the Riccati solution, the coefficients
of these ODEs are uniformly bounded over the whole geodesic family.
An induction on the construction order, together with Gr\"onwall's
inequality, therefore yields uniform bounds for all phase and amplitude
coefficients involved in the fixed finite-order Gaussian beam
construction. Moreover, along the central geodesic one has
\[
\varphi_t(s,0)=-\frac1{\sqrt2},
\]
while the uniform positivity of $\Im H(s)$ gives
\[
\Im\varphi(s,z')
=
\frac12\langle \Im H(s)z',z'\rangle
+O(|z'|^3).
\]
Hence, after decreasing the uniform transverse radius $\rho_*>0$ if
necessary, there exists $c>0$, independent of the geodesic, such that
\begin{equation}\label{eq:uniform-phase-positive}
\Im\varphi(s,z')\ge c|z'|^2,
\qquad
|\varphi_t(s,z')|\ge c,
\qquad
|z'|\le\rho_*.
\end{equation}

The uniform bounds established above for the phase and amplitude coefficients imply that all residual estimates arising from the finite-order eikonal and transport construction hold with constants independent of the underlying ray. The errors generated by the transverse cutoffs are exponentially small in \(\sigma\), whereas the derivatives of the longitudinal cutoffs are supported in the prolonged end regions lying outside \(\overline\Omega_T\), and hence do not contribute to the residual inside the cylinder. The forward and adjoint source estimates on the fixed domain \(\Omega_T\) therefore yield the corresponding exact correction estimates with constants uniform over the geodesic family.
The boundary trace estimates are obtained from the standard integer-order trace theorem on the fixed cylinder. More precisely,
$$
\|f_\sigma\|_{H^l(\Gamma_T)}\le C_{\rm tr}\|v_{\sigma}\|_{H^{l+1}(\Omega_T)},\qquad l=0,1,\ldots,
$$
and the same estimate holds for the adjoint Gaussian beam. Since the trace constant \(C_{\rm tr}\) depends only on the fixed domain and the relevant Sobolev order, it is independent of the geodesic parameters \((p_0,\zeta)\). Consequently, no loss depending on the incidence angle occurs in these boundary estimates. This completes the proof of {\rm (i)}.

\medskip
\emph{Proof of {\rm (ii)}.}
We work in the interior part of the collar neighborhood introduced above. Let $\gamma=\gamma_{p_0,\zeta}$
be a geodesic entering \(\Omega\) at \(u=0\), and set
$$
a\vcentcolon=\iota_->0,
\qquad
q(u)\vcentcolon=r(\gamma(u)).
$$
By the definition of the signed distance function and of the incoming
incidence parameter,
$$
q(0)=0,
\qquad
q'(0)=a.
$$
Fix \(A>2/c_0\). We then choose \(a_*>0\) sufficiently small so that
$$
Aa_*<\frac{r_c}{2},
\qquad
(1+AB)a_*<\eta_0.
$$
Suppose that \(0<a\le a_*\). As long as the geodesic remains in
\(\Omega\) and \(0\le u\le Aa\), the bound
$$
|\nabla_{g_e}^2r|\le B
$$
implies
$$
|q'(u)-a|
\le Bu.
$$
Consequently,
$$
|q'(u)|
\le a+BAa
=(1+AB)a
\le\eta_0,
$$
and hence the strict concavity estimate in
\eqref{eq:two-sided-collar} is applicable throughout this interval.
Therefore,
$$
-B\le q''(u)\le-c_0,
\qquad
0\le u\le\min\{\tau,Aa\}.
$$
Integrating twice and using \(q(0)=0\) and \(q'(0)=a\), we obtain
\begin{equation}\label{eq:tau-angle-local}
au-\frac B2u^2\le q(u)\le au-\frac{c_0}{2}u^2,\qquad 0\le u\le\min\{\tau,Aa\}.
\end{equation}
Since \(A>2/c_0\), the upper bound in
\eqref{eq:tau-angle-local} is nonpositive at \(u=Aa\). Hence the first
exit time satisfies $\tau\le Aa.$
On the other hand, the lower bound in
\eqref{eq:tau-angle-local} is strictly positive whenever $0<u<\frac{2a}{B}.$
Thus the geodesic cannot return to \(\partial\Omega\) before a time
comparable to \(a\). It follows that there exist constants
\(c,C>0\), independent of the ray, such that
$$
c\,\iota_-
\le
\tau(p_0,\zeta)
\le
C\,\iota_-,
\qquad
0<\iota_-\le a_*.
$$
Applying the same argument to the reversed geodesic at the exit point
gives
\begin{equation}\label{eq:tau-angle-exit}
c\,\iota_+
\le
\tau(p_0,\zeta)
\le
C\,\iota_+,
\qquad
0<\iota_+\le a_*.
\end{equation}
After fixing \(0<\kappa_0<a_*\), the estimate
\eqref{eq:grazing-length} follows.

We next establish the uniform lower bound in
\eqref{eq:good-ray-lower-length}. Since $|q''|\le B$ in the collar,
choose $c_1>0$ such that $c_1\le \frac{r_c}{2}$ and $Bc_1\le 1.$
We claim that
\[
\tau(p_0,\zeta)\ge c_1\iota_-.
\]
Indeed, writing $a=\iota_-$, suppose for contradiction that $\tau<c_1a$. Since $a\le1$, one has
\[
\tau<c_1a\le c_1\le \frac{r_c}{2},
\]
so the geodesic segment up to the first exit remains inside the collar. The bound $|q''|\le B$, together with
$q(0)=0$ and $q'(0)=a$, therefore yields
\[
q(u)\ge au-\frac B2u^2,
\qquad 0\le u\le\tau.
\]
Evaluating at the first exit time and using $q(\tau)=0$, we obtain
\[
0=q(\tau)\ge a\tau-\frac B2\tau^2=a\tau\Bigl(1-\frac{B\tau}{2a}\Bigr).
\]
Since $\tau<c_1a$ and $Bc_1\le1$, $q(\tau)\ge \frac12a\tau>0$, which is a contradiction. Hence $\tau(p_0,\zeta)\ge c_1\iota_-.$
In particular, if $\iota_-\ge\kappa$, then $\tau(p_0,\zeta)\ge c_1\kappa,$ which proves \eqref{eq:good-ray-lower-length}.

It remains to establish \eqref{eq:grazing-measure}. Recall that
$$
\iota_-(p_0,\zeta)=-\langle\zeta,\nu(p_0)\rangle_g.
$$
For each \(p_0\in\partial\Omega\), the grazing directions form the codimension-one submanifold
$$
\{\zeta\in S_{p_0}\Omega:\langle\zeta,\nu(p_0)\rangle_g=0\},
$$
which is the equatorial sphere orthogonal to $\nu(p_0)$.
In a neighborhood of this equator, \(\iota_-\) is a smooth defining function with nonvanishing differential. Hence, by compactness of \(\partial\Omega\) and smoothness of the Riemannian density on \(\partial_-S\Omega\),
\begin{equation}\label{eq:incoming-strip-measure}
{\rm meas}_{\partial_-S\Omega}
\bigl(\{(p_0,\zeta)\in\partial_-S\Omega:
\iota_-(p_0,\zeta)<\kappa\}\bigr)\le C\kappa,\qquad 0<\kappa<\kappa_0.
\end{equation}

We next relate a small exit incidence angle to the incoming one.
Suppose that \(\iota_+<\kappa\). After decreasing \(\kappa_0\), if necessary, so that \(\kappa_0<a_*\), estimate
\eqref{eq:tau-angle-exit} yields
$$
\tau(p_0,\zeta)\le C\iota_+\le C\kappa.
$$
Choosing \(\kappa_0\) sufficiently small, the corresponding geodesic
segment is contained in the fixed collar neighborhood of
\(\partial\Omega\). Extend the outward unit normal field \(\nu\)
smoothly to this collar, and denote by
\(P_{p_+\to p_-}\) parallel transport along \(\gamma\) from the exit
point \(p_+\) to the entry point \(p_-\). Since the velocity field of
a geodesic is parallel along \(\gamma\),
$$
P_{p_+\to p_-}\zeta_+=\zeta.
$$
The smoothness of the extended normal field and the uniform bounds in
the collar therefore imply
$$
\left|\nu(p_-)-P_{p_+\to p_-}\nu(p_+)\right|_g\le C\tau \le C\kappa.
$$
Consequently,
$$
\begin{aligned}
\iota_-+\iota_+
&=\left|
\langle\zeta,\nu(p_-)\rangle_g
-\langle\zeta_+,\nu(p_+)\rangle_g
\right|\\
&=\left|
\left\langle
\zeta,\nu(p_-)-P_{p_+\to p_-}\nu(p_+)
\right\rangle_g
\right|
\le C\kappa.
\end{aligned}
$$
In particular,
$$
\{\iota_+<\kappa\}
\subset
\{\iota_-<C\kappa\}.
$$
Since
$$
\mathcal G_\kappa
=
\{\iota_-<\kappa\}\cup\{\iota_+<\kappa\},
$$
the preceding inclusion and
\eqref{eq:incoming-strip-measure} give
$$
{\rm meas}_{\partial_-S\Omega}(\mathcal G_\kappa)
\le C\kappa,
$$
which proves \eqref{eq:grazing-measure}.

\medskip
\emph{Proof of {\rm (iii)}.}
Fix \((p_0,\zeta)\in\mathcal M_\kappa\), and introduce the \(s\)-parametrization of the corresponding spatial geodesic by
$$\gamma_s(s)\vcentcolon=\gamma_{p_0,\zeta}\Bigl(\frac{s-s_-}{\sqrt2}\Bigr).$$
Then $|\dot\gamma_s(s)|_g=2^{-1/2}.$
Define $h\vcentcolon=\kappa^{-1}\sigma^{-1/4}.$
By enlarging the constant in
\eqref{eq:sigma-kappa-condition}, if necessary, we may assume that $h\le c\kappa$,
where \(c>0\) is a sufficiently small constant independent of \((p_0,\zeta)\), \(\kappa\), and \(\sigma\). It follows from
\eqref{eq:good-ray-lower-length} that the two endpoint subintervals
$$E_{\rm in}=[s_-,s_-+h]\cup[s_+-h,s_+]$$
are disjoint. We further set
$$ I_{\rm mid}=[s_-+h,s_+-h],\qquad J_{\rm ext}=J_\gamma\setminus[s_-,s_+].$$
At the incoming endpoint, the definition of the incidence parameter yields
$$ (r\circ\gamma_s)'(s_-)=\frac{\iota_-}{\sqrt2}\ge \frac{\kappa}{\sqrt2}.$$
Using the uniform bound for the second derivative of
\(r\circ\gamma_s\), we obtain
$$ r(\gamma_s(s_-+h)) \ge \frac{\kappa h}{\sqrt2}-Ch^2.$$
After increasing the constant in \eqref{eq:sigma-kappa-condition}, if necessary, we may further require $ Ch\le\frac{\kappa}{2\sqrt2}.$
Consequently,
\begin{equation}\label{eq:entry-depth} 
r(\gamma_s(s_-+h))
\ge c\kappa h
= c\sigma^{-1/4}.
\end{equation}
Applying the same argument to the reversed geodesic at the exit endpoint gives
\begin{equation}\label{eq:exit-depth}
r(\gamma_s(s_+-h))\ge c\sigma^{-1/4}.
\end{equation}

We first extend the endpoint depth estimates
\eqref{eq:entry-depth}--\eqref{eq:exit-depth} to the entire middle
interval \(I_{\rm mid}\). Observe that \(r\circ\gamma_s\) cannot attain
a local minimum in the collar neighborhood. Indeed, at any critical
point \(s_0\) of \(r\circ\gamma_s\), one has
$$ dr(\dot\gamma_s(s_0))=0,$$
so that \(\dot\gamma_s(s_0)\) is tangent to the corresponding level
set of \(r\). Since \(|\dot\gamma_s|_g=2^{-1/2}\),
\eqref{eq:two-sided-collar} yields
$$(r\circ\gamma_s)''(s_0)=\nabla_g^2r\bigl(\dot\gamma_s(s_0),\dot\gamma_s(s_0)\bigr)\le-\frac{c_0}{2}<0.$$
Thus every critical point of \(r\circ\gamma_s\) lying in the collar is
a strict local maximum. Combining this observation with \eqref{eq:entry-depth} and \eqref{eq:exit-depth}, and using the fixed
positive distance from \(\partial\Omega\) on the complement of a smaller collar, we obtain, after increasing \(\sigma_*\) if necessary,
\begin{equation}\label{eq:distance-middle-ray}
{\rm dist}_g(\gamma_s(s),\partial\Omega)\ge c\sigma^{-1/4}, \qquad s\in I_{\rm mid}.
\end{equation}
The uniform \(C^1\)-bounds for the inverse Fermi coordinate maps imply
$$ d_g\bigl(x(s,z'),\gamma_s(s)\bigr)\le C|z'|. $$
By part {\rm (i)}, the Fermi tube over \(I_{\rm mid}\) remains uniformly
separated from the time faces \(t=0\) and \(t=T\). Hence, if \(\mathcal F_\gamma(s,z')\notin\Omega_T\) for some
\(s\in I_{\rm mid}\), the corresponding spatial point \(x(s,z')\) must lie outside \(\Omega\). It then follows from
\eqref{eq:distance-middle-ray} and the preceding \(C^1\)-estimate that
$|z'|\ge c\sigma^{-1/4}$.
Using \(\Im\varphi(s,z')\ge c|z'|^2\), the normalized Gaussian mass of the portion omitted by the domain is therefore bounded by \begin{equation}\label{eq:middle-gaussian-tail}
\sigma^{n/2}\int_{|z'|\ge c\sigma^{-1/4}} e^{-c\sigma|z'|^2}\,dz'\le Ce^{-c\sqrt{\sigma}}.
\end{equation}
Consequently, over \(I_{\rm mid}\), the integration over the transverse sections may be replaced by integration over the full transverse sections, at an error of order \(O(e^{-c\sqrt{\sigma}})\).

On each full transverse section, the uniform positivity of
\(\Im H(s)\), together with the uniform finite-order bounds for
\(A_\gamma\) and \(\varphi\), permits a parameter-uniform Laplace
expansion:
$$\sigma^{n/2}\int_{\mathbb R^n}A_\gamma(s,z')e^{-2\sigma\Im\varphi(s,z')}\,dz'=\left(\frac{\pi}{2}\right)^{n/2}
\frac{A_\gamma(s,0)}{\sqrt{\det\Im H(s)}}+O(\sigma^{-1}).$$
Indeed, after the rescaling \(z'=\sigma^{-1/2}w\), the terms of order
\(\sigma^{-1/2}\) arising from the Taylor expansions of the amplitude
and the phase are odd in \(w\), and hence their integrals against the
centered quadratic Gaussian vanish. The remaining terms are of order
\(O(\sigma^{-1})\), uniformly with respect to the geodesic parameters
and the longitudinal variable \(s\), by the finite-order estimates
established in {\rm (i)}.
If several local Fermi charts are required, we apply the expansion
after inserting the uniformly bounded subordinate partition of unity.
Since the partition functions sum to one along the central geodesic
and at most \(N_0\) charts are involved, both the leading term and the
uniform \(O(\sigma^{-1})\) remainder are preserved. 

We denote by
\begin{equation}\label{eq:middle-G-def}
\mathfrak G_{\sigma,\gamma}^{\,{\rm mid}}
\vcentcolon=
\sigma^{n/2}
\int_{I_{\rm mid}}
\int_{\{z'\in B(0,\rho_*):
\,\mathcal F_\gamma(s,z')\in\Omega_T\}}
A_\gamma(s,z')
e^{-2\sigma\Im\varphi(s,z')}
\,dz'\,ds
\end{equation}
the contribution of the middle longitudinal interval
\(I_{\rm mid}\) to \(\mathfrak G_{\sigma,\gamma}\). Correspondingly,
we define
\begin{equation}\label{eq:lead-middle-G-def}
\mathfrak G_{\gamma}^{\,{\rm lead,mid}}
\vcentcolon=
\left(\frac{\pi}{2}\right)^{n/2}
\int_{I_{\rm mid}}
\frac{A_\gamma(s,0)}
{\sqrt{\det\Im H(s)}}
\,ds.
\end{equation}
Hence, we can get
\begin{equation}\label{eq:middle-SP-error}
\left|\mathfrak G_{\sigma,\gamma}^{\,{\rm mid}}-\mathfrak G_{\gamma}^{\,{\rm lead,mid}}\right|\le C\sigma^{-1}.
\end{equation}

Denote by
$$\mathfrak G_{\sigma,\gamma}^{\,{\rm end,in}}\vcentcolon=\sigma^{n/2}\int_{E_{\rm in}}
\int_{\{z'\in B(0,\rho_*):\,\mathcal F_\gamma(s,z')\in\Omega_T\}}A_\gamma(s,z')
e^{-2\sigma\Im\varphi(s,z')}\,dz'ds$$
the contribution of the physical endpoint layers to
\(\mathfrak G_{\sigma,\gamma}\). Correspondingly, we define
$$\mathfrak G_{\gamma}^{\,{\rm lead,end}}\vcentcolon=\left(\frac{\pi}{2}\right)^{n/2}\int_{E_{\rm in}}
\frac{A_\gamma(s,0)}{\sqrt{\det\Im H(s)}}\,ds.
$$
The contributions arising from the two endpoint layers are
estimated directly. Since the normalized transverse Gaussian mass is
uniformly bounded, and the corresponding leading ray density is
uniformly bounded as well, the fact that \(E_{\rm in}\) has total
length \(2h\) implies
\begin{equation}\label{eq:endpoint-error}
\left|\mathfrak G_{\sigma,\gamma}^{\,{\rm end,in}}\right|+\left|\mathfrak G_{\gamma}^{\,{\rm lead,end}}\right|\le
Ch=C\kappa^{-1}\sigma^{-1/4}.
\end{equation}
It remains to estimate the contribution from the prolonged portions
of the tube lying beyond the physical endpoints of the central
geodesic. Although the center of such a transverse section lies
outside \(\Omega\), its Gaussian tail may still intersect the cylinder \(\Omega_T\). We denote the resulting contribution over
\(s\in J_{\rm ext}\) by \(\mathfrak G_{\sigma,\gamma}^{\,{\rm ext}}\).
For \(s<s_-\), set $u=\frac{s_--s}{\sqrt2}$, whereas for \(s>s_+\), set $u=\frac{s-s_+}{\sqrt2}$.

By the uniform estimate \eqref{eq:uniform-exterior-clearance} and the condition
\(\iota_\pm\ge\kappa\), one has
$$\operatorname{dist}_{g_e}\bigl(\gamma_s(s),\Omega\bigr)\ge c\kappa u.$$
On the other hand, the uniform \(C^1\)-bounds for the Fermi coordinate maps imply that any point \(x(s,z')\in\Omega\) lying in such a transverse section must satisfy $|z'|\ge c\kappa u.$ Using
\(\Im\varphi(s,z')\ge c|z'|^2\), we can obtain
$$\sigma^{n/2}\int_{|z'|\ge c\kappa u}e^{-c\sigma|z'|^2}\,dz'\le Ce^{-c\sigma\kappa^2u^2}.$$
Since the prolonged exterior pieces have uniformly bounded longitudinal length and the principal density \(A_\gamma\) is uniformly bounded, it follows that
\begin{equation}\label{eq:exterior-physical-error}
\left|\mathfrak G_{\sigma,\gamma}^{\,{\rm ext}}\right|\le C\int_0^{\ell_*}e^{-c\sigma\kappa^2u^2}\,du\le
C\kappa^{-1}\sigma^{-1/2}.
\end{equation}
Here the estimate is performed over the actual integration region and relies only on the fixed smooth extensions of the
coefficients. Combining the middle, endpoint, and exterior portions, we have
$$\mathfrak G_{\sigma,\gamma}=\mathfrak G_{\sigma,\gamma}^{\,{\rm mid}}+\mathfrak G_{\sigma,\gamma}^{\,{\rm end,in}}
+\mathfrak G_{\sigma,\gamma}^{\,{\rm ext}},$$
whereas the leading ray integral decomposes as
$$\mathfrak G_{\gamma}^{\,{\rm lead}}=\mathfrak G_{\gamma}^{\,{\rm lead,mid}}+\mathfrak G_{\gamma}^{\,{\rm lead,end}}.
$$
Consequently, by \eqref{eq:middle-SP-error}, \eqref{eq:endpoint-error}, and \eqref{eq:exterior-physical-error},
$$\left|\mathfrak G_{\sigma,\gamma}-\mathfrak G_{\gamma}^{\,{\rm lead}}\right|\le
C\sigma^{-1}+C\kappa^{-1}\sigma^{-1/4}+C\kappa^{-1}\sigma^{-1/2}.$$
Since \(\sigma\ge1\), the last term is bounded by
\(C\kappa^{-1}\sigma^{-1/4}\). Hence
$$
\left|\mathfrak G_{\sigma,\gamma}-\mathfrak G_{\gamma}^{\,{\rm lead}}\right|\le C\bigl(
\sigma^{-1}+\kappa^{-1}\sigma^{-1/4}\bigr),
$$
which is precisely \eqref{eq:uniform-SP-error}. This completes the
proof of {\rm (iii)}.    \hfill $\square$
\medskip

\emph{Proof of Lemma \ref{lem:uniform-GO-xi}.}\quad 
Let $\overline{\partial_-S_E\Omega}$ denote the compact closure of the
incoming Euclidean unit sphere bundle, including tangential directions. Strict convexity implies that the length
$\tau_E(x_-,\theta)$ extends continuously to this compact set, with $\tau_E=0$ on tangential directions. Hence $t_\pm$ and the maps
\[
p_0(x_-,\theta)=x_- -\sqrt b\,t_-(x_-,\theta)\theta,\quad p_T(x_-,\theta)=p_0(x_-,\theta)+\sqrt b\,T\,\theta
\]
are continuous on $\overline{\partial_-S_E\Omega}$.

For an incoming ray, $p_0$ lies strictly before the entry point, whereas $p_T$ lies strictly beyond the exit point. If
$(x_-,\theta)$ is tangential, the strict convexity of $\partial\Omega$ implies that the line
$x_-+\mathbb R\theta=\{x_-+r\theta: r\in\mathbb R\}$ intersects $\overline\Omega$ only at $x_-$.
Moreover, \eqref{eq:uniform-time-margin-GO-xi} yields uniform positive
lower bounds for $t_-$ and $T-t_+$. Hence, for every
$(x_-,\theta)\in\overline{\partial_-S_E\Omega}$,
\[
p_0,p_T\in\mathbb R^n\setminus\overline\Omega.
\]
The functions
\[
(x_-,\theta)\mapsto{\rm dist}\bigl(p_0(x_-,\theta),\overline\Omega\bigr),
\quad
(x_-,\theta)\mapsto {\rm dist}\bigl(p_T(x_-,\theta),\overline\Omega\bigr)
\]
are continuous and strictly positive on the compact set $\overline{\partial_-S_E\Omega}$. Their common positive lower bound
defines $d_*>0$, and \eqref{eq:endpoint-center-separation-xi} follows.

Let
\[
R_\Psi=\sup\{|z|:z\in\supp\Psi\}.
\]
Choose
\[
0<\ell_*
<
\frac{d_*}{2\max\{1,R_\Psi\}}.
\]
Since $\supp\Psi_\ell\subset B(0,R_\Psi\ell),$ we obtain \eqref{eq:uniform-endpoint-support-GO-xi}. The transport
recursion preserves the translating support structure of the amplitudes. Therefore, the required initial compatibility conditions for the forward construction and terminal compatibility conditions for the adjoint construction hold uniformly for all incoming rays.

The trajectories
\[
p_0(x_-,\theta)+\sqrt bt\,\theta,
\qquad
0\le t\le T,
\]
remain in a fixed compact subset of $\mathbb R^n$. Since the
coefficient extensions are fixed with uniformly bounded $C^K$ norms
and the transport integrals are taken over intervals of length at most
$T$, repeated differentiation of the transport equations gives
\[
\|\partial_{t,x}^{\beta}a_k\|_{L^\infty}
+
\|\partial_{t,x}^{\beta}d_k\|_{L^\infty}
\le
C_{\beta,k}
\ell^{-\frac n3-|\beta|-2k},
\]
with constants independent of $(x_-,\theta)$. The corresponding error
terms of the truncated forward and adjoint GO constructions satisfy the
same uniform estimates.

The energy estimates used to construct the exact corrections depend
only on the fixed cylinder $\Omega_T$, the coefficient $b$, the time
interval $[0,T]$, and the prescribed bounds for the coefficients.
Hence \eqref{est:remainder-r-sig-GO} holds uniformly for both exact
remainders. The boundary estimates follow from the trace theorem on
the fixed cylinder. The boundary trace estimates are obtained directly on the fixed cylinder
\(\Omega_T\), with constants independent of the incoming ray. Hence
\eqref{eq:GO-boundary-xi} and \eqref{eq:A-s-GO-xi} hold uniformly for
\((x_-,\theta)\in\partial_-S_E\Omega\). The remaining estimates follow
from the corresponding uniform bounds for the amplitudes and the exact
remainders.     \hfill $\square$




\medskip
\noindent {\bf Acknowledgments.} 
S.~Fu was supported by the Fundamental Research Funds for the Central Universities, NPU, under grant number D5000250416, and the Natural Science Basic Research Program of Shaanxi under grant number 2026JC-YBQN-0003.  T. Zhou is partially supported by the National Key Research and Development Program of China under grant number 2024YFA1012301, the Zhejiang Provincial Basic Public Welfare Research Program under grant number LDQ24A010001, the National Natural Science Foundation of China (NSFC) under grant number 12371426.

\noindent {\bf Statements and Declarations.}

\noindent {\bf Data availability statement.}
No datasets were generated or analyzed during the current study.
	
\noindent {\bf Conflict of Interests.} The authors declare that they have no conflicts of interest.
		
\bibliography{refs} 

@article{acosta2022nonlinear,
  title={Nonlinear ultrasound imaging modeled by a {Westervelt} equation},
  author={Acosta, Sebastian and Uhlmann, Gunther and Zhai, Jian},
  journal={SIAM Journal on Applied Mathematics},
  volume={82},
  number={2},
  pages={408--426},
  year={2022},
  publisher={SIAM}
}

@article{Aicha_2015,
doi = {10.1088/0266-5611/31/12/125010},
year = {2015},
publisher = {IOP Publishing},
volume = {31},
number = {12},
pages = {125010},
author = {A\"{i}cha, Ibtissem Ben},
title = {Stability estimate for a hyperbolic inverse problem with time-dependent coefficient},
journal = {Inverse Problems},
}

@article{ArancibiaLecarosMercadoZamorano+2022+659+675,
title = {An inverse problem for the {Moore--Gibson--Thompson} equation arising in high-intensity ultrasound},
author = {Arancibia, Rogelio and Lecaros, Rodrigo and Mercado, Alberto and Zamorano, Sebasti{\'a}n},
pages = {659--675},
volume = {30},
number = {5},
journal = {Journal of Inverse and Ill-posed Problems},
doi = {10.1515/jiip-2020-0090},
year = {2022},
}

@article{babich1981complex,
  title={A complex space-time ray method and ``quasi-photons''},
  author={Babich, Vasilii Mikhailovich and Ulin, V. V.},
  journal={Mathematical Aspects of Wave Propagation Theory. 11},
  pages={5--12},
  year={1981}
}

@article{bao2014sensitivity,
  title={Sensitivity analysis of an inverse problem for the wave equation with caustics},
  author={Bao, Gang and Zhang, Hai},
  journal={Journal of the American Mathematical Society},
  volume={27},
  number={4},
  pages={953--981},
  year={2014}
}

@incollection{belishev2011boundary,
  title={Boundary control method in dynamical inverse problems---an introductory course},
  author={Belishev, Mikhail I},
  booktitle={Dynamical inverse problems: theory and application},
  pages={85--150},
  year={2011},
  publisher={Springer}
}

@inproceedings{belishev1987approach,
  title={An approach to multidimensional inverse problems for the wave equation},
  author={Belishev, Mikhail Igorevich},
  booktitle={Doklady Akademii Nauk},
  volume={297},
  pages={524--527},
  year={1987},
  organization={Russian Academy of Sciences}
}

@article{belishev1992boundary,
  title={Boundary controls and quasiphotons in a {Riemannian} manifold reconstruction problem via dynamical data},
  author={Belishev, Mikhail Igorevich and Kachalov, Alexander Pavlovich},
  journal={Zapiski Nauchnykh Seminarov POMI},
  volume={203},
  pages={21--50},
  year={1992},
  publisher={St. Petersburg Department of the Steklov Institute of Mathematics, Russian Academy of Sciences}
}

@article{bohr2021stability,
  title={Stability of the non-abelian {X}-ray transform in dimension $\geq 3$},
  author={Bohr, Jan},
  journal={The Journal of Geometric Analysis},
  volume={31},
  number={11},
  pages={11226--11269},
  year={2021},
  publisher={Springer}
}

@article{chen2025stable,
  title={Stable inversion of potential in nonlinear wave equations with cubic nonlinearity},
  author={Chen, Xi and Lu, Shuai and Zhang, Ruochong},
  journal={Mathematische Annalen},
  volume={392},
  number={3},
  pages={4283--4314},
  year={2025},
  publisher={Springer}
}

@article{feizmohammadi2021recovery,
  title={Recovery of time-dependent coefficients from boundary data for hyperbolic equations},
  author={Feizmohammadi, Ali and Ilmavirta, Joonas and Kian, Yavar and Oksanen, Lauri},
  journal={Journal of Spectral Theory},
  volume={11},
  number={3},
  pages={1107--1143},
  year={2021}
}

@article{feizmohammadi2021light,
  title={The light ray transform in stationary and static {Lorentzian} geometries},
  author={Feizmohammadi, Ali and Ilmavirta, Joonas and Oksanen, Lauri},
  journal={The Journal of Geometric Analysis},
  volume={31},
  number={4},
  pages={3656--3682},
  year={2021},
  publisher={Springer}
}

@article{feizmohammadi2022recovery,
  title={Recovery of zeroth-order coefficients in nonlinear wave equations},
  author={Feizmohammadi, Ali and Oksanen, Lauri},
  journal={Journal of the Institute of Mathematics of Jussieu},
  volume={21},
  number={2},
  pages={367--393},
  year={2022},
  publisher={Cambridge University Press}
}

@article{fu2024inverse,
  title={Inverse problem of recovering a time-dependent nonlinearity appearing in third-order nonlinear acoustic equations},
  author={Fu, Song-Ren and Yao, Peng-Fei and Yu, Yongyi},
  journal={Inverse Problems},
  volume={40},
  number={7},
  pages={075001},
  year={2024},
  publisher={IOP Publishing}
}

@article{Fu03052026,
author = {Song-Ren Fu and Yongyi Yu},
title = {On the stability of determining a nonlinear coefficient for the third-order nonlinear acoustic equation},
journal = {Applicable Analysis},
volume = {105},
number = {7},
pages = {1456--1472},
year = {2026},
publisher = {Taylor \& Francis},
doi = {10.1080/00036811.2025.2564730},
}

@article{fu2026calderon,
  title={The {Calder{\'o}n} problem for third-order nonlocal wave equations with time-dependent nonlinearities and potentials},
  author={Fu, Song-Ren and Yu, Yongyi and Zimmermann, Philipp},
  journal={Journal of Differential Equations},
  volume={463},
  pages={114164},
  year={2026},
  publisher={Elsevier}
}

@article{greene1976c,
  title={{$C^\infty$} convex functions and manifolds of positive curvature},
  author={Greene, Robert E and Wu, Hung-hsi},
  journal={Acta Math.},
  volume={137},
  pages={209--245},
  year={1976}
}

@book{jost2005riemannian,
  title={{Riemannian Geometry and Geometric Analysis}},
  author={Jost, J{\"u}rgen},
  year={2008},
  publisher={Springer, Berlin}
}

@book{kachalov2001inverse,
  title={{Inverse Boundary Spectral Problems}},
  author={Kachalov, Alexander and Kurylev, Yaroslav and Lassas, Matti},
  year={2001},
  publisher={Chapman and Hall/CRC}
}

@misc{kaltenbacher2025imagingnonlinearitycoefficientsound,
      title={Imaging nonlinearity coefficient and sound speed with the {JMGT} equation in the frequency domain},
      author={Barbara Kaltenbacher},
      year={2025},
      eprint={2512.18431},
      archivePrefix={arXiv},
      primaryClass={math.AP},
      note={arXiv:2512.18431},
      url={https://arxiv.org/abs/2512.18431},
}

@article{kaltenbacher2025acoustic,
  title={Acoustic nonlinearity parameter tomography with the {Jordan--Moore--Gibson--Thompson} equation in the frequency domain},
  author={Kaltenbacher, Barbara},
  journal={Inverse Problems},
  volume={41},
  number={9},
  pages={095010},
  year={2025},
  publisher={IOP Publishing}
}

@article{BKILR,
  author = {Kaltenbacher, Barbara and Lasiecka, Irena and Marchand, Richard},
  year = {2011},
  title = {Well-posedness and exponential decay rates for the {Moore--Gibson--Thompson} equation arising in high-intensity ultrasound},
  journal = {Control Cybernet.},
  volume = {40},
  number = {4},
  pages = {971--988},
}

@article{BKVN,
  author = {Kaltenbacher, Barbara and Nikoli{\'c}, Vanja},
  year = {2021},
  title = {The inviscid limit of third-order linear and nonlinear acoustic equations},
  journal = {SIAM J. Appl. Math.},
  volume = {81},
  number = {4},
  pages = {1461--1482},
}

@article{kian2017unique,
  title={Unique determination of a time-dependent potential for wave equations from partial data},
  author={Kian, Yavar},
  journal={Annales de l'Institut Henri Poincar{\'e} C, Analyse non lin{\'e}aire},
  volume={34},
  number={4},
  pages={973--990},
  year={2017},
}

@article{kian2016stability,
  title={Stability in the determination of a time-dependent coefficient for wave equations from partial data},
  author={Kian, Yavar},
  journal={Journal of Mathematical Analysis and Applications},
  volume={436},
  number={1},
  pages={408--428},
  year={2016},
  publisher={Elsevier}
}

@article{kian2016recovery,
  title={Recovery of time-dependent damping coefficients and potentials appearing in wave equations from partial data},
  author={Kian, Yavar},
  journal={SIAM Journal on Mathematical Analysis},
  volume={48},
  number={6},
  pages={4021--4046},
  year={2016},
  publisher={SIAM}
}

@article{kumar2026holder,
  title={H{\"o}lder stability estimates for the determination of time-independent potentials in a relativistic wave equation in an infinite waveguide},
  author={Kumar, Mandeep and Zimmermann, Philipp},
  journal={Inverse Problems},
  volume={42},
  number={3},
  pages={035004},
  year={2026},
  publisher={IOP Publishing}
}

@article{kurylev2018inverse,
  title={Inverse problems for {Lorentzian} manifolds and nonlinear hyperbolic equations},
  author={Kurylev, Yaroslav and Lassas, Matti and Uhlmann, Gunther},
  journal={Inventiones Mathematicae},
  volume={212},
  number={3},
  pages={781--857},
  year={2018},
  publisher={Springer}
}

@article{lai2024partial,
  title={Partial data inverse problems for the nonlinear time-dependent {Schr\"odinger} equation},
  author={Lai, Ru-Yu and Lu, Xuezhu and Zhou, Ting},
  journal={SIAM Journal on Mathematical Analysis},
  volume={56},
  number={4},
  pages={4712--4741},
  year={2024},
  publisher={SIAM}
}

@article{LASIECKA20157610,
title = {{Moore--Gibson--Thompson} equation with memory, part {II}: general decay of energy},
journal = {Journal of Differential Equations},
volume = {259},
number = {12},
pages = {7610--7635},
year = {2015},
author = {Lasiecka, Irena and Wang, Xiaojun},
}

@article{lassas2025stability,
  title={Stability and {Lorentzian} geometry for an inverse problem of a semilinear wave equation},
  author={Lassas, Matti and Liimatainen, Tony and Potenciano-Machado, Leyter and Tyni, Teemu},
  journal={Analysis \& PDE},
  volume={18},
  number={5},
  pages={1065--1118},
  year={2025},
  publisher={Mathematical Sciences Publishers}
}

@article{lassas2022uniqueness,
  title={Uniqueness, reconstruction and stability for an inverse problem of a semilinear wave equation},
  author={Lassas, Matti and Liimatainen, Tony and Potenciano-Machado, Leyter and Tyni, Teemu},
  journal={Journal of Differential Equations},
  volume={337},
  pages={395--435},
  year={2022},
  publisher={Elsevier}
}

@article{lassas2020light,
  title={The light ray transform on {Lorentzian} manifolds},
  author={Lassas, Matti and Oksanen, Lauri and Stefanov, Plamen and Uhlmann, Gunther},
  journal={Communications in Mathematical Physics},
  volume={377},
  number={2},
  pages={1349--1379},
  year={2020},
  publisher={Springer}
}

@misc{liu2025partial,
  title={On a partial-data inverse problem for the semilinear wave equation},
  author={Liu, Boya and Wang, Weinan},
  year={2025},
  note={arXiv:2511.08794}
}

@article{LiuTriggiani+2013+825+869,
title = {An inverse problem for a third-order {PDE} arising in high-intensity ultrasound: global uniqueness and stability by one boundary measurement},
author = {Liu, Shitao and Triggiani, Roberto},
pages = {825--869},
volume = {21},
number = {6},
journal = {Journal of Inverse and Ill-Posed Problems},
doi = {10.1515/jip-2012-0096},
year = {2013},
}

@article{LIZAMA20197813,
title = {Controllability results for the {Moore--Gibson--Thompson} equation arising in nonlinear acoustics},
journal = {Journal of Differential Equations},
volume = {266},
number = {12},
pages = {7813--7843},
year = {2019},
author = {Lizama, Carlos and Zamorano, Sebasti{\'a}n},
}

@article{lizama2023boundary,
  title={Boundary controllability for the {1D} {Moore--Gibson--Thompson} equation},
  author={Lizama, Carlos and Zamorano, Sebastian},
  journal={Meccanica},
  volume={58},
  number={6},
  pages={1031--1038},
  year={2023},
  publisher={Springer}
}

@article{oksanen2025interplay,
  title={On the interplay between the light ray and the magnetic {X}-ray transforms},
  author={Oksanen, Lauri and Paternain, Gabriel P and Sarkkinen, Miika},
  journal={SIAM Journal on Mathematical Analysis},
  volume={57},
  number={6},
  pages={6522--6541},
  year={2025},
  publisher={SIAM}
}

@article{paternain2019geodesic,
  title={The geodesic {X}-ray transform with matrix weights},
  author={Paternain, Gabriel P and Salo, Mikko and Uhlmann, G{\"u}nther and Zhou, Hanming},
  journal={American Journal of Mathematics},
  volume={141},
  number={6},
  pages={1707--1750},
  year={2019},
  publisher={Johns Hopkins University Press}
}

@book{paternain2023geometric,
  title={{Geometric Inverse Problems}},
  author={Paternain, Gabriel P and Salo, Mikko and Uhlmann, Gunther},
  volume={204},
  year={2023},
  publisher={Cambridge University Press}
}

@misc{qiu2026gauge,
  title={Gauge symmetry and uniqueness in inverse problems for the {JMGT} equation},
  author={Qiu, Dong and Xu, Xiang and Ye, Yeqiong and Zhou, Ting},
  year={2026},
  note={arXiv:2604.28023}
}

@misc{qiu2026inverse,
  title={Inverse boundary value problems of determining nonlinear coefficients for the {JMGT} equation},
  author={Qiu, Dong and Xu, Xiang and Ye, Yeqiong and Zhou, Ting},
  year={2026},
  note={arXiv:2603.14194}
}

@article{stefanov2018inverse,
  title={The inverse problem for the {Dirichlet-to-Neumann} map on {Lorentzian} manifolds},
  author={Stefanov, Plamen and Yang, Yang},
  journal={Analysis \& PDE},
  volume={11},
  number={6},
  pages={1381--1414},
  year={2018},
  publisher={Mathematical Sciences Publishers}
}

@article{SU04,
  title={Stability estimates for the {X}-ray transform of tensor fields and boundary rigidity},
  author={Stefanov, Plamen and Uhlmann, Gunther},
  journal={Duke Math. J.},
  volume={123},
  number={3},
  pages={445--467},
  year={2004},
  doi={10.1215/S0012-7094-04-12332-2}
}

@article{uhlmann2016inverse,
  title={The inverse problem for the local geodesic ray transform},
  author={Uhlmann, Gunther and Vasy, Andr{\'a}s},
  journal={Inventiones Mathematicae},
  volume={205},
  number={1},
  pages={83--120},
  year={2016},
  publisher={Springer}
}

@article{uhlmann2021inverse,
  title={On an inverse boundary value problem for a nonlinear elastic wave equation},
  author={Uhlmann, Gunther and Zhai, Jian},
  journal={Journal de Math{\'e}matiques Pures et Appliqu{\'e}es},
  volume={153},
  pages={114--136},
  year={2021},
  publisher={Elsevier}
}

@article{uhlmann2023inverse,
  title={An inverse boundary value problem arising in nonlinear acoustics},
  author={Uhlmann, Gunther and Zhang, Yang},
  journal={SIAM Journal on Mathematical Analysis},
  volume={55},
  number={2},
  pages={1364--1404},
  year={2023},
  publisher={SIAM}
}

@article{vasy2021light,
  title={On the light ray transform of wave equation solutions},
  author={Vasy, Andr{\'a}s and Wang, Yiran},
  journal={Communications in Mathematical Physics},
  volume={384},
  number={1},
  pages={503--532},
  year={2021},
  publisher={Springer}
}

@misc{wendels2025stable,
  title={Stable determination of the nonlinear parameter in the nondiffusive {Westervelt} equation from the {Dirichlet-to-Neumann} map},
  author={Wendels, Mike},
  year={2025},
  note={arXiv:2510.02553}
}

@book{yao2011modeling,
  title={{Modeling and Control in Vibrational and Structural Dynamics: A Differential Geometric Approach}},
  author={Yao, Peng-Fei},
  year={2011},
  publisher={CRC Press}
}
	
\bibliographystyle{alpha}



\end{document}